\documentclass[11 pt]{amsart}
\usepackage{latexsym,amscd,amssymb, graphicx, amsthm}  
\usepackage{young}
\usepackage{tikz}

\usepackage[margin=1in]{geometry}
\usepackage[colorinlistoftodos]{todonotes}

\numberwithin{equation}{section}

\newtheorem{theorem}{Theorem}[section]
\newtheorem{proposition}[theorem]{Proposition}
\newtheorem{corollary}[theorem]{Corollary}
\newtheorem{lemma}[theorem]{Lemma}
\newtheorem{conjecture}[theorem]{Conjecture}
\newtheorem{observation}[theorem]{Observation}

\newtheorem{problem}[theorem]{Problem}
\newtheorem{example}[theorem]{Example}

\newtheorem{defn}[theorem]{Definition}
\theoremstyle{definition}

\newcommand{\maj}{{\mathrm {maj}}}
\newcommand{\inv}{{\mathrm {inv}}}

\newcommand{\Val}{{\mathrm {Val}}}
\newcommand{\Rise}{{\mathrm {Rise}}}
\newcommand{\Stir}{{\mathrm {Stir}}}
\newcommand{\Hilb}{{\mathrm {Hilb}}}
\newcommand{\grFrob}{{\mathrm {grFrob}}}
\newcommand{\conv}{{\mathrm {conv}}}
\newcommand{\std}{{\mathrm {std}}}

\newcommand{\rev}{{\mathrm {rev}}}
\newcommand{\stat}{{\mathrm {stat}}}
\newcommand{\Mat}{{\mathrm {Mat}}}

\newcommand{\SYT}{{\mathrm {SYT}}}

\newcommand{\Frob}{{\mathrm {Frob}}}

\newcommand{\initial}{{\mathrm {in}}}

\newcommand{\shape}{{\mathrm {shape}}}

\newcommand{\symm}{{\mathfrak{S}}}

\newcommand{\PM}{{\mathrm{PM}}}
\newcommand{\OPM}{{\mathrm{OPM}}}

\newcommand{\KK}{{\mathbb {K}}}
\newcommand{\FF}{{\mathbb {F}}}
\newcommand{\CC}{{\mathbb {C}}}
\newcommand{\QQ}{{\mathbb {Q}}}
\newcommand{\ZZ}{{\mathbb {Z}}}
\newcommand{\NN}{{\mathbb{N}}}
\newcommand{\PP}{{\mathbb{P}}}
\newcommand{\OP}{{\mathcal{OP}}}

\newcommand{\WWW}{{\mathcal{W}}}

\newcommand{\UUU}{{\mathcal{U}}}
\newcommand{\NNN}{{\mathcal{N}}}

\newcommand{\MMM}{{\mathcal{M}}}

\newcommand{\WW}{{\mathcal{W}}}
\newcommand{\Fl}{{\mathcal{F\ell}}}
\newcommand{\NS}{{\mathcal{NS}}}
\newcommand{\TTT}{{\mathcal{T}}}

\newcommand{\xx}{{\mathbf {x}}}

\DeclareMathOperator{\codim}{codim}

\begin{document}

\title[A flag variety for the Delta Conjecture]
{A flag variety for the Delta Conjecture}

\author{Brendan Pawlowski}
\address
{Department of Mathematics \newline \indent
University of Michigan \newline \indent
Ann Arbor, MI, 48103-1043, USA}
\email{pawlows@umich.edu}

\author{Brendon Rhoades}
\address
{Department of Mathematics \newline \indent
University of California, San Diego \newline \indent
La Jolla, CA, 92093-0112, USA}
\email{bprhoades@math.ucsd.edu}

\begin{abstract}
The {\em Delta Conjecture} of Haglund, Remmel, and Wilson predicts 
the monomial expansion of the symmetric function $\Delta'_{e_{k-1}} e_n$, where
$k \leq n$ are positive integers and $\Delta'_{e_{k-1}}$ is a Macdonald eigenoperator.
When $k = n$, the specialization $\Delta'_{e_{n-1}} e_n|_{t = 0}$ 
is the Frobenius image of the graded $S_n$-module afforded by the cohomology
ring of the {\em flag variety} consisting of complete flags in $\CC^n$.
We define and study a variety $X_{n,k}$ which carries an
action of $S_n$ whose cohomology ring $H^{\bullet}(X_{n,k})$ has 
Frobenius image given by $\Delta'_{e_{k-1}} e_n|_{t = 0}$,
up to a minor twist.
 The variety $X_{n,k}$ has a cellular decomposition with cells $C_w$ indexed
by length $n$ words $w = w_1 \dots w_n$ in the alphabet $\{1, 2, \dots, k\}$
in which each letter appears at least once.
When $k = n$, the variety $X_{n,k}$ is homotopy equivalent to the flag variety.
We give a presentation for the cohomology ring $H^{\bullet}(X_{n,k})$ as a quotient of 
the polynomial ring
$\ZZ[x_1, \dots, x_n]$ and describe polynomial representatives for the classes
$[ \overline{C}_w]$ of the closures of the cells $C_w$; these representatives 
generalize the classical Schubert polynomials.
\end{abstract}

\maketitle

\section{Introduction}
\label{Introduction}

The purpose of this paper is to introduce a new algebraic variety $X_{n,k}$ which generalizes
the flag variety.  Our new variety will have geometric and combinatorial properties
which generalize those of the flag variety, as well as representation theoretic properties
related to the theory of Macdonald polynomials.

Given any symmetric function $f$, the (primed) {\em delta operator} 
$\Delta'_f$ indexed by $f$
is a Macdonald eigenoperator on the space of symmetric functions.
Symmetric function theory has produced many (sometimes conjectural)
combinatorial   expansions of 
expressions of the form $\Delta'_f e_n$ \cite{CarlssonMellit, HRW}. 
When $f = e_{n-1}$, the specialization
$\Delta'_{e_{n-1}} e_n|_{t = 0}$ has deep connections to the 
algebra of the coinvariant ring and the geometry of the flag manifold.
For other symmetric functions $f$, the algebraic and geometric meanings of
$\Delta'_f e_n$ or $\Delta'_f e_n|_{t = 0}$ are less clear.
Our new variety $X_{n,k}$  serves as a `flag variety' for 
the symmetric function
$\Delta'_{e_{k-1}} e_n$ appearing in the {\em Delta Conjecture} \cite{HRW}.

Let the symmetric group $S_n$ act on the polynomial ring $\ZZ[\xx_n] := \ZZ[x_1, \dots, x_n]$ by subscript permutation.
The polynomials $f \in \ZZ[\xx_n]^{S_n}$ which are invariant under this action are called {\em symmetric}.  
The ring $\ZZ[\xx_n]^{S_n}$ is a graded $\ZZ$-algebra which is freely generated by the 
 {\em elementary symmetric polynomials} $e_1(\xx_n), \dots, e_n(\xx_n)$, where
 \begin{equation*}
 e_d(\xx_n) = \sum_{1 \leq i_1 < \cdots < i_d \leq n} x_{i_1} \cdots x_{i_d}.
 \end{equation*}
 
 The {\em invariant ideal} $I_n \subseteq \ZZ[\xx_n]$ is the ideal generated by the set $\ZZ[\xx_n]^{S_n}_+$ of 
 symmetric polynomials with vanishing constant term.  Equivalently, we have
 $I_n = \langle e_1(\xx_n), \dots, e_n(\xx_n) \rangle$.
The corresponding quotient $R_n := \ZZ[\xx_n]/I_n$ is called the {\em coinvariant ring}; it is a free graded
$\ZZ$-module of rank $n!$.

Borel gave a geometric interpretation of $R_n$ as the cohomology of the flag variety \cite{Borel}.
A {\em (complete) flag} in the vector space $\CC^n$ is a chain of nested subspaces 
$V_{\bullet} = (V_1 \subset \cdots \subset V_n)$ in $\CC^n$ such that $\dim(V_i) = i$.
The {\em flag variety} $\Fl(n)$ is the collection of all flags in $\CC^n$.  The set $\Fl(n)$ attains the structure of a 
smooth manifold and a complex algebraic variety by its identification
$\Fl(n) = GL_n(\CC)/B$, where $B \subseteq GL_n(\CC)$ is the subgroup of upper triangular matrices.

The cohomology of $\Fl(n)$ may be presented \cite{Borel} as 
$H^{\bullet}(\Fl(n)) = R_n$, 
where $H^{\bullet}$ stands for singular cohomology with $\ZZ$-coefficients, where we regard
each variable $x_i$ as having degree $2$.
Explicitly, the variable $x_i$ represents $-c_1(V_i/V_{i-1}) \in H^2(\Fl(n))$, 
where $c_1(V_i/V_{i-1})$ is the first Chern
class of the quotient bundle $V_i/V_{i-1}$ over $\Fl(n)$.

The algebraic and geometric properties of $R_n$ and $\Fl(n)$ 
are governed by combinatorial properties of permutations in $S_n$.
\begin{itemize}
\item  The Hilbert series $\Hilb(R_n; q)$ of the free graded $\ZZ$-module $R_n$ is 
given by the Mahonian distribution 
$[n]!_q = (1+q) (1+q+q^2) \cdots (1+q+ \cdots + q^{n-1})$.  There are monomial bases of $R_n$ 
(Artin \cite{Artin} and Garsia-Stanton \cite{Garsia, GarsiaStanton}; see e.g. \cite{HRS} for their definitions)
which reflect the two most important Mahonian statistics (inversion count $\inv$ and 
major index $\maj$) on $S_n$.
\item  Extending coefficients to $\QQ$, Chevalley proved that the {\em ungraded} isomorphism type of 
$R_{n}^{\QQ} := \QQ \otimes_{\ZZ} R_n$ as an $S_n$-module is given by the regular representation of $S_n$, that is
$R_n^{\QQ} \cong \QQ[S_n]$ as $S_n$-modules \cite{Chevalley}.
\item  
Lusztig (unpublished) and Stanley \cite{Stanley} proved that
the {\em graded} isomorphism type of $R_n^{\QQ}$ (or equivalently the symmetric function
$\grFrob(R_n^{\QQ}; q) = \sum_{d \geq 0} \Frob((R_n^{\QQ})_d; q)$ given by its graded Frobenius image) is given by
\begin{equation*}
\grFrob(R_n^{\QQ}; q) =
\sum_{T \in \SYT(n)} q^{\maj(T)} s_{\shape(T)}.
\end{equation*}
Here  $T$ ranges over all standard
tableaux with $n$ boxes, $\maj(T)$ is the major index of $T$, 
and $s_{\shape(T)}$ is the Schur function indexed by the shape of $T$.
\item  We have a CW decomposition of the flag variety $\Fl(n)$ whose open cells $\{ X_w \,:\, w \in S_n \}$ 
are the {\em Schubert cells} $X_w := B w B/B$ (where we identify $w$ with its $n \times n$ permutation matrix).
In particular, the Poincar\'e series $P_{\Fl(n)}(q)$ of $\Fl(n)$ is given by $[n]!_{q^2}$.  Under Borel's identification
$H^{\bullet}(\Fl(n)) = R_n$, the class $[\overline{X}_w] \in H^{\bullet}(\Fl(n))$ of the closure of $X_w$ is represented
by the {\em Schubert polynomial} $\symm_{w^{\vee}} \in \ZZ[\xx_n]$.  Here $w^{\vee} \in S_n$ has one-line notation
$w^{\vee}  = (n-w_1+1) \dots (n-w_n+1)$, where $w = w_1 \dots w_n$ is the one-line notation for $w$.
In particular, the set $\{ \symm_w \,:\, w \in S_n \}$ of Schubert polynomials descends to a $\ZZ$-basis for $R_n$.
\end{itemize}

For positive integers $k \leq n$,
let $I_{n,k} \subseteq \ZZ[\xx_n]$ be the ideal
\begin{equation}
I_{n,k} = \langle x_1^k, x_2^k, \dots, x_n^k, e_n(\xx_n), e_{n-1}(\xx_n), \dots, e_{n-k+1}(\xx_n) \rangle
\end{equation}
and let $R_{n,k} = \ZZ[\xx_n]/I_{n,k}$ be the corresponding quotient.
Since $I_{n,k}$ is homogeneous, the quotient $R_{n,k}$ is a graded $\ZZ$-algebra.
Since $I_{n,k}$ is stable under the $S_n$-action, the quotient $R_{n,k}$ is a graded $\ZZ[S_n]$-module.
When $k = 1$, we have $R_{n,1} = \ZZ$ (the trivial representation in degree $0$); when 
$k = n$ we recover the classical coinvariant ring $R_{n,n} = R_n$. 

The $\QQ$-version $R_{n,k}^{\QQ} = \QQ \otimes_{\ZZ} R_{n,k}$ of the generalized coinvariant ring $R_{n,k}$ was
defined by Haglund-Rhoades-Shimozono \cite{HRS}.
The connection between $R_n$ and permutations  is shown in \cite{HRS}
to generalize to a corresponding connection between $R_{n,k}^{\QQ}$ and  ordered set partitions.

An {\em ordered set partition of size $n$ with $k$ blocks} is a $k$-block set partition of 
$[n] := \{1, 2, \dots, n\}$ equipped with a total order on its blocks.  
Write $\OP_{n,k}$ for the family of ordered set partitions of size $n$ with $k$ blocks.  For example, we have 
\begin{equation*}
\{2,5\} \prec \{3\} \prec \{1,4\} \in \OP_{5,3}.
\end{equation*}
We can think of ordered set partitions algebraically as parabolic cosets in the symmetric group $S_n$
or as faces in the Coxeter complex attached to $S_n$.
The set $\OP_{n,k}$ is counted by
\begin{equation}
|\OP_{n,k}| = k! \cdot \Stir(n,k),
\end{equation}
where $\Stir(n,k)$ is the (signless) Stirling number of the second kind counting $k$-block set partitions
of $[n]$.  When $k = n$ we may identify ordered set partitions with permutations: $\OP_{n,n} = S_n$.

In this paper we  represent ordered set partitions with words.
A length $n$ word $w = w_1 \dots w_n$ in the alphabet of positive integers is called {\em Fubini} if for any 
letter $i > 1$ appearing in $w$, the letter $i-1$ also appears in $w$.
Let $\WW_{n,k}$ be the set of Fubini words of length $n$ with maximum letter $k$.  Equivalently, 
\begin{equation*}
\WW_{n,k} = \{ w = w_1 \dots w_n \in [k]^n \,:\, \text{the letters $1, 2, \dots, k$ all appear in $w$} \}.
\end{equation*}
There is a natural bijection between $\OP_{n,k}$ and $\WW_{n,k}$: if 
$\sigma \in \OP_{n,k}$ the corresponding word $w_1 \dots w_n \in \WW_{n,k}$ has 
$w_i = j$ if and only if $i$ is in the $j^{th}$ block of $\sigma$.  If $\sigma \in \OP_{5,3}$ is as above,
\begin{equation*}
\{2,5\} \prec \{3\} \prec \{1,4\} \leftrightarrow 31231 \in \WW_{5,3}.
\end{equation*}

Haglund-Rhoades-Shimozono established the following algebraic relationships between $R^{\QQ}_{n,k}$ and
$\WW_{n,k}$ \cite{HRS}.  Let $\rev_q$ be the operator on polynomials in $q$ which reverses 
coefficient sequences, so that e.g. $\rev_q(q^3 + 2q - 1) = -q^3 + 2q^2 + 1$.
Also let ${a \brack b}_q$ be the usual $q$-binomial coefficient 
${a \brack b}_q := \frac{[a]!_q}{[b]!_q \cdot [a-b]!_q}$.
If $b > a$ or $a < 0$ we adopt the convention
${a \brack b}_q = 0$.
Postponing some definitions to Section~\ref{Background}, we 
have the following results from \cite{HRS}.

\begin{itemize}
\item  The Hilbert series $\Hilb(R^{\QQ}_{n,k}; q)$ of the graded $\QQ$-vector space 
$R^{\QQ}_{n,k}$ is $\rev_q([k]!_q \cdot \Stir_q(n,k))$,
where $\Stir_q(n,k)$ is the {\em $q$-Stirling number}.
There are monomial bases of $R_{n,k}^{\QQ}$
generalizing  the Artin and Garsia-Stanton bases of $R_n^{\QQ}$.
\item  As an {\em ungraded} $S_n$-module, we have $R_n^{\QQ} \cong \QQ[\WW_{n,k}]$, where $S_n$ acts on the set
of Fubini words $\WW_{n,k}$ by $\pi.(w_1 \dots w_n) = w_{\pi_1} \dots w_{\pi_n}$ for all $\pi \in S_n$ and
$w_1 \dots w_n \in \WW_{n,k}$.
\item  As a {\em graded} $S_n$-module, we have 
\begin{equation*}
\grFrob(R_{n,k}^{\QQ}; q) = \sum_{T \in \SYT(n)} q^{\maj(T)} {n - \mathrm{des}(T) - 1 \brack n-k}_q s_{\shape(T)}.
\end{equation*}
Here $\mathrm{des}(T)$ is the number of descents in the tableau $T$.
\end{itemize}

The construction of $R_{n,k}^{\QQ}$ in \cite{HRS} was motivated by the theory of Macdonald polynomials.
Let $\{ \widetilde{H}_{\mu} \,:\, \mu \vdash n \}$ be the modified Macdonald basis of the vector space 
$\Lambda_n$ of homogeneous degree $n$ symmetric functions over the field $\QQ(q,t)$.
If $f$ is an arbitrary symmetric function, the corresponding {\em delta operator}
is the Macdonald eigenoperator $\Delta'_f: \Lambda_n \rightarrow \Lambda_n$ defined by
\begin{equation}
\Delta'_f: \widetilde{H}_{\mu} \mapsto f( \dots, q^i t^j, \dots ) \cdot \widetilde{H}_{\mu},
\end{equation}
where $(i,j)$ range over all coordinates $\neq (0,0)$ of the boxes of the Ferrers diagram of $\mu$
(and we set all other variables of $f$ to zero).
As an example, if $\mu = (3,2) \vdash 5$ then we fill the Ferrers diagram of $\mu$ as
\begin{center}
\begin{Young}
$\cdot$ & $q$ & $q^2$ \cr
$t$ & $qt$
\end{Young},
\end{center}
so that
\begin{equation*}
\Delta'_f: \widetilde{H}_{\mu} \mapsto f(q,q^2,t,qt) \cdot \widetilde{H}_{\mu}.
\end{equation*}

For any positive integers $k \leq n$, the {\em Delta Conjecture} of Haglund-Remmel-Wilson \cite{HRW} predicts 
the monomial expansion of the delta operator application $\Delta'_{e_{k-1}} e_n$.  It reads
\begin{equation}
\Delta'_{e_{k-1}} e_n = \Rise_{n,k}(\xx;q,t) = \Val_{n,k}(\xx;q,t).
\end{equation}
Here $\Rise_{n,k}(\xx;q,t)$ and $\Val_{n,k}(\xx;q,t)$ are certain formal power series 
in the infinite variable set $\xx = (x_1, x_2, \dots )$ depending on the two parameters $q,t$.  
The combinatorial definition of $\Rise$ and $\Val$ depends on various statistics on length $n$ labeled Dyck paths;
see \cite{HRW} for details.

In the case $k = n$ we have $\Delta'_{e_{n-1}} e_n = \nabla e_n$, where $\nabla$ is the Bergeron-Garsia nabla operator,
and the Delta Conjecture reduces to the Shuffle Theorem \cite{CarlssonMellit}.
In particular, the classical coinvariant ring $R_n^{\QQ}$ has graded Frobenius image
$\nabla e_n|_{t = 0}$ obtained by specializing $\nabla e_n$ at $t = 0$.

While the Delta Conjecture is open in general, it has been proven in the case where one of the variables $q,t$ is set to zero.

\begin{theorem}
\label{zero-delta-conjecture}
(Garsia-Haglund-Remmel-Yoo \cite{GHRY}, R. \cite{RhoadesOSP}, Wilson \cite{Wilson})  Let $k \leq n$ be positive integers.  We have
\begin{equation*}
\Delta'_{e_{k-1}} e_n|_{t = 0} = 
\Delta'_{e_{k-1}} e_n|_{q = 0, t = q} =
\Rise_{n,k}(\xx;q,0) =
\Rise_{n,k}(\xx;0,q) = 
\Val_{n,k}(\xx;q,0) = 
\Val_{n,k}(\xx;0,q).
\end{equation*}
\end{theorem} 

It is an open problem to find a `nice' bigraded $S_n$-module whose Frobenius image is 
given by (even after a `nice' modification) any of the expressions $\Delta'_{e_{k-1}} e_n, \Rise_{n,k}(\xx;q,t),$ or 
$\Val_{n,k}(\xx;q,t)$ in the Delta Conjecture.
The ring $R_{n,k}^{\QQ}$ solves this problem at $t = 0$. 

\begin{theorem}
\label{delta-coinvariant-algebra}  
(Haglund-R.-Shimozono \cite{HRS})
Let $k \leq n$ be positive integers.  If $C_{n,k}(\xx;q)$ is the common symmetric function of Theorem~\ref{zero-delta-conjecture}
then 
\begin{equation*}
\grFrob(R_{n,k}^{\QQ}; q) = (\rev_q \circ \omega) C_{n,k}(\xx;q),
\end{equation*}
where $\omega$ is the symmetric function involution which interchanges the elementary 
symmetric function $e_r$ and the homogeneous symmetric function $h_r$.
\end{theorem}

Theorem~\ref{delta-coinvariant-algebra} justifies the claim that $R_{n,k}^{\QQ}$ is the coinvariant algebra for the Delta 
Conjecture.
In \cite[Prob. 7.2]{HRS} the problem was posed of finding a version of the {\em flag variety} of the Delta Conjecture.
That is, one would like a complex algebraic variety whose cohomology is isomorphic to 
$R_{n,k}$.
This variety should have a cellular decomposition
with cells indexed by $\WW_{n,k}$ giving rise to a Schubert polynomial
basis $\{ \symm_w \,:\, w \in \WW_{n,k} \}$ for $R_{n,k}$.

In this paper we solve \cite[Prob. 7.2]{HRS} by means of the following construction;
it is the main object of study.

\begin{defn}
\label{main-definition-x}
Let $k \leq n$ be positive integers.  We denote by $X_{n,k}$ be the following space of line 
configurations in $\CC^k$:
\begin{equation}
X_{n,k} := \{ (\ell_1, \dots , \ell_n) \,:\, \text{$\ell_i$ a line through the origin in $\CC^k$ and $\ell_1 + \cdots + \ell_n = \CC^k$} \}.
\end{equation}
\end{defn}

A point in $X_{n,k}$ is an $n$-tuple $\ell_{\bullet} = (\ell_1, \dots, \ell_n)$ 
of $1$-dimensional subspaces of $\CC^k$ which together
span $\CC^k$.

If $\PP^{k-1}$ is $(k-1)$-dimensional complex projective space and 
$(\PP^{k-1})^n$ is its $n$-fold Cartesian product, the space $X_{n,k}$ is 
an open subvariety of $(\PP^{k-1})^n$, and therefore a smooth complex
manifold.
When $k = 1$, the space $X_{n,1}$ is the one-point space.
When $k = n$, if $T \subseteq GL_n(\CC)$ is the diagonal torus, the canonical projection
\begin{equation*}
X_{n,n} = GL_n(\CC)/T \twoheadrightarrow GL_n(\CC)/B = \Fl(n)
\end{equation*}
is a homotopy equivalence.   
\footnote{Indeed, the map $X_{n,n} \twoheadrightarrow \Fl(n)$ is a fiber bundle
over a Hausdorff base space
whose fiber is isomorphic to the (contractible) unipotent group of upper triangular
$n \times n$ matrices with $1$'s on the diagonal.}
We prove the following results about $X_{n,k}$.

\begin{itemize}
\item The variety $X_{n,k}$ has a cellular decomposition with cells $\{ C_w \,:\, w \in \WW_{n,k} \}$ indexed by
Fubini words (Theorem~\ref{cohomology-of-x} (2)).
\item The cohomology ring $H^{\bullet}(X_{n,k})$ may be presented as 
$H^{\bullet}(X_{n,k}) = R_{n,k}$, where each generating variable $x_i$ of $R_{n,k}$ lies in cohomological
degree $2$. Under this presentation, we may regard the variable $x_i$ as 
representing the first Chern class $c_1(\ell_i^*)$ of the dual $\ell_i^*$ of the $i^{th}$ tautological line bundle 
$\ell_i \twoheadrightarrow X_{n,k}$ over $X_{n,k}$ (Theorem~\ref{cohomology-of-x} (3)).
\item  For any Fubini word $w \in \WW_{n,k}$, we have an explicitly defined
polynomial $\symm_w \in \ZZ_{\geq 0}[\xx_n]$ 
(see Definition~\ref{fubini-schuberts}) such that $\symm_w$ represents the class
of $[\overline{C}_w]$ in $H^{\bullet}(X_{n,k})$.  In particular, the set 
$\{ \symm_w \,:\, w \in \WW_{n,k} \}$ descends to a $\ZZ$-basis of $R_{n,k}$.
\end{itemize}

The symmetric group $S_n$ acts on $X_{n,k}$ by permuting the lines in any $n$-tuple
$(\ell_1, \dots, \ell_n) \in X_{n,k}$.  This gives an action of $S_n$ on the cohomology ring 
$H^{\bullet}(X_{n,k})$.  The interpretation of the variables $x_i$ as representing 
$c_1(\ell_i^*)$ shows that our presentation $H^{\bullet}(X_{n,k}) = R_{n,k}$ is not just an identification
of graded $\ZZ$-algebras, but also of graded $\ZZ[S_n]$-modules
(see Theorem~\ref{thm:equivariance}).

Although many geometric properties of $X_{n,k}$ generalize those of the flag manifold $\Fl(n)$,
there there are some significant differences between $\Fl(n)$ and $X_{n,k}$ for general $k \leq n$.
These stem from the fact that the Poincar\'e series 
$P_{X_{n,k}}(q) = \Hilb(R_{n,k}; q^2) = \rev_q ( [k]!_{q^2} \cdot \Stir_{q^2}(n,k))$ 
of $X_{n,k}$ is not in general palindromic.
For example, we have $P_{X_{3,2}}(q) = \Hilb(R_{3,2}; q^2) = 1 + 3 q^2 + 2 q^4$.

Since $P_{X_{n,k}}(q)$ is not in general palindromic, the space $X_{n,k}$ is not homotopy equivalent to 
a compact smooth manifold for all $k \leq n$.
Since any finite CW complex is compact, our cellular decomposition
$\{ C_w \,:\, w \in \WW_{n,k} \}$ cannot come from a CW decomposition of $X_{n,k}$.
In fact, there in general exist $v, w \in \WW_{n,k}$ such that 
$C_v \not\subseteq \overline{C}_w$ and yet $C_v \cap \overline{C}_w \neq \varnothing$.

The rest of the paper is organized as follows.
In {\bf Section~\ref{Background}} we give combinatorial and geometric background.
In {\bf Section~\ref{Orbit}} we work over an arbitrary field $\KK$ and introduce the $\KK$-variety $X_{n,k}^{\KK}$
together with a related $\KK$-variety $Y_{n,k}^{\KK}$; we use a variant to Gaussian elimination to prove 
our first facts about these objects.
In {\bf Section~\ref{Cellular}} we build up the tools from cellular cohomology we will use to present
$H^{\bullet}(X_{n,k})$.
In {\bf Section~\ref{Cohomology}} we present $H^{\bullet}(X_{n,k}) = R_{n,k}$ as a graded $\ZZ$-algebra.
We do this by giving a cellular decomposition of $X_{n,k}$.  This gives rise to our Schubert basis
$\{ \symm_w \,:\, w \in \WW_{n,k} \}$ for $R_{n,k}$.
In {\bf Section~\ref{Chern}} we reformulate the identification $H^{\bullet}(X_{n,k}) = R_{n,k}$ using Chern 
classes. 
In {\bf Section~\ref{Stability}} we consider stable analogs of our results under natural towers of embeddings
$X_{n,k} \subset X_{n+1,k}$ and $X_{n,k} \subset X_{n+1,k+1}$.
In {\bf Section~\ref{Variations}} we present variants on the variety $X_{n,k}$ whose cohomology rings 
are (sometimes conjecturally) equal to variants of the ring $R_{n,k}$.
In {\bf Section~\ref{Closing}} we make closing remarks.

\section{Background}
\label{Background}

\subsection{Ordered set partitions and words}
An {\em ordered set partition of size $n$ with $k$ blocks} is a $k$-block set partition
of $[n]$ with a total order on its blocks.  We let $\OP_{n,k}$ be the family 
of ordered set partitions of $[n]$ with $k$ blocks.  As an example, we have
$\sigma = (\{2,5\} \prec \{3\} \prec \{1,4\}) \in \OP_{5,3}$.  
The {\em bar notation} for ordered set partitions uses bars to denote 
separation between blocks,
so our example ordered set partition becomes
\begin{equation*}
\sigma = (25 \mid 3 \mid 14) \in \OP_{5,3}.
\end{equation*}
The {\em star notation} for ordered set partitions uses stars to indicate connectives
between elements of the same block, so our example ordered set partition becomes
\begin{equation*}
\sigma = 2_*5 \, \, 3 \, \, 1_*4 \in \OP_{5,3}.
\end{equation*}
Rather than bars or stars, we will represent ordered set partitions with words in this
paper.

Let $w = w_1 \dots w_n$ be an arbitrary word over the positive integers.  A position $1 \leq i \leq n$ is called {\em initial in $w$} if 
$w_{i'} \neq w_i$ for all $i' < i$.   Let $\initial(w) = \{ 1 \leq i \leq n \,:\, \text{$i$ is initial in $w$} \}$ be the set of all 
initial positions in $w$.  As an example, $\initial(3343414) = \{1,3,6\}$.
The letters appearing in initial positions are called {\em initial letters}; the initial letters in $3343414$ are (from left to right)
$3, 4,$ and $1$.

A word $w = w_1 \dots w_n$ over the positive 
integers is called {\em Fubini} if, for all $i \geq 2$ such that $i$ appears in 
$w$, the letter $(i-1)$ also appears in $w$.  We let $\WW_{n,k}$ be the family
of Fubini words of length $n$ 
with maximum letter $k$.  There is a natural bijection between
$\WW_{n,k}$ and $\OP_{n,k}$.

Suppose $w = w_1 \dots w_n \in \WW_{n,k}$ is a Fubini word.  
By listing the initial letters in $w$ from left to right, 
we get
a permutation $\pi(w) = \pi_1 \dots \pi_k \in S_k$.  
For example, for $331443123 \in \WW_{9,4}$
we have $\pi(331443123) = 3142 \in S_4$.

A word $w = w_1 \dots w_n$ is called {\em convex} if $w$ contains no subword of the form
$i \dots j \dots i$ for $i \neq j$.  For example, $224443$ is convex but $244234$ is not.
The {\em convexification} $\conv(w) = w_1' \dots w_n'$ of a word $w = w_1 \dots w_n$ is the unique convex word with the same
letter multiplicities as $w$, with the initial letters appearing in the same order from left to right,
e.g. $\conv(244234) = 224443$.

Given $n \geq 0$, we have the $q$-analogs 
\begin{align}
[n]_q &:= 1 + q + q^2 + \cdots + q^{n-1}, \\
[n]!_q &:= [n]_q [n-1]_q \cdots [1]_q.
\end{align}
Recall that a permutation statistic $\stat: S_n \rightarrow \ZZ_{\geq 0}$ is called {\em Mahonian} 
if its distribution $\sum_{\pi \in S_n} q^{\stat(\pi)}$ is equal to $[n]!_q$.
The most important Mahonian statistic on permutations is the {\em inversion count}.  For 
$\pi = \pi_1 \dots \pi_n \in S_n$ the statistic $\inv(\pi)$ is
given by $\inv(\pi) = | \{ 1 \leq i < j \leq n \,:\, \pi_i > \pi_j \} |$.

More generally, a statistic $\stat: \WW_{n,k} \rightarrow \ZZ_{\geq 0}$ is called {\em Mahonian}
if its distribution $\sum_{w \in \WW_{n,k}} q^{\stat(w)}$ is equal to
$[k]!_q \cdot \Stir_q(n,k)$, where $\Stir_q(n,k)$ is the $q$-analog of the Stirling number 
recursively defined by
\begin{equation}
\Stir_q(n,k) = \begin{cases}
\delta_{k,0} & \text{if $n = 0$} \\
\Stir_q(n-1,k-1) + [k]_q \cdot \Stir_q(n-1,k) & \text{if $n > 0$}.
\end{cases}
\end{equation}

\subsection{Schubert polynomials}
For $1 \leq i \leq n-1$, let $s_i \in S_n$ be the adjacent transposition $s_i = (i, i+1)$.
The group $S_n$ has Coxeter presentation with generators $s_1, \dots, s_{n-1}$ subject to the relations
\begin{equation}
\begin{cases}
s_i^2 = 1 & 1 \leq i \leq n-1, \\
s_i s_j = s_j s_i & |i-j| > 1, \\
s_i s_{i+1} s_i = s_{i+1} s_i s_{i+1} & 1 \leq i \leq n-2.
\end{cases}
\end{equation}

For $1 \leq i \leq n-1$ the {\em divided difference operator} $\partial_i: \ZZ[\xx_n] \rightarrow \ZZ[\xx_n]$ is given by
\begin{equation}
\partial_i(f) := \frac{f - s_i.f}{x_i - x_{i+1}},
\end{equation}
where $s_i.f$ is the result of interchanging $x_i$ and $x_{i+1}$ in $f = f(x_1, \dots, x_n)$.  The operators $\partial_i$
satisfy the {\em nilCoxeter relations}
\begin{equation}
\begin{cases}
\partial_i^2 = 0 & 1 \leq i \leq n-1, \\
\partial_i \partial_j = \partial_j \partial_i & |i-j| > 1, \\
\partial_i \partial_{i+1} \partial_i = \partial_{i+1} \partial_i \partial_{i+1} & 1 \leq i \leq n-2.
\end{cases}
\end{equation}

The {\em Schubert polynomials} $\{ \symm_w \,:\, w \in S_n \}$ are the family of polynomials in
$\ZZ[\xx_n]$
defined recursively as follows.  
If $w_0 \in S_n$ is the reversal permutation with one-line notation
$w_0 = n \dots 21$, then $\symm_{w_0} = x_1^{n-1} x_2^{n-2} \dots x_{n-1}^1$.  If 
$w = w_1 \dots w_n \in S_n$
and $1 \leq i \leq n-1$ is such that $w_i > w_{i+1}$, then 
$\symm_{w  s_i} = \partial_i \symm_{w}$.
The set $\{ \symm_{\pi} \,:\, \pi \in S_n \}$ of Schubert polynomials 
descends to a $\ZZ$-basis for the coinvariant algebra
$R_n$.

We will also need the {\em double Schubert polynomials}
$\symm_w(x_1, \dots, x_n; y_1, \dots, y_n)$ depending on two sets of $n$ variables.
These are given in terms of the ordinary Schubert polynomials:
\begin{equation}
\symm_w(x_1, \dots, x_n; y_1, \dots, y_n) := 
\sum_{\substack{w = v^{-1} u \\ \inv(u) + \inv(v) = \inv(w)}}
\symm_u(x_1, \dots, x_n) \cdot \symm_v(-y_1, \dots, -y_n).
\end{equation}

\subsection{Symmetric functions}
If $n$ is a positive integer, a {\em partition of $n$} is a weakly decreasing sequence
$\lambda = (\lambda_1 \geq \cdots \geq \lambda_k)$ of positive integers with
$\lambda_1 + \cdots + \lambda_k = n$.  
The partition $\lambda$ is identified with its (English) Ferrers diagram,
consisting of $\lambda_i$ left justified boxes in row $i$.
We write $\lambda \vdash n$ to indicate that 
$\lambda$ is a partition of $n$.  
We also write $\ell(\lambda) = k$ for the number of parts of $\lambda$.

Let $\Lambda = \bigoplus_{n \geq 0} \Lambda_n$ be the ring of symmetric functions
in the infinite variable set $\xx = (x_1, x_2, \dots )$ over the ground field $\QQ(q,t)$.  
If $\lambda \vdash n$ is a partition, we let
\begin{equation*}
e_{\lambda} = e_{\lambda}(\xx), \quad
h_{\lambda} = h_{\lambda}(\xx), \quad
s_{\lambda} = s_{\lambda}(\xx), \quad
Q'_{\lambda} = Q'_{\lambda}(\xx;q), \quad
\widetilde{H}_{\lambda} = \widetilde{H}_{\lambda}(\xx;q,t)
\end{equation*}
denote the {\em elementary, (complete) homogeneous, Schur, dual Hall-Littlewood,
and modified Macdonald} symmetric function attached to $\lambda$
(respectively).
Each of the five sets 
$\{ e_{\lambda} \,:\, \lambda \vdash n \},
\{h_{\lambda} \,:\, \lambda \vdash n \},
\{s_{\lambda} \,:\, \lambda \vdash n \},
\{ Q'_{\lambda} \,:\, \lambda \vdash n \},
\{ \widetilde{H}_{\lambda} \,:\, \lambda \vdash n \}$
gives a basis for the vector space $\Lambda_n$ of homogeneous degree $n$
symmetric functions.

If $f \in \Lambda$ is any symmetric function, we have the Macdonald eigenoperator
$\Delta'_f: \Lambda \rightarrow \Lambda$ defined in the introduction.
We will also refer to an `unprimed' version 
$\Delta_f: \Lambda \rightarrow \Lambda$ defined on the Macdonald basis by
\begin{equation}
\Delta_f: \widetilde{H}_{\mu} \mapsto f(\dots, q^i t^j , \dots ) \cdot \widetilde{H}_{\mu},
\end{equation}
where $(i,j)$ range over {\em all} coordinates of the boxes in the Ferrers diagram of $\mu$
(including the northwest corner $(0,0)$).  For example, if $\mu = (3,2)$ we fill
in the boxes of $\mu$ as
\begin{equation*}
\begin{Young}
1 & q & q^2 \cr
t & qt
\end{Young}
\end{equation*}
so that
\begin{equation*}
\Delta_f: \widetilde{H}_{(3,2)} \mapsto f(1, q, q^2, t, qt) \cdot \widetilde{H}_{(3,2)}.
\end{equation*}
The only difference between the operator $\Delta_f$ and the operator $\Delta'_f$
is the presence of $1$ among the arguments of the eigenvalue of 
$\widetilde{H}_{\mu}$.  Using the notation of plethysm, the operators $\Delta_f$ and
$\Delta'_f$ may be described more succinctly as 
\begin{align}
&\Delta_f:  \widetilde{H}_{\mu} \mapsto f[B_{\mu}] \cdot \widetilde{H}_{\mu}, \\
&\Delta'_f: \widetilde{H}_{\mu} \mapsto f[B_{\mu} - 1] \cdot \widetilde{H}_{\mu}.
\end{align}

Schur functions are related to Schubert polynomials as follows.
A permutation $w = w_1 \dots w_m \in S_m$ is called 
{\em Grassmannian} if there exists a unique number $1 \leq n \leq m-1$ such that
$w_n > w_{n+1}$.  If $w \in S_m$ is a Grassmannian permutation whose 
unique descent is at $n$, we have
\begin{equation}
\symm_{w} = s_{\lambda}(x_1, x_2, \dots, x_n),
\end{equation}
where $\lambda = (w_n - n, \dots, w_2 - 2, w_1 - 1)$.  
More generally, if $w \in S_m$ is any permutation for which
$w_{n+1} < w_{n+2} < \cdots < w_m$ then
$\symm_{w} \in \ZZ[x_1, \dots, x_n]$ only involves the first $n$ variables
$x_1, \dots, x_n$.

The irreducible representations of the symmetric group $S_n$ are indexed by 
partitions $\lambda \vdash n$.
If $\lambda \vdash n$ is a partition, we let $S^{\lambda}$ be the associated 
irreducible $S_n$-module.  
If $V$ is any finite-dimensional $S_n$-module, there exist unique integers
$c_{\lambda}$ such that 
$V \cong_{S_n} \bigoplus_{\lambda \vdash n} c_{\lambda} S^{\lambda}$.
The {\em Frobenius characteristic} $\Frob(V) \in \Lambda_n$ is the symmetric function
\begin{equation}
\Frob(V) := \sum_{\lambda \vdash n} c_{\lambda} s_{\lambda}.
\end{equation}
More generally, if $V = \bigoplus_{d \geq 0} V_d$ is a graded $S_n$-module
such that each graded piece $V_d$ is finite-dimensional, the 
{\em graded Frobenius characteristic} is 
\begin{equation}
\grFrob(V;q) := \sum_{d \geq 0} \Frob(V_d) \cdot q^d.
\end{equation}

\subsection{Gr\"obner theory}
If $\KK$ is any
field, a {\em term order} is a total order $<$ on the monomials in
$\KK[\xx_n]$ such that 
\begin{itemize}
\item  we have $1 \leq m$ for all monomials $m$, and
\item  if $m, m',$ and $m''$ are monomials with $m \leq m'$, then
$m \cdot m'' \leq m' \cdot m''$.
\end{itemize}
The only term order we will use in this paper is the {\em lexicographic} term order
defined by $x_1^{a_1} \cdots x_n^{a_n} < x_1^{b_1} \cdots x_n^{b_n}$
if and only if there exists $1 \leq i \leq n$ such that 
$a_1 = b_1, \dots, a_{i-1} = b_{i-1}$ and $a_i < b_i$.

Fix a term order $<$ on the monomials in $\KK[\xx_n]$.  If $f \in \KK[\xx_n]$ is any
nonzero polynomial, let $\initial_<(f)$ be the largest monomial appearing in $f$
with respect to $<$.  If $I \subseteq \KK[\xx_n]$ is an ideal, 
the {\em standard monomial basis} is the collection
\begin{equation}
\{ \text{monomials $m \in \KK[\xx_n]$} \,:\, \initial_<(f) \nmid m \text{ for all $f \in I - \{0\}$} \}
\end{equation}
of all monomials in $\KK[\xx_n]$ which are not divisible by the leading terms
of any nonzero elements of $I$.
The standard monomial basis  descends to a basis of the $\KK$-vector
space $\KK[\xx_n]/I$.

\subsection{Matrix notation}
In this paper we will use lowercase letters $a, b, u, m \dots $ to refer to matrices 
in order to avoid confusion with the uppercase letters $B, T, U, GL_n, \dots$
used to denote {\em sets} of matrices.
Let $m = (m_{i,j})$ be a $k \times n$ matrix with entries in some field $\KK$.
Given any subsets $I \subseteq [k]$ and $J \subseteq [n]$ with $|I| = |J|$,
the {\em $I,J$-minor of $m$} is
\begin{equation}
\Delta_{I,J}(m) :=
\text{determinant of the submatrix of $m$ with row set $I$ and column set $J$}.
\end{equation}
We adopt the convention $\Delta_{\varnothing, \varnothing}(m) = 1$.

\subsection{Chern classes}
Let $X$ be a smooth complex manifold and let 
$E \twoheadrightarrow X$ be a complex vector bundle over $X$.
For all $p > 0$, we have the {\em Chern class} 
$c_p(E) \in H^{2p}(X)$, where 
$H^{2p}(X)$ is singular cohomology with integer coefficients.  
We let $c_{\bullet}(E) := 1 + \sum_{p \geq 0} c_p(E) \cdot t^p$ be the 
generating function for these classes.

We will use the following facts about Chern classes.
\begin{itemize}
\item  If $E = \CC^r \times X$ is a trivial vector bundle over $X$, then
$c_p(E) = 0$ for all $p > 0$.
\item  If $E$ is a rank $r$ vector bundle, then $c_p(E) = 0$ for all $p > r$.
\item  If $E$ and $F$ are two vector bundles over $X$ and $E \oplus F$ is the 
associated direct sum, we have the {\em Whitney sum formula}
\begin{equation*}
c_{\bullet}(E \oplus F) = c_{\bullet}(E) \cdot c_{\bullet}(F).
\end{equation*}
\item If $\ell$ is a {\em line bundle} (i.e. a rank 1 vector bundle) over $X$ and $\ell^*$ is the dual
bundle, then $c_1(\ell^*) = - c_1(\ell)$.
\end{itemize}

\section{The orbit sets $X_{n,k}^{\KK}$ and $Y_{n,k}^{\KK}$}
\label{Orbit}

For this section, we work over an arbitrary field $\KK$.  
Given positive integers $k \leq n$, we let $\Mat_{k \times n}(\KK)$ stand for the set of $k \times n$ matrices with entries in 
$\KK$.  If $U \subseteq GL_k(\KK)$ denotes the subgroup of {\em lower} 
triangular $k \times k$ matrices with entries on the diagonal equal to $1$
and $T \subseteq GL_n(\KK)$ denotes the subgroup of diagonal matrices, then $\Mat_{k \times n}(\KK)$ carries a
left action of $U$ (on rows) and a right action of $T$ (on columns). 

Let $\MMM^{\KK}_{n,k} \subseteq \Mat_{k \times n}(\KK)$ be the set of full rank matrices in $\Mat_{k \times n}$
which have no columns equal to zero.  The set $\MMM^{\KK}_{n,k}$ is stable under the left action of $U$ and the right
action of $T$.  We define the sets of orbits $X_{n,k}^{\KK}$ and $Y_{n,k}^{\KK}$ by
\begin{align}
X_{n,k}^{\KK} &:= \MMM^{\KK}_{n,k}/T = \{ mT \,:\, m \in \MMM_{n,k}^{\KK} \}, \\
Y_{n,k}^{\KK} &:= U \backslash \MMM_{n,k}^{\KK} / T = \{ UmT \,:\, m \in \MMM_{n,k}^{\KK} \}.
\end{align}
The orbits in $Y_{n,k}^{\KK}$ are unions of orbits in $X_{n,k}^{\KK}$; this gives a canonical surjection of sets
$X_{n,k}^{\KK} \twoheadrightarrow Y_{n,k}^{\KK}$.  The set $X_{n,k}^{\KK}$ may be identified with the 
collection of $n$-tuples of $1$-dimensional subspaces $(\ell_1, \dots, \ell_n)$ of $\KK^k$ 
such that $\ell_1 + \cdots + \ell_n = \KK^k$.  

The product group $U \times T$ acts on the set $\Mat_{k \times n}(\KK)$ by the rule 
$(u,t).m := umt$. (This is a well-defined action because $T$ is abelian.)
Although this action is not free, it becomes free when restricted to $\MMM_{n,k}^{\KK}$.

\begin{lemma}
\label{action-is-free}
The action of the product group $U \times T$ on the set $\MMM_{n,k}^{\KK}$ given by
$(u,t).m := umt$ is free.
\end{lemma}

\begin{proof}
Let $m \in \MMM_{n,k}^{\KK}$ and $(u,t) \in U \times T$ be such that $umt = m$.  We must show that $(u,t) = (I_k, I_n)$
are the identity matrices of sizes $k$ and $n$.

The equation $umt = m$ is an equation of $k \times n$ matrices.  Since $m$ has full rank,
there exists $I \subseteq [n]$ such that $|I| = k$ and the $k \times k$ submatrix of $m$ with column set $I$ is nonsingular.
Let $m'$ be this $k \times k$ submatrix and let $t' \in GL_k(\KK)$ be the diagonal matrix obtained by restricting $t$ 
to the rows and columns indexed by $I$.  We have the equation $u m' t' = m'$ of $k \times k$ matrices, where $m'$ is invertible.
We argue that $u = t' = I_k$ as follows.

Write $t' = \mathrm{diag}(t_1, \dots, t_k)$.  Since $u$ is unipotent, for $1 \leq r \leq k$ and any 
$J \subseteq [k]$ with $|J| = r$ the minor
$\Delta_{[r], J}(m')$ is given by
\begin{equation}
\label{free-product-action}
\Delta_{[r], J}(m') = \Delta_{[r], J}(u m' t') = \Delta_{[r], J} (m' t') = \prod_{j \in J} t_j \cdot \Delta_{[r], J}(m').
\end{equation}

Since $m'$ is an invertible $k \times k$ matrix, we can find a nest of sets 
$\varnothing \subset J_1 \subset J_2 \subset \cdots \subset J_k = [k]$ such that $|J_r| = r$ and
$\Delta_{[r], J_r}(m') \neq 0$ for all $1 \leq r \leq k$. 
(Indeed, expanding $\det(m') \neq 0$ along the last row implies the existence of
$J_{k-1} \subset [k]$ with $|J_{k-1}| = k-1$ and $\Delta_{[k-1], J_{k-1}}(m') \neq 0$; now induct on $k$.)
Applying Equation~\eqref{free-product-action}
to each of the sets $J = J_1, \dots, J_k$ in turn implies $t' = I_k$.  The equation $um't' = m'$ now reads
$um' = m'$.  Since $m'$ is invertible, we have $u = I_k$.

Consider the original equation $umt = m$ of $k \times n$ matrices.  We know $u = I_k$, so that this equation
reads $mt = m$.  Since $m \in \MMM_{n,k}^{\KK}$, the matrix $m$ does not have any zero columns.  
This implies $t = I_n$.
\end{proof}

We may use Fubini words and a version of Gaussian elimination to parametrize the orbits in
$Y_{n,k}^{\KK}$, and hence in $X_{n,k}^{\KK}$.
More precisely, given a Fubini word $w = w_1 \dots w_n \in \WW_{n,k}$, the {\em pattern matrix}
$\PM(w)$ is the $k \times n$ matrix with entries in $\{0, 1, \star \}$ whose entries 
$\PM(w)_{i,j}$ (for $1 \leq i \leq k$ and $1 \leq j \leq n$) are as follows.
The rule for constructing the $j^{th}$ column of $\PM(w)$ depends on whether $j$ is initial in $w = w_1 \dots w_n$.
\begin{itemize}
\item  We have $\PM(w)_{i,j} = 1$ if and only if $w_j = i$.
\item  Suppose $j \in \initial(w)$ is an initial position of $w$ and $w_j \neq i$.  If $w_j > i$ and there exists $j' < j$ with
$w_{j'} = i$ then $\PM(w)_{i,j} = \star$.  Otherwise $\PM(w)_{i,j} = 0$.
\item  Suppose $j \notin \initial(w)$ is not an initial position of $w$ and $w_j \neq i$.  
If the first occurrence of $i$ in $w = w_1 \dots w_n$ is before the first occurrence of 
$w_j$ in $w = w_1 \dots w_n$ then $\PM(w)_{i,j} = \star$.
Otherwise $\PM(w)_{i,j} = 0$.
\end{itemize}

As an example of the pattern matrix construction, suppose $w = 2331231 \in \WW_{7,3}$.  The initial positions are
$\initial(w) = \{1,2,4\}$ and the pattern matrix is
\begin{equation*}
\PM(w) = \PM(2331231) = 
\begin{pmatrix}
0 & 0 & 0 & 1 & 0 & 0 & 1 \\
1 & \star & \star & 0 & 1 & \star & \star \\
0 & 1 & 1 & 0 & 0 & 1 & \star
\end{pmatrix}.
\end{equation*}

If $w \in \WW_{n,k}$ is a Fubini word, a matrix $m \in \Mat_{k \times n}(\KK)$ is said to {\em fit the pattern of $w$}
if $m$ can be obtained by replacing the $\star$'s in the pattern matrix $\PM(w)$ with elements of $\KK$.
Any matrix $m$ which fits the pattern of $w$ has full rank (indeed, the maximal minor with column set $\initial(w)$ equals
$\pm 1$) and has no zero columns, and so lies in $\MMM_{n,k}^{\KK}$.  
In fact, we have the following version of Gaussian elimination for the orbits in $Y_{n,k}^{\KK}$.

\begin{proposition}
\label{pattern-matrix-representatives}
Let $m \in \MMM_{n,k}^{\KK}$ and let $UmT \in Y_{n,k}^{\KK}$ be the corresponding orbit.  There exists a unique
Fubini word $w \in \WW_{n,k}$ and a unique matrix $m' \in \MMM_{n,k}^{\KK}$ which fits the pattern of $w$ such that
$m' \in UmT$.
\end{proposition}

\begin{proof} (Sketch)
The left action of $U$ allows us to add any nonzero multiple of a row of $m$ to a lower row and the right
action of $T$ allows us to scale the columns of $m = (m_{i,j})$.
To form $m'$ from $m$, we process the columns $1, 2, \dots, n$ of $m$ from left to right.

Starting with column $1$,  let $i$ be minimal such that $m_{i,1} \neq 0$.  Use the action of $U$ to clear every 
entry of $m$ below $(i,1)$ and then use the action of $T$ to scale $m_{i,1}$ to $1$.  We form the one-letter word
$w_1 = i$.

Assuming columns $1, 2, \dots, j-1$ have already been processed and we currently have the word $w_1 \dots w_{j-1}$, 
we look at column $j$ in our intermediate matrix $m$.
The processing of column $j$ depends on whether there exists a letter $1 \leq i \leq k$ not appearing in the word
$w_1 \dots w_{j-1}$ such that $m_{i,j} \neq 0$.
\begin{itemize}
\item  Suppose there exists $1 \leq i \leq k$ not appearing in the word $w_1 \dots w_{j-1}$ such that  $m_{i,j} \neq 0$.
Choose such a value of $i$ minimal.  Use the action of $U$ to clear every entry of $m$ below $(i,j)$, and then use
the action of $T$ to scale $m_{i,j}$ to $1$.  Append $i$ to the word $w_1 \dots w_{j-1}$ to get the $j$-letter word
$w_1 \dots w_{j-1} i$.
\item  Suppose for every index $1 \leq i \leq k$ such that $m_{i,j} \neq 0$, the letter $i$ appears in $w_1 \dots w_{j-1}$.
Let $(i_1, i_2, \dots, i_r)$ be the initial letters of $w$, read from left to right in order of their initial positions in $w$.
Choose $s$ {\em maximal} such that $m_{i_s,j} \neq 0$ and let $i = i_s$.  Use the action of $T$ so scale $m_{i,j}$ to $1$.
Append $i$ to the word $w_1 \dots w_{j-1}$ to get the $j$-letter word $w_1 \dots w_{j-1} i$.
\end{itemize}

After each of the columns $1, 2, \dots, n$ have been processed, we get a matrix $m' \in \MMM_{n,k}$ and 
a word $w = w_1 \dots w_n$ in the alphabet $[k]$.  Since $m'$ is invertible, the letters $1, 2, \dots, k$
all appear in $w$, so that $w \in \WW_{n,k}$ is Fubini.  By construction, the matrix $m'$ fits the pattern of $w$
and we have $m' \in UmT$.  This proves existence.

It remains to show uniqueness.  Let $a' \in \MMM_{n,k}^{\KK}$ and $v = v_1 \dots v_n \in \WW_{n,k}$ be such that
$a'$ fits the pattern of $v$ and $a' \in UmT$.  Write $m' = (m'_1 \cdots m'_n)$ and $a' = (a'_1 \cdots a'_n)$ for the columns
of $m'$ and $a'$; we have $a' = u m' t$ for some  $u \in U$ and $t \in T$.
As with the existence proof, the idea is to prove the prefix equality $w_1 \dots w_j = v_1 \dots v_j$
and the corresponding matrix equality $(m'_1 \cdots m'_j) = (a'_1 \cdots a'_j)$ by induction on $j$.  We leave the details
to the reader.
\end{proof}

To see how the `Gaussian elimination' of 
Proposition~\ref{pattern-matrix-representatives} works, let 
$(n,k) = (7,3)$, consider the ground field $\KK = \QQ$, and let
\begin{equation*}
m = \begin{pmatrix}
0 & 0 & 0 & 2 & 0 & 0 & 3 \\
1 & 6 & 0 & 2 & 1 & 4 & 0 \\
-\frac{1}{3} & 0 & -4 & -\frac{8}{3} &  -\frac{1}{3} & -\frac{2}{3} & 3
\end{pmatrix}  \in \MMM_{7,3}^{\QQ}.
\end{equation*}
We want to find the unique matrix $m' \in UmT$ which fits the pattern of some
Fubini word $w \in \WW_{7,3}$.  To do this, we process the columns of 
$m$ from left to right as follows.
\begin{align*}
\begin{pmatrix}
0 & 0 & 0 & 2 & 0 & 0 & 3 \\
1 & 6 & 0 & 2 & 1 & 4 & 0 \\
-\frac{1}{3} & 0 & -4 & -\frac{8}{3} &  -\frac{1}{3} & -\frac{2}{3} & 3
\end{pmatrix} 
&\leadsto
\begin{pmatrix}
0 & 0 & 0 & 2 & 0 & 0 & 3 \\
1 & 6 & 0 & 2 & 1 & 4 & 0 \\
0 & 2 & -4 & -2 &  0 & 2 & 3
\end{pmatrix} 
\leadsto
\begin{pmatrix}
0 & 0 & 0 & 2 & 0 & 0 & 3 \\
1 & 3 & 0 & 2 & 1 & 4 & 0 \\
0 & 1 & -4 & -2 &  0 & 2 & 3
\end{pmatrix} 
\\ 
 \leadsto \begin{pmatrix}
0 & 0 & 0 & 2 & 0 & 0 & 3 \\
1 & 3 & 0 & 2 & 1 & 4 & 0 \\
0 & 1 & 1 & -2 &  0 & 2 & 3
\end{pmatrix}& \leadsto
\begin{pmatrix}
0 & 0 & 0 & 1 & 0 & 0 & 3 \\
1 & 3 & 0 & 0 & 1 & 4 & -3 \\
0 & 1 & 1 & 0 & 0 & 2 & 6
\end{pmatrix} 
\leadsto
\begin{pmatrix}
0 & 0 & 0 & 1 & 0 & 0 & 1 \\
1 & 3 & 0 & 0 & 1 & 2 & -1 \\
0 & 1 & 1 & 0 & 0 & 1 & 2
\end{pmatrix} = m'.
\end{align*}
Explicitly, the algorithm in the proof of 
Proposition~\ref{pattern-matrix-representatives}
\begin{itemize}
\item uses the action of $U$ to add $1/3$ times row 2 to row 3,
\item uses the action of $T$ to scale column 2 by $1/2$,
\item uses the action of $T$ to scale column 3 by $-1/4$,
\item uses the action of $U$ to add $-1$ times row $1$ to row $2$, uses the 
action of $U$ to add $1$ times row $1$ to row $3$, and uses the action of $T$ 
to scale column $4$ by $1/2$, and finally
\item uses the action of $T$ to scale column $6$ by $1/2$ and column 7 by $1/3$.
\end{itemize}
We conclude that $m'$ fits the pattern of
$2331231 \in \WW_{7,3}$.

Given any Fubini word $w \in \WW_{n,k}$, let $\widehat{C}_w^{\KK} \subseteq \MMM_{n,k}^{\KK}$ 
be the set of matrices which fit the 
pattern of $w$.
We get  subsets $C_w^{\KK} \subseteq X_{n,k}^{\KK}$ and $D_w^{\KK} 
\subseteq Y_{n,k}^{\KK}$ by
\begin{equation}
C_w^{\KK} := \text{ image of $U \widehat{C}_w^{\KK}$ in $X_{n,k}^{\KK}$,} \quad
D_w^{\KK} := \text{ image of $\widehat{C}_w^{\KK}$ in $Y_{n,k}^{\KK}$.}
\end{equation}
The notational strangeness of letting $C_w^{\KK}$ rather than $D_w^{\KK}$ correspond to
$U \widehat{C}_w^{\KK}$ 
is rationalized by the greater geometric importance of the space $X_{n,k}$.
By Lemma~\ref{action-is-free} and Proposition~\ref{pattern-matrix-representatives}
we have  disjoint union decompositions
\begin{equation}
X_{n,k}^{\KK} = \bigsqcup_{w \in \WW_{n,k}} C_w^{\KK}, \quad Y_{n,k}^{\KK} = \bigsqcup_{w \in \WW_{n,k}} D_w^{\KK}
\end{equation}
and the natural maps 
$U \times \widehat{C}_w^{\KK} \xrightarrow{\times} U \widehat{C}_w^{\KK} \rightarrow C_w^{\KK}$
and $\widehat{C}_w^{\KK} \rightarrow D_w^{\KK}$
are both bijections.

Let $q$ be a prime power and let $\FF_q$ be the field with $q$ elements.  We conclude this section by 
relating $Y_{n,k}^{\FF_q}$ to the Mahonian distribution on $\WW_{n,k}$.   

The {\em dimension} $\dim(w)$ of a Fubini word $w \in \WW_{n,k}$ is given by
\begin{equation}
\dim(w) := \text{number of $\star$'s in the pattern matrix $\PM(w)$}.
\end{equation}
In particular, the set $\widehat{C}_w^{\KK}$ of matrices may be viewed as a copy
of the product
$\KK^{\dim(w)}$.
Our pattern matrix example shows that 
if $w = 2331231 \in \WW_{7,3}$ then $\dim(w) = 5$.
The dimension $\dim(w)$ of a Fubini word $w \in \WW_{n,k}$
may be alternatively described as follows.

\begin{observation}
\label{dimension-alternative}
Let $w = w_1 \dots w_n \in \WW_{n,k}$ and let $\pi = \pi(w) =  \pi_1 \dots \pi_k \in S_k$ 
be the associated permutation.
We have
\begin{align}
\dim(w) &= {k \choose 2} - \inv(\pi) + \sum_{i \notin \initial(w)} (\pi^{-1}_{w_i} - 1) \\
&=  - \inv(\pi) - n + \sum_{i = 1}^n \pi^{-1}_{w_i}.
\end{align}
\end{observation}

To see how the formulas of 
Observation~\ref{dimension-alternative} work, we continue our example  
$(n,k) = (7,3)$ and $w = 2331231 \in \WW_{n,k}$.  The initial positions of $w$
are $\initial(w) = \{1,2,4\}$ and the permutation corresponding to $w$
is $\pi = 231 \in S_3$.
In particular, we have $\inv(\pi) = 2$ and
$\pi^{-1}_1 = 3, \pi^{-1}_2 = 1,$ and $\pi^{-1}_3 = 2$.  We have the expected
values
\begin{align*}
{k \choose 2} - \inv(\pi) + \sum_{i \notin \initial(w)} (\pi^{-1}_{w_i} - 1) &=
{3 \choose 2} - 2 + [(2-1) + (1-1) + (2-1) + (3-1)] = 5, \\
- \inv(\pi) - n + \sum_{i = 1}^n \pi^{-1}_{w_i} &=
- 2 - 7 + [1 + 2 + 2 + 3 + 1 + 2 + 3] = 5.
\end{align*}

\begin{proposition}
\label{dimension-is-mahonian}
Let $k \leq n$ be positive integers.
 The dimension 
$\dim$ is a Mahonian statistic on $\WW_{n,k}$.
\end{proposition}

\begin{proof}
Let $D_{n,k}(q) := \sum_{w \in \WW_{n,k}} q^{\dim(w)}$ be the generating function for $\dim$.  We want to show
$D_{n,k}(q) = [k]!_q \cdot \Stir_q(n,k)$.  This reduces to showing the recursion
\begin{equation}
\label{desired-dimension-recursion}
D_{n,k}(q) = [k]_q \cdot (D_{n-1,k-1}(q) + D_{n-1,k}(q)),
\end{equation}
where we interpret $D_{n-1,k}(q) = 0$ if $k = n$.

Equation~\eqref{desired-dimension-recursion} will be proven by considering a 
growth process which uniquely generates all possible Fubini words.  Let $w = w_1 \dots w_n \in \WW_{n,k}$ 
be a Fubini word of length $n$ with maximum letter $k$.  There are two ways to grow $w$ to a Fubini word of length
$n+1$.
\begin{itemize}
\item  Choose a letter $1 \leq j \leq k$ and transform $w = w_1 \dots w_n \mapsto w_1 \dots w_n j \in \WW_{n+1,k}$.
This is called a {\em star growth at $j$}.
\item  Choose a letter $1 \leq j \leq k+1$, increase the letters $j, j+1, \dots, k$ in $w = w_1 \dots w_n$ by $1$ to get a new
word $v_1 \dots v_n$, and transform $w = w_1 \dots w_n \mapsto v_1 \dots v_n j \in \WW_{n+1,k+1}$.
This is called a {\em bar growth at $j$}.
\end{itemize}
The star and bar terminology comes from the corresponding effect on ordered set partitions; whereas a bar growth increases the
number of blocks (and hence the number of bars), a star growth increases the number of `connectives' (which are 
represented as stars); see e.g. \cite[Sec. 4]{HRS}.  

As an example of these growth processes, suppose $w = 21124231 \in \WW_{7,4}$.  Star growth at $1, 2, 3,$ and $4$ 
leads to the following 
four elements of $\WW_{8,4}$:
\begin{equation*}
21124231{\bf 1}, \quad 21124231{\bf 2}, \quad 21124231{\bf 3}, \quad 21124231{\bf 4}.
\end{equation*}
Bar growth at $1, 2, 3, 4,$ and $5$ leads to the following five elements of $\WW_{8,5}$:
\begin{equation*}
32235342{\bf 1}, \quad 31135341{\bf 2}, \quad 21125241{\bf 3}, \quad 21125231 {\bf 4}, \quad 21124231 {\bf 5}.
\end{equation*}
Starting with a given word $w \in \WW_{n,k}$, there are always $k$ possibilities for star growth and $k+1$ possibilities
for bar growth.  It is clear that any Fubini word $w \in \WW_{n,k}$ may be obtained uniquely from $1 \in \WW_{1,1}$ by 
a sequence of bar and star growths.

Let us consider the effect of star and bar growths on dimension.  Fix $w \in \WW_{n,k}$ and let $\pi(w) = \pi_1 \dots \pi_k \in S_k$
be the corresponding permutation.
\begin{itemize}
\item  Let $1 \leq j \leq k$ and suppose $v = v_1 \dots v_n v_{n+1} \in \WW_{n+1,k}$ is obtained from $w$ by a star growth at $j$.
Then $\pi(v) = \pi(w)$ so that $\dim(v) = \dim(w) + (v_{n+1} - 1) = \dim(w) + (\pi^{-1}_j - 1)$.  Since $\pi_1 \dots \pi_k$ is a 
permutation in $S_k$, performing star growth at all  $1 \leq j \leq k$ gives dimension increases of $0, 1, \dots, k-1$.
\item  Let $1 \leq i \leq k+1$ and suppose $v = v_1 \dots v_n v_{n+1} \in \WW_{n+1,k+1}$ is obtained from $w$ by a bar 
growth at $j$.   Then $\pi(v) \in S_{k+1}$ is obtained from $\pi(w) \in S_k$ by appending the letter $j$ to the end of 
$\pi(w) = \pi_1 \dots \pi_k$ and increasing the letters $j, j+1, \dots, k$ in $\pi(w)$ by $1$.  It follows that 
$\dim(v) = \dim(w) + (j-1)$.  Performing bar growth at all $1 \leq j \leq k+1$ gives dimension increases of 
$0, 1, \dots, k$.
\end{itemize}
The effects on dimension of bar and star growth noted above are equivalent to the desired recursion
\eqref{desired-dimension-recursion} for $D_{n,k}(q)$.
\end{proof}

There are many statistics on $\OP_{n,k}$ which are known to be Mahonian; see
\cite{RemmelWilson, RhoadesOSP, Stein} for some examples.
The Mahonian statistic $\dim$ on $\WW_{n,k}$ (or equivalently on $\OP_{n,k}$) appears to be new.
The dimension statistic leads to the following interpretation of the Mahonian distribution 
$[k]!_q \cdot \Stir_q(n,k)$ in terms of finite fields.

\begin{corollary}
\label{y-set-count}
Let $k \leq n$ be positive integers,
let $q$ be a prime power, and let $\FF_q$ be the field with $q$ elements.  We have
\begin{equation}
|Y_{n,k}^{\FF_q}| = [k]!_q \cdot \Stir_q(n,k).
\end{equation}
\end{corollary}

\begin{proof}
We have $Y_{n,k}^{\FF_q} = \bigsqcup_{w \in \WW_{n,k}} D^{\FF_q}_w$ and the natural map
$\widehat{C}^{\FF_q}_w \rightarrow D^{\FF_q}_w$ is a bijection.  By the definition of dimension 
$|\widehat{C}^{\FF_q}_w| = q^{\dim(w)}$.  Now apply Proposition~\ref{dimension-is-mahonian}.
\end{proof}

Corollary~\ref{y-set-count} may be viewed as an analog of the classical result that 
the number of complete flags in the $\FF_q$-vector space $(\FF_q)^n$ is 
$[n]!_q$.  Lemma~\ref{action-is-free} and Proposition~\ref{pattern-matrix-representatives} give the enumeration
\begin{equation}
|X_{n,k}^{\FF_q}| = q^{{k \choose 2}} \cdot [k]!_q \cdot \Stir_q(n,k).
\end{equation}

\section{Cellular Decompositions}
\label{Cellular}

For the remainder of the paper, we specialize our  field to $\CC$.  We drop the superscript field reminders,
so $\MMM_{n,k}^{\CC} = \MMM_{n,k}, X_{n,k}^{\CC} = X_{n,k}, C_w^{\CC} = C_w$, etc.

The Euclidean topology on $\MMM_{n,k}$ induces a quotient topology on the space $X_{n,k}$ and $Y_{n,k}$.
The natural projection $X_{n,k} \twoheadrightarrow Y_{n,k}$ is the (continuous) quotient topology 
projection.  The following example illustrates the bad geometric behavior of the space $Y_{n,k}$, despite
the combinatorial result of Corollary~\ref{y-set-count}.

\begin{example}
\label{y-is-bad}
Let $(n,k) = (3,2)$.  For any complex number $t \neq 0$ let $m(t)$ be the point 
\begin{equation*}
m(t) = \begin{pmatrix} 0 & 1 & 1 \\ 1 & 0 & t + 1/t \end{pmatrix} \in Y_{3,2}.
\end{equation*}
We claim that the limit of $m(t)$ as $t \rightarrow \infty$ takes on more than one value in $Y_{3,2}$
(with respect to the quotient topology induced from the Euclidean topology on $\MMM_{3,2}$).  In particular, the space
$Y_{3,2}$ is not Hausdorff.  

Indeed, we have
\begin{equation*}
\lim_{t \rightarrow \infty} m(t) = \lim_{t \rightarrow \infty} \begin{pmatrix} 0 & 1 & 1 \\ 1 & 0 & t + 1/t \end{pmatrix} 
= \lim_{t \rightarrow \infty} \begin{pmatrix} 0 & 1 & 1/t \\ 1 & 0 & 1 + 1/t^2 \end{pmatrix} = 
\begin{pmatrix} 0 & 1 & 0 \\ 1 & 0 & 1 \end{pmatrix}
\end{equation*}
and also
\begin{equation*}
\lim_{t \rightarrow \infty} m(t) = \lim_{t \rightarrow \infty} \begin{pmatrix} 0 & 1 & 1 \\ 1 & 0 & t + 1/t \end{pmatrix}  =
\lim_{t \rightarrow \infty} \begin{pmatrix} 0 & 1 & 1 \\ 1 & -t & 1/t \end{pmatrix} =
\lim_{t \rightarrow \infty} \begin{pmatrix} 0 & -1/t & 1 \\ 1 & 1 & 1/t \end{pmatrix} =
\begin{pmatrix} 0 & 0 & 1 \\ 1 & 1 & 0 \end{pmatrix}.
\end{equation*}
These two limits represent different points in $Y_{3,2}$.
\end{example}

The space $X_{n,k}$ is a much nicer geometric object than $Y_{n,k}$.  Writing $\PP^{k-1}$  for the 
complex projective space of $1$-dimensional subspaces of $\CC^k$, we may identify
\begin{equation*}
X_{n,k} = \{ \ell_{\bullet} = (\ell_1, \dots, \ell_n) \,:\, \
\text{$\ell_i \in \PP^{k-1}$ for $1 \leq i \leq n$ and $\ell_1 + \cdots + \ell_n = \CC^k$} \}.
\end{equation*}
Explicitly, the $j^{th}$ column $m_j$ of a matrix $m = (m_1 \cdots m_n)$ representing a point in $X_{n,k}$
spans the line $\ell_j$.
The space $X_{n,k}$ is an open subvariety of the $n$-fold product $(\PP^{k-1})^n$, and as such 
is a smooth variety of complex dimension $n(k-1)$.

We want to compute the cohomology of $X_{n,k}$.  To do this, 
we make use of the embedding 
$X_{n,k} \subseteq (\PP^{k-1})^n$ and realize the disjoint union decomposition 
$X_{n,k} = \bigsqcup_{w \in \WW_{n,k}} C_w$ as a `cellular decomposition' of the 
space $X_{n,k}$.   
We work out the theory of cellular decompositions in greater generality as follows.

Let $X$ be a complex algebraic variety or a complex manifold of (complex) dimension $d$.
We let $X^+ = X \cup \{ \infty \}$ be the one-point compactification of $X$, where $\infty$ denotes the point at infinity.
The {\em Borel-Moore homology groups} of $X$ are the relative homology groups
$\bar{H}_k(X) := H_k(X^+, \{\infty\}; \ZZ)$.
We will need the following facts about the Borel-Moore homology groups $\bar{H}_{\bullet}(X)$; 
see \cite[Ch. 19]{FultonIntersection} or \cite[\S 2.6]{ChrissGinzburg}.
\begin{enumerate}
\item  If $X$ is an irreducible variety (not necessarily smooth), 
then $\bar{H}_{2d}(X) \cong \ZZ$ and has a canonical choice of generator, the 
{\em fundamental class} $[X]$ of $X$.  If $i > 2d$ or $i$ is odd, then $\bar{H}_i(X) = 0$.
\item  Poincar\'e duality: If $X$ is a complex manifold, we have 
$\bar{H}_i(X) \cong H^{2d-i}(X)$, where $H^{\bullet}$ is singular 
cohomology with integer coefficients.
Explicitly, this isomorphism is induced by capping with the fundamental class of
$X$ (see e.g. \cite[Appendix B.2.8]{PetersSteenbrink}).
\item  Pushforward maps:  If $f: Z \rightarrow X$ is a continuous proper map, then the induced
map $f^+: Z^+ \rightarrow X^+$ is also continuous, and so induces a map $f_*: \bar{H}_k(Z) \rightarrow \bar{H}_k(X)$.
The inclusion of a closed set is proper.  If $\iota: Z \hookrightarrow X$ is a closed subvariety of $X$, we abuse
notation by defining
$[Z] \in \bar{H}_{\bullet}(X)$ to be $\iota_*([Z])$.
\item  Restriction maps:  If $U \subseteq X$ is an open subset of a complex manifold, then inclusion
$j: U \hookrightarrow X$ induces pullback maps $j^*: H^i(X) \rightarrow H^i(U)$, and hence 
restriction maps $j^!: \overline{H}_{2d-i}(X) \rightarrow \overline{H}_{2d-i}(U)$ by Poincar\'e duality.
The restriction maps preserve fundamental classes in the sense that $j^*([X]) = [U]$.
\item  If $\iota: Z \hookrightarrow X$ is the inclusion of a closed subset, and $j: X - Z \hookrightarrow X$ is the inclusion
of its complement, there is a long exact sequence
\begin{equation*}
\cdots \rightarrow \bar{H}_{k+1}(X-Z) \rightarrow \bar{H}_k(Z) \xrightarrow{\iota_*} \bar{H}_k(X)
\xrightarrow{j^*} \bar{H}_k(X-Z) \rightarrow \bar{H}_{k-1}(Z) \rightarrow \cdots
\end{equation*}
\end{enumerate}

\begin{defn}
\label{cellular-decompositions-definition}
Let $X$ be a complex algebraic variety.  A {\em cellular decomposition} of $X$ is a filtration 
$X_{\bullet} = (X = X_0 \supset X_1 \supset \cdots \supset X_m = \varnothing)$ where each $X_i$ is a closed subvariety of
$X$ for $0 \leq i \leq m$ and $X_i - X_{i+1}$ is isomorphic
as an algebraic variety
to a nonempty disjoint
union of complex affine spaces.  
\end{defn}

If $X_{\bullet} = (X = X_0 \supset X_1 \supset \cdots \supset X_m = \varnothing)$ is a cellular decomposition
of $X$, we write each difference $X_i - X_{i+1}$ as a disjoint union $X_i - X_{i+1} = \bigsqcup_j A_{ij}$ of affine
spaces.  The $A_{ij}$ are called the {\em cells} of the decomposition $X_{\bullet}$.

The stratification of $X$ given in Definition~\ref{cellular-decompositions-definition}
is also called a `paving by affines' in the literature.

\begin{lemma}
\label{single-cell-lemma}
Let $X$ be a complex algebraic variety.  If $X$ admits a cellular decomposition $X_{\bullet}$, then $X$ admits
a cellular decomposition $X'_{\bullet}$ with the same cells as $X_{\bullet}$ in which each $X'_i - X'_{i+1}$ is a single cell.
\end{lemma}

\begin{proof}
We obtain $X'_{\bullet}$ from $X_{\bullet}$ by replacing each inclusion $X_{i+1} \subset X_i$ with associated cells
$X_i - X_{i+1} = \bigsqcup_{j = 1}^p A_{ij}$ by the chain
\begin{equation}
\label{single-cell-chain}
X_{i+1} \subseteq X_{i+1} \cup \overline{A}_{i1} \subseteq \cdots \subseteq 
X_{i+1} \cup \overline{A}_{i1} \cup \cdots \cup \overline{A}_{ip} = X_i,
\end{equation}
where the closures $\overline{A}_{ij}$ are taken in the Zariski topology.

We want to show that the cells associated to the chain \eqref{single-cell-chain} are still $A_{i1}, \dots, A_{ip}$.  Indeed, since $X_i - X_{i+1} = \bigsqcup_{j = 1}^p A_{ij}$
as varieties, if we compute Zariski closures in $X_i - X_{i+1}$ 
we have $\overline{A}_{ij} \cap \overline{A}_{ij'} = \varnothing$ whenever $j \neq j'$
and $\overline{A}_{ij} = A_{ij}$.
If we instead compute closures in the larger space $X_i$ (or in $X$), 
we see that $(\overline{A}_{ij} - A_{ij}) \subseteq X_{i+1}$  for all $j$
and $\overline{A}_{ij} \cap \overline{A}_{ij'} \subseteq X_{i+1}$  for all $j \neq j'$.
Computing closures inside $X_i$, it follows that 
\begin{equation}
(X_{i+1} \cup \overline{A}_{i1} \cup \cdots \cup \overline{A}_{i(j-1)} 
\cup \overline{A}_{ij}) -
(X_{i+1} \cup \overline{A}_{i1} \cup \cdots \cup \overline{A}_{i(j-1)}) = A_{ij}
\end{equation}
for $j = 1, 2, \dots, p$.
\end{proof}

Cellular decompositions may be used to compute the Borel-Moore homology of a variety $X$.
The following result may be found in \cite[19.1.11]{FultonIntersection}.

\begin{lemma}
\label{homology-from-cellular}
If a variety $X$ admits a cellular decomposition $X_{\bullet}$, the Borel-Moore homology $\bar{H}_{\bullet}(X)$ is free 
abelian, and the classes $[\overline{A}] \in \bar{H}_{2 \dim(A)}(X)$ 
where $A$ runs over the cells of $X_{\bullet}$ form a basis.  Here $\dim(A)$ is complex dimension of $A$.
\end{lemma}

Our fundamental tool for computing cohomology is as follows.

\begin{theorem}
\label{cohomology-presentation-theorem}
Suppose $X$ is a smooth complex algebraic variety which admits a cellular decomposition, and $Z \subseteq X$ is a 
union of cell closures.  Let $I$ be the ideal in $H^{\bullet}(X)$ generated by the classes $[\overline{A}]$ for all
cells $A \subseteq Z$.  The inclusion $X - Z \hookrightarrow X$ induces an isomorphism
$H^{\bullet}(X - Z) \cong H^{\bullet}(X)/I$ of  cohomology rings.
\end{theorem}

\begin{proof}
By Lemma~\ref{single-cell-lemma}, we may assume that our cellular decomposition 
$X_{\bullet} = (X_0 \supset  X_1 \supset \cdots  \supset X_m)$ is such that each $X_i - X_{i+1}$ is a single cell.
The cell closures in this case are precisely  $X_0, X_1, \dots, X_m$, so that $Z = X_i$ for some $i$.

The variety $X - Z = X - X_i$ admits a cellular decomposition, namely
\begin{equation*}
X - Z = X_0 - Z \supset X_1 - Z \supset \cdots \supset X_{i-1} - Z.
\end{equation*}
By Lemma~\ref{homology-from-cellular}, we have $\bar{H}_k(X-Z) = 0$ when $k$ is odd, so for each even $k$ 
we have a short exact sequence
\begin{equation}
\label{cohomology-presentation-sequence}
0 \rightarrow \bar{H}_k(Z) \rightarrow \bar{H}_k(X) \rightarrow \bar{H}_k(X - Z) \rightarrow 0.
\end{equation}

Both $X$ and $X-Z$ are complex manifolds, say of complex dimension $d$.
Thus, the surjection $\bar{H}_{\bullet}(X) \twoheadrightarrow \bar{H}_{\bullet}(X-Z)$ corresponds under Poincar\'e duality
to the ring homomorphism $H^{2d - \bullet}(X) \twoheadrightarrow \bar{H}^{2d - \bullet}(X - Z)$
induced by the inclusion $X - Z \hookrightarrow X$.
By the exactness of \eqref{cohomology-presentation-sequence}, the kernel of 
$\bar{H}_{\bullet}(X) \twoheadrightarrow \bar{H}_{\bullet}(X-Z)$ is generated by the image of  
$\bar{H}_{\bullet}(Z)$ in $\bar{H}_{\bullet}(X)$, which by Lemma~\ref{homology-from-cellular} is generated by 
the classes of the cell closures lying in $Z$.
\end{proof}

\section{The cohomology of $X_{n,k}$ via cellular decompositions}
\label{Cohomology}

Our fundamental example of a cellular decomposition is the decomposition of 
complex projective space $\PP^{k-1}$ given by 
\begin{equation}
\label{projective-cellular}
\PP^{k-1} = \{ [\star : \star : \cdots : \star ] \} \supset \{ [0 : \star : \cdots : \star] \} \supset \cdots \supset
 \{ [0 : 0 : \cdots : \star] \} \supset \varnothing.
\end{equation}
More precisely, we have the cellular decomposition $X_0 \supset X_1 \supset \cdots \supset X_{k-1} \supset \varnothing$
where the points of
$X_i$ are represented by the nonzero row vectors $v = (v_1, \dots, v_k) \in \CC^k$ with $v_1 = \cdots = v_i = 0$.
The cell $A_i = X_i - X_{i+1} \cong \CC^{k-i-1}$ is represented by row vectors
$v = (v_1, \dots, v_k) \in \CC^k$ with $v_1 = \cdots = v_i = 0$ and $v_{i+1} \neq 0$.
One checks by taking intersections that we have the cup product
$[\overline{A}_i] \smile [\overline{A}_j] = [\overline{A}_{i+j}]$, with
$[\overline{A}_{i+1}] = 0$ if $i + j \geq k$.  We therefore have
$H^{\bullet}(\PP^{k-1}) = \ZZ[x]/\langle x^k \rangle$, where $x  = [\overline{A}_1]$.

The cellular decomposition of $\PP^{k-1}$ in the last paragraph leads  to a cellular decomposition
of the $n$-fold product $(\PP^{k-1})^n$.  The cells of this decomposition are labeled by length $n$ words
$w = w_1 \dots w_n$ in the alphabet $\{0, 1, \dots, k-1 \}$.  Given a word $w$, points in the corresponding cell
$A_w \subseteq (\PP^{k-1})^n$ are represented by $k \times n$ matrices $m = (m_{i,j})$ with
$m_{i,j} = 0$ whenever $i < w_j$ and $m_{w_j, j} \neq 0$.  This gives rise to the presentation 
\begin{equation}
\label{classical-projective-cohomology-presentation}
H^{\bullet}((\PP^{k-1})^n) = \ZZ[x_1, \dots, x_n]/ \langle x_1^k, \dots, x_n^k \rangle,
\end{equation}
where for any word $w = w_1 \dots w_n$ the corresponding cell closure $\overline{A}_w$ is represented by
$x_1^{w_1} \cdots x_n^{w_n}$.

Our goal is to calculate $H^{\bullet}(X_{n,k})$ by means of Theorem~\ref{cohomology-presentation-theorem}.
To do this, we will use a different cellular decomposition of $(\PP^{k-1})^n$, this time indexed by length
$n$ words $w = w_1 \dots w_n$ in the alphabet $\{1, 2, \dots, k\}$. We give a detailed description of this decomposition below, but briefly, the cells will be defined as follows. Given an $n$-tuple of lines $\ell_{\bullet} = (\ell_1, \ldots, \ell_n) \in (\PP^{k-1})^n$, set 
$F_j = \ell_1 + \cdots + \ell_j$, and let $\{j_1 < \cdots < j_m\}$ be the set of indices $j$ such that $F_{j-1} \neq F_j$. Then $F_{j_1} \subseteq \cdots \subseteq F_{j_m} = \CC^k$ is a partial flag in $\CC^k$, and it belongs to a unique Schubert cell in the appropriate partial flag variety, indexed by a word $\pi \in [k]^m$ with distinct letters. For each $j \in [n]$ there is a unique $i \in [m]$ such that $\ell_j \subseteq F_{j_i}$ but $\ell_j \not\subseteq F_{j_{i-1}}$, and we set $w_j = \pi_{j_i}$. This word $w$ will index the cell containing $\ell_{\bullet}$.

If $m = (m_{i,j}) \in \Mat_{k \times n}$ is any $k \times n$ matrix, the {\em rank function}
\begin{equation}
r(m): [k] \times [n] \rightarrow \ZZ_{\geq 0}
\end{equation}
is defined by
\begin{equation}
r(m): (i,j) \mapsto \text{rank of the submatrix of $m$ with row set $[i]$ and column set $[j]$}.
\end{equation}
Since $T$ acts on $\Mat_{k \times n}$ by nonzero column scaling and since $U \subset GL_k$ is the group
of {\em lower} triangular unipotent matrices, there holds the rank function equality
$r(umt) = r(m)$ for all matrices $u \in U, m \in \Mat_{k \times n},$ and $t = T$.

If $w = w_1 \dots w_n$ is any word in the alphabet $[k]$, we may define $r(w): [k] \times [n] \rightarrow \ZZ_{\geq 0}$
by identifying $w$ with the $k \times n$ matrix which has $1$'s in positions $(w_j, j)$ for $1 \leq j \leq n$ and 
$0$'s elsewhere.  If 
$\ell_{\bullet} = (\ell_1, \dots, \ell_n) \in (\PP^{k-1})^n$ is any $n$-tuple of lines in $\CC^k$, we can think of 
$\ell_{\bullet}$ as a point in the orbit space $\UUU_{n,k}/T$, where
\begin{equation}
\UUU_{n,k} = \{ m \in \Mat_{k \times n} \,:\, \text{$m$ has no zero columns} \}.
\end{equation}
Since $r$ is invariant under the action of $T$, we get a well-defined rank function
$r(\ell_{\bullet}): [k] \times [n] \rightarrow \ZZ_{\geq 0}$.
We stratify the space $(\PP^{k-1})^n$ into pieces $\Omega_w$ as follows.

\begin{defn}
\label{omega-definition}
Let $w = w_1 \dots w_n$ be a length $n$ word in the alphabet $[k]$.  Define $\Omega_w \subseteq (\PP^{k-1})^n$ by
\begin{equation*}
\Omega_w := \{ \ell_{\bullet} \in (\PP^{k-1})^n \,:\, r(\ell_{\bullet}) = r(w) \}.
\end{equation*}
\end{defn}

It is clear that $(\PP^{k-1})^n = \bigcup_{w \in [k]^n} \Omega_w$.  
In fact, the disjoint union here need only run over the {\em convex} words $w \in [k]^n$.

\begin{lemma}
\label{convex-disjoint-union}
$(\PP^{k-1})^n$ is the disjoint union of the $\Omega_w$ over all convex words $w \in [k]^n$.
\end{lemma}

\begin{proof}
This is essentially due to Fulton \cite[Lem. 3.1]{FultonFlag}.  More precisely, Fulton shows that 
the possible rank functions of $k \times n$ matrices are the rank functions of $0,1$-matrices with at most one $1$
in any row and column; that is, rank functions of words $v = v_1 \dots v_n$ in $\{0, 1, \dots, k\}^n$ where 
$v_i = 0$ indicates that the $i^{th}$ column of the corresponding matrix is all $0$'s.  If $v$ is any such word
with $v_1 \neq 0$, and $w \in [k]^n$ is the convex word obtained by replacing each $0$ with the letter
immediately to its left, then $r(w) = r(v)$.  On the other hand, if $v_1 = 0$, no point $\ell_{\bullet} \in (\PP^{k-1})^n$
has rank function $r(v)$.  This shows that the rank functions of points $\ell_{\bullet} \in (\PP^{k-1})^n$ 
are precisely the rank functions of convex words in $[k]^n$.
\end{proof}

The closure of $\Omega_w$ for any convex word $w \in [k]^n$ has the following description.

\begin{lemma}
\label{omega-closure}
Let $w \in [k]^n$ be a convex word.  The closure $\overline{\Omega}_w$ is equal to a  union of some $\Omega_v$
for various convex words $v \in [k]^n$.  In fact, we have
\begin{equation*}
\overline{\Omega}_w = \bigcup_{\substack{v \in [k]^n \text{convex} \\ r(v) \leq r(w)}} \Omega_v,
\end{equation*}
where we interpret the functional inequality $r(v) \leq r(w)$ pointwise on the domain $[k] \times [n]$.
\end{lemma}

\begin{proof}
Define $R_w := \{ \ell_{\bullet} \in (\PP^{k-1})^n \,:\, r(\ell_{\bullet}) \leq r(w) \}$. By definition
\begin{equation}
R_w = \bigcup_{\substack{v \in [k]^n \\ r(v) \leq r(w)}} \Omega_v.
\end{equation}
We claim that 
\begin{itemize}
\item
we may replace this union with the (disjoint by Lemma~\ref{convex-disjoint-union}) union over all 
{\em convex} $v \in [k]^n$ with $r(v) \leq r(w)$, and
\item we have $R_w = \overline{\Omega}_w$. 
\end{itemize}
By the proof of Lemma~\ref{convex-disjoint-union}, given any word $v' \in [k]^n$ there is a convex word $v \in [k]^n$ with $r(v') = r(v)$. Hence $\Omega_{v'} = \Omega_v$, and we can assume that the union in the definition of $R_w$ runs over only the convex $v$ with $r(v) \leq r(w)$.

For $\ell_{\bullet} \in (\PP^{k-1})^n$, the condition that $r(\ell_{\bullet})(i,j) \leq p$ for a fixed $p$ is equivalent to the vanishing of certain minors of any matrix representing $\ell_{\bullet}$, so it is a closed condition: that is, if $r(\cdot)(i,j) \leq p$ on some $A \subseteq (\PP^{k-1})^n$, the same is true on $\overline{A}$. In particular, $\overline{\Omega}_w \subseteq R_w$. It therefore suffices to show that $\dim \Omega_w = \dim R_w$ and that $R_w$ is irreducible.

Let $\Omega_w^M = \{m \in \Mat_{k \times n} : r(m) = r(w)\}$ and $R_w^M = \{m \in \Mat_{k \times n} : r(m) \leq r(w)\}$. Fulton proves \cite[Prop. 3.3]{FultonFlag} that $\overline{\Omega}_w^{M} = R_w^M$, and that $R_w^M$ is irreducible. Since
\begin{equation*}
\Omega_w = (\Omega_w^M \cap \UUU_{n,k})/T \quad \text{and} \quad R_w = (R_w^M \cap \UUU_{n,k})/T,
\end{equation*}
and $\dim(\Omega_w^M \cap \UUU_{n,k}) = \dim(R_w^M \cap \UUU_{n,k})$ by Fulton's result, we see that $\dim \Omega_w = \dim R_w$. Also, $R_w^M \cap \UUU_{n,k}$ is irreducible because $\UUU_{n,k} \subseteq \Mat_{k \times n}$ is open, so its image $R_w$ in $(\PP^{k-1})^n$ is also irreducible.
\end{proof}

Let $w = w_1 \dots w_n \in [k]^n$ be a convex word with initial set $\initial(w) = \{j_1 < \cdots < j_m \}$.
We view $\Omega_w$ explicitly as follows.
The {\em omega pattern matrix} $\OPM(w)$ is the $k \times n$ matrix with entries in
$\{0, \star, \diamond\}$ given by the following process.
\begin{enumerate}
\item  Set $\OPM(w)_{w_{j_1}, j_1} = \cdots = \OPM(w)_{w_{j_m}, j_m} = \diamond$.
\item  For $1 \leq r \leq m$, in the column $j_r$ place a $\star$ in every position of $\OPM(w)$
above the $\diamond$
in this column which is to the right of a pre-existing $\diamond$.  Place $0$'s elsewhere in this column.
\item  For a noninitial column $j \in [n] - \initial(w)$ of $\OPM(w)$, place a $\diamond$ in rows 
$w_1, w_2, \dots, w_j$ and $0$'s elsewhere.
\end{enumerate}
As an example, suppose $(n,k) = (8,5)$ and $w = 44111533$.  The initial set 
is $\initial(w) = \{1,3,6,7\}$ and the omega pattern matrix $\OPM(w)$ is 
\begin{equation*}
\OPM(w) = \OPM(44111533) =
\begin{pmatrix}
0 & 0 & \diamond & \diamond & \diamond & \star & \star & \diamond \\
0 & 0 & 0 & 0 & 0 & 0 & 0 & 0 \\
0 & 0 & 0 & 0 & 0 & 0 & \diamond & \diamond \\
\diamond & \diamond & 0 & \diamond & \diamond  & \star & 0 & \diamond \\
0 & 0 & 0 & 0 & 0 & \diamond & 0 & \diamond
\end{pmatrix}.
\end{equation*}
Let $\widehat{P}_w \subseteq \UUU_{n,k}$ be the set of all matrices obtained by replacing the 
$\star$'s and $\diamond$'s in $\OPM(w)$ by complex numbers in such a way that, in any column, not all
of the $\diamond$'s in that column are assigned to $0$.
The set $\widehat{P}_w$ of matrices is closed under the column scaling action of $T$, so the inclusion
$\widehat{P}_w/T \subseteq (\PP^{k-1})^n$ makes sense.

Let $w \in [k]^n$ be a word.  We define a further subgroup $U(w) \subseteq U$ of the lower triangular
unipotent subgroup
$U \subseteq GL_k$ by 
\begin{equation}
U(w) := \{ (u_{i,j}) \in U \,:\, \text{$u_{i,j} = 0$ if $i > j$ and the letter $j$ does not appear in $w$} \}.
\end{equation}
If $(n,k) = (8,5)$ and $w = 44111533$ is as above, the subgroup $U(w) \subset GL_5$ is
\begin{equation*}
U(w) = U(44111533) = \left\{ \begin{pmatrix}
1 & 0 & 0 & 0 & 0 \\
\star & 1 & 0 & 0 & 0 \\
\star & 0 & 1 & 0 & 0 \\
\star & 0 & \star & 1 & 0 \\
\star & 0 & \star & \star & 1
\end{pmatrix} \right\}.
\end{equation*}
The proof of the following lemma is very similar to that of Lemma~\ref{action-is-free} and is omitted.

\begin{lemma}
\label{omega-action-is-free}
Let $w = w_1 \dots w_n \in [k]^n$ be a convex word and let $J = \{w_1, \dots, w_n \}$ be the letters appearing in $w$.
Let $\MMM(w) \subseteq \Mat_{k \times n}$ be the set of $k \times n$ matrices $m$ such that the restriction of $m$
to the row set $J$ has full rank and $m$ has no zero columns.  Then $\MMM(w)$ carries an action of $U(w) \times T$
given by $(u,t).m := umt$ and this action is free.
\end{lemma}

We have a matrix multiplication map $U(w) \times \widehat{P}_w/T \rightarrow \Omega_w$ given by $(u,mT) \mapsto umT$.
The analog of Proposition~\ref{pattern-matrix-representatives} for the spaces $\Omega_w$ reads as follows;
we again omit the similar Gaussian elimination style proof.

\begin{lemma}
\label{omega-pattern-matrix-representatives}
Let $w \in [k]^n$ be a convex word and let $mT \in \Omega_w$ for some $k \times n$ matrix $m$.  
There exists a unique coset $m'T \in \widehat{P}_w/T$
and a unique $u \in U(w)$ such that $mT = um'T$.  Consequently, the map
$U(w) \times \widehat{P}_w/T \rightarrow \Omega_w$ is a regular bijection.
\end{lemma}

We want to show that the regular bijection
$U(w) \times \widehat{P}_w/T \rightarrow \Omega_w$ of Lemma~\ref{omega-pattern-matrix-representatives}
is fact an isomorphism of varieties, i.e. has regular inverse.
The space $\widehat{P}_w/T$ is visibly isomorphic 
to a product of affine and projective spaces; in our running example
if $w = 44111533$ and $(n,k) = (8,5)$ 
then reading the columns of $\OPM(w)$ from left to right gives
\begin{equation*}
\widehat{P}_w/T \cong \PP^0 \times \PP^0 \times \PP^0 \times \PP^1 \times \PP^1 \times \CC^2 \times \CC^1\times \PP^3.
\end{equation*}
In general, a column of 
$\OPM(w)$ with $d$ $\diamond$'s and no $\star$'s contributes a factor of $\PP^{d-1}$ to $\widehat{P}_w/T$
and a column with 
a single $\diamond$ and
$s > 0$ $\star$'s contributes a factor of $\CC^s$.  If we can show $U(w) \times \widehat{P}_w/T \rightarrow \Omega_w$ is 
an isomorphism, this will imply the isomorphism
\begin{equation}
\Omega_w \cong U(w) \times \widehat{P}_w/T \cong 
\CC^{\text{number of $\star$'s in $U(w)$}} \times \text{a product of affine and projective spaces},
\end{equation}
so that $\Omega_w$ is again a product of affine and projective spaces.

Recall that we are constructing a cellular decomposition of $(\PP^{k-1})^n$ whose cells are indexed by words
in $v = v_1 \dots v_n \in [k]^n$.
The cells $C_v$ of our decomposition will partition the spaces $\Omega_w \subseteq (\PP^{k-1})^n$; we describe
these cells as follows.  The first step is to associate a pattern matrix $\PM(v)$ to $v$ when $v$ is not necessarily Fubini.

Let $v = v_1 \dots v_n \in [k]^n$ be an arbitrary word with initial positions $\initial(v)$.  
The {\em pattern matrix} $\PM(v)$ is the $k \times n$
matrix with entries in $\{0, 1, \star\}$ defined as follows.
\begin{itemize}
\item  We have $\PM(v)_{i,j} = 1$ if and only if $v_j = i$.
\item Suppose $j \in \initial(v)$ and $w_j \neq i$.  If $v_j > i$ and there exists $j' < j$ such that 
$v_{j'} = i$ then $\PM(v)_{i,j} = \star$.  Otherwise $\PM(v)_{i,j} = 0$.
\item  Suppose $j \notin \initial(v)$ and $w_j \neq i$.  
If the first occurrence of $i$
in $v = v_1 \dots v_n$ is before the first 
occurrence of $w_j$ in $v = v_1 \dots v_n$ then $\PM(v)_{i,j} = \star$.  
Otherwise $\PM(v)_{i,j} = 0$.
\end{itemize}

If $v \in \WW_{n,k}$ is Fubini, the pattern matrix $\PM(v)$ agrees with the definition given in Section~\ref{Orbit}.
These more general pattern matrices $\PM(v)$ have zero rows corresponding to letters $1 \leq i \leq n$
not appearing in $v = v_1 \dots v_n$.

Let $v = v_1 \dots v_n \in [k]^n$.  Let $\widehat{C}_v \subseteq \UUU_{n,k}$ be the set of matrices obtained by replacing
the $\star$'s in $\PM(v)$ by arbitrary complex numbers.  We have 
$\widehat{C}_v \cong \CC^{\text{number of $\star$'s in $\PM(v)$}}$.  
Define $C_v \subseteq (\PP^{k-1})^n$ by
\begin{equation}
C_v := \text{image of $U(v) \widehat{C}_v$ in $(\PP^{k-1})^n$}.
\end{equation}
Said differently, we have $C_v = U(v) \widehat{C}_vT/T$.
The space $C_v$ is  a cell.

\begin{lemma}
\label{word-cells-are-affine}
Let $v = v_1 \dots v_n \in [k]^n$ be a word.  The subspace $C_v \subseteq (\PP^{k-1})^n$ is 
isomorphic to an
affine space. 
\end{lemma}

\begin{proof}
Matrix multiplication gives a bijective regular map 
$\mu: U(v) \times \widehat{C}_v T \rightarrow U \widehat{C}_v T$. 
If we could show that $\mu$ is an isomorphism, i.e. that $\mu^{-1}$ is also regular, the
isomorphism $U(v) \times \widehat{C}_v T \cong U \widehat{C}_v T$ would induce 
isomorphisms
\begin{equation*}
U(v) \times \widehat{C}_v  \cong U(v) \times \widehat{C}_v T/T \cong U(v) \widehat{C}_v T/T \cong C_v.
\end{equation*}
Since both $U(v)$ and $\widehat{C}_v$ are affine spaces, this would imply that $C_v$ is an affine space.

Given $(u,a) \in U(v) \times \widehat{C}_v T$ such that $ua = b \in U \widehat{C}_v T$, we need to verify that
$(u,a)$ is determined by some rational function in the entries of $b$.  
This amounts to expressing the entries of $u$ and $a$ in terms of 
judiciously chosen minors of the matrix $b$.    
The somewhat tedious details are similar to a corresponding result
for the space $\Omega_w$ proven in Lemma~\ref{omega-structural-result} below; 
to avoid redundancy, we will 
only include that argument. 
\end{proof}

We have the following structural result relating the spaces $\Omega_w$ and the cells $C_v$.

\begin{lemma}
\label{omega-structural-result}
Let $w = w_1 \dots w_n \in [k]^n$ be a convex word and consider the subspace $\Omega_w \subseteq (\PP^{k-1})^n$.
\begin{enumerate}
\item  $\Omega_w$ is isomorphic to a product of affine and projective spaces.
\item  $\Omega_w$ has a disjoint union decomposition $\Omega_w = \bigsqcup_v C_v$, where $v$ runs
over all words $v \in [k]^n$ such that $\initial(v) = \initial(w)$ and $v_i = w_i$ for all $i \in \initial(w)$.
\item $\Omega_w$ admits a cellular decomposition whose cells are precisely those $C_v$ appearing in the disjoint union of 
(2).
\item  The $n$-fold product $(\PP^{k-1})^n$ admits a cellular decomposition whose cells are $\{ C_v \,:\, v \in [k]^n \}$.
\end{enumerate}
\end{lemma}

\begin{proof}
(1)  Let $\initial(w) = \{ i_1 < \cdots < i_m \}$, so that $w$ has $m$ distinct letters $w_{i_1}, \dots, w_{i_m}$.
The matrix multiplication map $\mu: U(w) \times \widehat{P}_w   \rightarrow U(w) \widehat{P}_w $ is a regular bijection,
as noted in Lemma~\ref{omega-pattern-matrix-representatives}.  We show that the inverse of $\mu$ is regular,
so that $\mu$ is an isomorphism.  From this it follows that $\mu$ induces an isomorphism
$U(w) \times \widehat{P}_w/T \cong U(w) \widehat{P}_w/T = \Omega_w$.  Since $U(w)$ is an
 affine space and $\widehat{P}_w/T$ is a product of affine and projective spaces, this  completes the proof.

To prove that $\mu^{-1}: U(w) \widehat{P}_w \rightarrow U(w) \times \widehat{P}_w$ is regular, suppose
$(u,a) \in U(w) \times \widehat{P}_w$ satisfy $ua = b$.  We need to show that the entries
of the
pair $(u,a)$ of matrices are given by rational functions of the entries of the matrix $b$.
This will be proven in several stages.

The argument establishing the regularity 
of $\mu^{-1}$ is most easily followed using a running example.  We take
$(n,k) = (8,5)$ and $w = 44111533$, so that $\initial(w) = \{1,3,6,7\}$ and $(w_1, w_3, w_5, w_7) = (4,1,5,3)$.  The matrices
$u$ and $a$ have the form
\begin{equation*} u =
\begin{pmatrix}
1 & 0 & 0 & 0 & 0 \\
u_{2,1} & 1 & 0 & 0 & 0 \\
u_{3,1} & 0 & 1 & 0 & 0 \\
u_{4,1} & 0 & u_{4,3} & 1 & 0 \\
u_{5,1} & 0 & u_{5,3}  & u_{5,4} & 1 
\end{pmatrix}, \quad
a = 
\begin{pmatrix}
0 & 0 & a_{1,3} & a_{1,4} & a_{1,5} & a_{1,6} & a_{1,7} &  a_{1,8} \\
0 & 0 & 0 & 0 & 0 & 0 & 0 & 0 \\
0 & 0 & 0 & 0 & 0 & 0 & a_{3,7} & a_{3,8} \\
a_{4,1} & a_{4,2} & 0 & a_{4,4} & a_{4,5} & a_{4,6} & 0 & a_{4,8} \\
0 & 0 & 0 & 0 & 0 & a_{5,6} & 0 & a_{5,8}
\end{pmatrix}.
\end{equation*}
Here the $u_{i,j}$ and $a_{i,j}$ are complex numbers with the property that $a_{3,7}, a_{5,6} \neq 0$ and 
the matrix $a$ has no zero columns.

The equation $ua = b$ is an equation of $k \times n$ matrices.  Let $u'$ be the restriction of $u$ to the row and column
sets $\{i_1, \dots, i_m\}$ and let $a'$ be the restriction of $a$ to the 
column set $\{i_1, \dots, i_m\}$ and the row
set $\{ w_{i_1}, \dots, w_{i_m} \}$.  In our example, we have $m = 4$ and
\begin{equation*}
u' = \begin{pmatrix}
1 & 0 & 0 & 0 \\
u_{3,1} & 1 & 0 & 0 \\
u_{4,1} & u_{4,3} & 1 & 0 \\
u_{5,1} & u_{5,3} & u_{5,4} & 1
\end{pmatrix}, \quad
a' = \begin{pmatrix}
0 & a_{1,3} & a_{1,6} & a_{1,7} \\
0 & 0 & 0 & a_{3,7} \\
a_{4,1} & 0 & a_{4,6} & 0 \\
0 & 0 & a_{5,6} & 0
\end{pmatrix}.
\end{equation*}
Let $b'$ be the restriction of $b$ to the column
 set $\{i_1, \dots, i_m\}$ and the row set $\{w_{i_1}, \dots, w_{i_m} \}$.
It is easy to see that we have the matrix equation $u' a' = b'$, where each of $u', a',$ and $b'$ are $m \times m$ invertible
matrices.

We prove that the nonzero entries in $a'$ are rational functions of the entries of $b'$ (and hence rational functions of the entries of $b$).
Let $\pi = \pi_1 \dots \pi_m \in S_m$ be the  permutation whose one-line notation is order isomorphic
to the sequence $w_{i_1} \dots w_{i_m}$.  In our example we have $w_{i_1} \dots w_{i_m} = 4153$ so that 
$\pi = 3142 \in S_4$.  For $r = 1, 2, \dots, m$ the following equation recursively gives the (necessarily nonzero)
entries $a'_{r, \pi_r^{-1}}$ of $a'$
as rational functions of the entries of $b'$:
\begin{align}
a'_{1,\pi^{-1}_1} a'_{2,\pi^{-1}_2} \cdots a'_{r, \pi^{-1}_r} &=
\epsilon_r \Delta_{[r], \{\pi^{-1}_1, \dots, \pi^{-1}_r \}}(a') \\
&= \epsilon_r \Delta_{[r], \{\pi^{-1}_1, \dots, \pi^{-1}_r \}}(u'a') \\
&= \epsilon_r \Delta_{[r], \{\pi^{-1}_1, \dots, \pi^{-1}_r \}}(b'), 
\end{align}
where the sign $\epsilon_r$ is the parity of the number of inversions in the sequence 
$(\pi^{-1}_1, \pi^{-1}_2, \dots, \pi^{-1}_r)$  and
we used the fact that $u'$ is unipotent.
In our example, these $m = 4$ equations read
\begin{equation*}
\begin{cases}
a_{1,3} = + \Delta_{ \{1\}, \{2\} } (b'), \\
a_{1,3} a_{3,7} = + \Delta_{ \{1,2\}, \{2,4\} } (b'), \\
a_{1,3} a_{3,7} a_{4,1} = + \Delta_{ \{1,2,3\}, \{1,2,4\}}(b'), \\
a_{1,3} a_{3,7} a_{4,1} a_{5,6} = - \Delta_{ \{1,2,3,4\}, \{1,2,3,4\} }(b') = - \det(b').
\end{cases}
\end{equation*}

The above paragraph establishes that the necessarily nonzero entries 
\begin{equation*}
(a'_{\pi_1,1}, \dots, a'_{\pi_m,m}) = (a_{w_{i_1}, i_1}, \dots, a_{w_{i_m}, i_m})
\end{equation*}
of $a'$
are rational functions of the entries of $b'$.
For the remaining potentially nonzero entries of $a'$, we use the following similar argument.  
Let $1 \leq r, s \leq m$ be such that $\pi_s > r$ so that $a'_{r,s}$ is a potentially nonzero entry of $a$.  Then  we have
\begin{align}
a'_{1,\pi^{-1}_1}  \cdots a'_{r-1, \pi^{-1}_{r-1}} a'_{r,s} &=
\epsilon'_{r,s} \Delta_{[r], \{\pi^{-1}_1, \dots, \pi^{-1}_{r-1}, s \}}(a') \\
&= \epsilon'_{r,s} \Delta_{[r], \{\pi^{-1}_1, \dots, \pi^{-1}_{r-1}, s \}}(u'a') \\
&= \epsilon'_{r,s} \Delta_{[r], \{\pi^{-1}_1, \dots, \pi^{-1}_{r-1}, s \}}(b'), 
\end{align}
where the sign $\epsilon'_{r,s}$ is the parity of the number of inversions in the sequence 
$(\pi^{-1}_1, \dots, \pi^{-1}_{r-1}, s)$. 
The last paragraph shows that each of the numbers 
$a'_{1,\pi^{-1}_1}, \dots, a'_{r-1,\pi^{-1}_{r-1}}$ is nonzero
and is a rational function of the entries of $b'$.
It follows that $a'_{r,s}$ is a rational function of the entries of $b'$.
In our example,
the above equations read
\begin{equation*}
\begin{cases}
a_{1,6} = + \Delta_{\{1\}, \{3\}}(b') \\
a_{1,7} = + \Delta_{\{1\}, \{4\}}(b') \\
a_{1,3} a_{3,7} a_{4,6} = - \Delta_{\{1,2,3\}, \{2,3,4\}} (b'),
\end{cases}
\end{equation*}
where the numbers $a_{1,3}, a_{3,7}$ are nonzero.

At this point, we know that all of the entries of the invertible $m \times m$ matrix $a'$
are rational functions of the entries of $b$.
The nonzero subdiagonal entries 
$\{ u_{i,j} \,:\, \text{$j$ appears in $w$} , j < i \}$ 
of $u$ may be seen to depend
rationally on $b$ as follows.  For $1 \leq j \leq k$, consider the length $m$ row vector 
$v_j = (a_{j, i_1}, \dots , a_{j,i_m})$ where $\initial(w) = \{i_1, \dots, i_m\}$.  In our example, we have
\begin{equation*}
v_1 = (0, a_{1,3}, a_{1,6}, a_{1,7}), v_2 = (0,0,0,0), v_3 = (0,0,0,a_{3,7}), v_4 = (a_{4,1}, 0, a_{4,6}, 0), v_5 = (0,0,a_{5,6},0).
\end{equation*}
The vectors $v_1, v_2, \dots, v_k$ are just the rows of the matrix $a'$, together with zero
vectors corresponding to the letters $1 \leq j \leq m$ which do not appear in $w$.
In particular,
the set $\{ v_j \,:\, j$ appears in $w \}$ is linearly independent and its $m$ vectors depend
rationally on $b$.

Let $1 \leq i \leq k$.
If $x_i = (b_{i, i_1}, \dots, b_{i,i_m})$ is the row vector obtained by restricting the $i^{th}$ row of $b$ to the columns
$\initial(w)$, then
\begin{equation}
\label{system-of-equations}
x_i = \sum_{j < i} u_{i,j} v_j = \sum_{\substack{j < i \\ \text{$j$ appears in $w$}}} u_{i,j} v_j,
\end{equation}
since $v_j = 0$ if $j \notin \initial(w)$.
There is a unique tuple of complex numbers
\begin{equation}
(u_{i,j} \,:\, j < i \text{ and $j$ appears in $w$})
\end{equation}
which solves the system \eqref{system-of-equations}.
By Cramer's Rule, the components $u_{i,j}$ of this solution depend 
rationally on the entries of $x_i$
and $\{ v_j \,:\, j < i \text{ and $j$ appears in $w$ } \}$, and hence rationally on $b$.
Therefore the entries of both $u$ and $a = u^{-1} b$ are rational functions of 
the entries of $b$.

(2)  It follows from the definition of the pattern matrices $\OPM(w)$ and $\PM(v)$ that we have a disjoint union
decomposition in matrix space
\begin{equation}
\label{matrix-disjoint-union-decomposition}
\widehat{P}_w = \bigsqcup_v \widehat{C}_v T,
\end{equation}
where the disjoint union runs over all words $v \in [k]^n$ such $\initial(v) = \initial(w)$ and 
$v_i = w_i$ for all $i \in \initial(w)$.
To see how this disjoint union decomposition works, consider the case $(n,k) = (6,4)$ and 
$w = 441122 \in [4]^6$.  Then $\OPM(w)$ is given by
\begin{equation*}
\OPM(441122) = \begin{pmatrix}
0 & 0 & \diamond & \diamond & \star & \diamond \\ 
0 & 0 & 0 & 0 & \diamond & \diamond \\
0 & 0 & 0 & 0 & 0 & 0 \\
\diamond & \diamond & 0 & \diamond & 0 & \diamond 
\end{pmatrix}
\end{equation*}
and $\widehat{P}_w$ is obtained by replacing the $\star$ and the $\diamond$'s with complex numbers in 
such a way that not all of the $\diamond$'s in a single column are assigned to zero.
The six words $v \in [4]^6$ with $\initial(v) = \{1,3,5\}$ and $v_1 = 4, v_3 = 1, v_5 = 2$ are 
$441122, 441422, 441121, 441421, 441124,$ and $441424$; their pattern matrices are 
\begin{align*}
&\PM(441122) = \begin{pmatrix}
0 & 0 & 1 & 1 & \star & \star \\
0 & 0 & 0 & 0 & 1 & 1 \\
0 & 0 & 0 & 0 & 0 & 0 \\
1 & 1 & 0 & \star & 0 & \star 
\end{pmatrix}, \quad
\PM(441422) = \begin{pmatrix}
0 & 0 & 1 & 0 & \star & \star \\
0 & 0 & 0 & 0 & 1 & 1 \\
0 & 0 & 0 & 0 & 0 & 0 \\
1 & 1 & 0 & 1 & 0 & \star 
\end{pmatrix}, \quad \\
&\PM(441121) = \begin{pmatrix}
0 & 0 & 1 & 1 & \star & 1 \\
0 & 0 & 0 & 0 & 1 & 0 \\
0 & 0 & 0 & 0 & 0 & 0 \\
1 & 1 & 0 & \star & 0 & \star 
\end{pmatrix}, \quad
\PM(441421) = \begin{pmatrix}
0 & 0 & 1 & 0 & \star & 1 \\
0 & 0 & 0 & 0 & 1 & 0 \\
0 & 0 & 0 & 0 & 0 & 0 \\
1 & 1 & 0 & 1 & 0 & \star 
\end{pmatrix}, \quad \\
&\PM(441124) = \begin{pmatrix}
0 & 0 & 1 & 1 & \star & 0 \\
0 & 0 & 0 & 0 & 1 & 0 \\
0 & 0 & 0 & 0 & 0 & 0 \\
1 & 1 & 0 & \star & 0 & 1 
\end{pmatrix}, \quad
\PM(441424) = \begin{pmatrix}
0 & 0 & 1 & 0 & \star & 0 \\
0 & 0 & 0 & 0 & 1 & 0 \\
0 & 0 & 0 & 0 & 0 & 0 \\
1 & 1 & 0 & 1 & 0 & 1 
\end{pmatrix} \quad \\
\end{align*}
and a typical matrix in one of the corresponding $\widehat{C}_v$ is obtained by replacing the $\star$'s with complex
numbers.
Since $T$ multiplies columns by nonzero scalars, the decomposition \eqref{matrix-disjoint-union-decomposition}
is evident.  Since $\Omega_w = U(w) \widehat{P}_w/T$ and $C_v =U(v) \widehat{C}_v T/T$ 
the result follows.

(3)  By (1), the space $\Omega_w$ is isomorphic to a product of affine and projective spaces.
By Lemma~\ref{word-cells-are-affine} the $C_v$ involved in the disjoint union decomposition
$\Omega_w = \bigsqcup_v C_v$ are isomorphic to affine spaces.
The result follows by observing that the disjoint union decomposition $\Omega_w = \bigsqcup_v C_v$
is (up to coordinate permutation) induced by the standard cellular decomposition of each projective
space coming from a column $1 \leq j \leq n$ of $\Omega_w$ with $j \notin \initial(w)$; see the above example.

(4)  List the convex words in $[k]^n$ as $w^1, w^2, \dots, w^p$ in such a way that if we have the pointwise
inequality of rank functions $r(w^i) \leq r(w^j)$ then $i \leq j$.  By Lemma~\ref{convex-disjoint-union} we know
$(\PP^{k-1})^n = \Omega_{w^1} \sqcup \cdots \sqcup \Omega_{w^p}$.  By Lemma~\ref{omega-closure}
the cellular decompositions of the $\Omega_{w^1}, \dots, \Omega_{w^p}$ described in (3) may be concatenated
to give a cellular decomposition of $(\PP^{k-1})^n$ whose cells are 
$\{ C_v \,:\, v \in [k]^n \}$.
\end{proof}

The space $X_{n,k}$ can be realized from $(\PP^{k-1})^n$ by removing those cell closures $\overline{C}_w$ for those
$w \in [k]^n$ such that $v \notin \WW_{n,k}$. 
In order to apply Theorem~\ref{cohomology-presentation-theorem} to present the cohomology of $X_{n,k}$ as a further
quotient of the presentation
$H^{\bullet}((\PP^{k-1})^n) \cong \ZZ[\xx_n]/\langle x_1^k, \dots, x_n^k \rangle$ given in 
\eqref{classical-projective-cohomology-presentation}, we need to understand the representatives in
$\ZZ[\xx_n]/\langle x_1^k, \dots, x_n^k \rangle$ of the classes $[\overline{C}_w]$  for arbitrary words $w \in [k]^n$.

In order to describe the classes $[\overline{C}_w]$, we need the following notation of 
{\em standardization}.  Let $w = w_1 \dots w_n \in [k]^n$ be a word with $m$ distinct letters.
Write the non-initial positions of $w$ as $[n] - \initial(w) = \{i_1 < \cdots < i_{n-m} \}$
and write $[k] - \{w_1, \dots, w_n \} = \{j_1 < \cdots < j_{k-m} \}$ for the letters not appearing in $w$.
The permutation
$\std(w) \in S_{n+k-m}$ is given by
\begin{equation}
\std(w)_i = \begin{cases}
w_i & \text{if $i \in \initial(w)$} \\
k + r & \text{if $i = i_r$ for some $1 \leq r \leq n-m$} \\
j_{i-n} & \text{if $i > n$.}
\end{cases}
\end{equation}
In other words, we replace the letters in noninitial positions of $w$ from left to right with
$k+1, k+2, \dots, k+(n-m)$, and then append the letters in $[k]$ which do not appear in $w$ 
to the end of $w$, in increasing order.
As an example, if $w = 215235 \in [5]^6$ (so that $m = 4$) then $\std(w) = 2156374 \in S_7$.

\begin{lemma} \label{lem:Cw-codimension}
    For a convex word $w$, it holds that $\overline{C}_w = \overline{\Omega}_w$, and $\codim \overline{C}_w = \inv(\std(w))$.
\end{lemma}
\begin{proof}
    Recall from Observation~\ref{dimension-alternative} that if $w$ is a Fubini word, then the number of $\star$'s in the pattern matrix $\PM(w)$ is
    \begin{equation*}
        \dim(w) = -\inv(\pi) - n + \sum_{i=1}^n \pi^{-1}_{w_i},
    \end{equation*}
    where $\pi$ is the permutation $w_{i_1} \cdots w_{i_k}$ and $\initial(w) = \{i_1 < \cdots < i_k\}$. For an arbitrary word $w \in [n]^k$ the word $\pi$ may no longer be a permutation, but the same formula works, interpreting $\pi^{-1}_{w_i}$ as the position of $w_i$ in $\pi$. Letting $m = |\initial(w)|$, we get
    \begin{equation} \label{eq:Cw-dimension}
        \dim C_w  = \dim U(w) + \dim(w) = \sum_{i=1}^m (k-\pi_i) - \inv(\pi) - n + \sum_{i=1}^n \pi^{-1}_{w_i}.
    \end{equation}

    By Lemma~\ref{omega-structural-result}(c), if $w$ is convex then $\Omega_w = \bigsqcup_{v} C_v$ where $v \in [k]^n$ runs over words such that $\initial(w) = \initial(v)$ and $v_i = w_i$ for all $i \in \initial(w)$. For such a word $v$, the formula \eqref{eq:Cw-dimension} makes clear that $\dim C_v \leq \dim C_w$ with equality if and only if $v = w$. Since $\Omega_w$ is irreducible by Lemma~\ref{omega-structural-result}(a), it follows that $\overline{\Omega}_w = \overline{C}_w$.

    Now let us see that $\dim C_w + \inv(\std(w)) = (k-1)n = \dim\,(\PP^{k-1})^n$ for convex $w$. Write $w = \pi_1^{a_1} \cdots \pi_m^{a_m}$, where $\pi_i^{a_i}$ is the word consisting of $a_i$ copies of $\pi_i$. On the one hand,
    \begin{equation*}
        \dim C_w = \sum_{i=1}^m (k-\pi_i) - \inv(\pi) - n + \sum_{i=1}^m i a_i.
    \end{equation*}
    On the other,
    \begin{equation*}
        \inv(\std(w)) = \inv(\pi) + \sum_{i=1}^m (\pi_i - i) + \sum_{i=1}^m (a_i-1)(k-i);
    \end{equation*}
 the three terms in this sum count, respectively, the number of inversions in $\std(w)$ which start and end in $\initial(w)$, which start in $\initial(w)$ and end in $\{n+1, \ldots, n+k-m\}$, and which start in $[n] - \initial(w)$. Thus
 \begin{align*}
    \dim C_w + \inv(\std(w)) &= -n + \sum_{i=1}^m \left[(k-\pi_i) + ia_i + (\pi_i - i) + (a_i-1)(k-i)\right]\\
    &= -n + \sum_{i=1}^m a_i k = n(k-1).
 \end{align*}
\end{proof}

Let $w \in [k]^n$ be a word with $m$ distinct letters.
By construction, the standardization $\std(w) \in S_{n+k-m}$ satisfies
$\std(w)_{n+1} < \std(w)_{n+2} < \cdots < \std(w)_{n+k-m}$.  This means that the Schubert polynomial
$\symm_{\std(w)} \in \ZZ[\xx_{n+k-m}]$ only involves the first $n$ variables $x_1, x_2, \dots, x_n$,
and so may be regarded as lying in $\ZZ[\xx_n]$.
We are ready to describe the classes $[ \overline{C}_w ] \in H^{\bullet}((\PP^{k-1})^n)$ 
for convex words $w \in [k]^n$.

\begin{lemma}
\label{convex-word-representatives}
Let $w \in [k]^n$ be a convex word.  
Under the presentation 
$H^{\bullet}((\PP^{k-1})^n) = \ZZ[\xx_n]/\langle x_1^k, \dots, x_n^k \rangle$ given in
\eqref{classical-projective-cohomology-presentation}, the class $[ \overline{C}_w]$ is represented by $\symm_{\std(w)}$.
\end{lemma}

\begin{proof} 
    Suppose $X$ is a smooth complex variety, equipped with flags of vector bundles $0 = E_0 \subseteq E_1 \subseteq \cdots \subseteq E_n \twoheadrightarrow X$ and $0 = F_0 \twoheadleftarrow F_1 \twoheadleftarrow \cdots \twoheadleftarrow F_k \twoheadrightarrow X$, where $\operatorname{rank} E_i = \operatorname{rank} F_i = i$. Let $f : E_n \to F_k$ be a vector bundle map. Given $p \in X$, define the rank function $r_p(f) : [k] \times [n] \to \mathbb{Z}_{\geq 0}$ of $f$ at $p$ by
    \begin{equation*}
        r_p(f) : (i,j) \mapsto \operatorname{rank}((E_j)_p \hookrightarrow (E_n)_p \xrightarrow{f} (F_k)_p \twoheadrightarrow (F_i)_p),
    \end{equation*}
    where $(E_j)_p$ denotes the fiber of $E_j$ over $p$. The associated \emph{degeneracy locus} of shape $w$ is then $\Upsilon = \{p \in X : r_p(f) \leq r(w) \}$. Define elements
    \begin{equation*}
        a_i = c_1(\ker(F_i \twoheadrightarrow F_{i-1})) \quad \text{and} \quad b_j = c_1(E_j/E_{j-1})
    \end{equation*}
    of $H^2(X)$. Fulton shows \cite[Theorem 8.2]{FultonFlag} that if $\Upsilon \subseteq X$ has the expected codimension $\inv(\std(w))$, then $[\Upsilon] \in H^{\bullet}(X)$ is the double Schubert polynomial $\symm_{\std(w)^{-1}}(a_1, a_2, \ldots; b_1, b_2, \ldots)$; here we take $a_i = b_j = 0$ if $i > n$ or $j > k$.

    In particular, let $\ell_1, \ldots, \ell_n$ be the tautological line bundles over $(\PP^{k-1})^n$. Set $E_j = \ell_1 \oplus \cdots \oplus \ell_j$ for $j \in [n]$, and for $i \in [k]$ let $F_i$ be the trivial bundle $\CC^i$, with the surjection $F_i \twoheadrightarrow F_{i-1}$ being projection onto the first $i-1$ coordinates. Take $f : E_n \to F_k$ to be the map sending $(v_1, \ldots, v_n) \in (E_n)_p$ to $v_1 + \cdots + v_n \in \CC^k$. Then $\operatorname{rank}((E_j)_p \hookrightarrow (E_n)_p \xrightarrow{f} (F_k)_p \twoheadrightarrow (F_i)_p)$ is exactly the rank of the upper left $i \times j$ submatrix of a matrix representing $p \in (\PP^{k-1})^n$, and so the associated degeneracy locus of shape $w$  is $\overline{\Omega}_w$, by Lemma~\ref{omega-closure}.

    Lemma~\ref{lem:Cw-codimension} shows that $\overline{C}_w = \overline{\Omega}_w$ has codimension $\inv(\std(w))$, so Fulton's formula applies. In the presentation $H^{\bullet}((\PP^{k-1})^n) = \ZZ[\xx_n]/\langle x_1^k, \ldots, x_n^k \rangle$, the variable $x_i$ is $c_1(\ell_i^*) = -c_1(\ell_i) = -b_i$, while $a_1 = \cdots = a_k = 0$ since the bundles $F_i$ are trivial. Thus,
    \begin{align*}
        [\overline{C}_w] &= \symm_{\std(w)^{-1}}(0, \ldots, 0; -x_1, \ldots, -x_n)\\
        &= \symm_{\std(w)}(x_1, \ldots, x_n; 0, \ldots, 0)\\
        &= \symm_{\std(w)}(x_1, \ldots, x_n).
    \end{align*}
\end{proof}

In order to describe $[ \overline{C}_w]$ for words $w \in [k]^n$ which may not be convex, we will need the 
convexification $\conv(w)$.  
The word $\conv(w)$ may be characterized as follows.
Let $S_n$ act on length $n$ words by $\pi.w = \pi.(w_1 \dots w_n) := w_{\pi_1} \cdots w_{\pi_n}$.
If $\initial(w) = \{ i_1 < \cdots < i_m \}$, define $\sigma(w) \in S_n$ in terms of one-line notation
$\sigma(w) = \sigma(w)_1 \dots \sigma(w)_n$ by listing the positions (in increasing order) of all instances
of $w_{i_1}$ in $w$, then listing the positions (in increasing order) of all instances of $w_{i_2}$ in $w$,
and so on.  For example, if $w = 215235 \in [5]^6$ then $\conv(w) = 221553$ so that
$\sigma(w) = 142365 \in S_6$.  

By construction, we have
$\sigma(w)^{-1}.w = \conv(w)$ for any word $w \in [k]^n$, or equivalently 
\begin{equation*}
w = \sigma(w).\conv(w).  
\end{equation*}
In fact, the permutation $\sigma(w) \in S_n$ may be characterized
as the unique minimal permutation in the weak Bruhat order on $S_n$ which sends $\conv(w)$ to $w$.
The following factorization property of $\sigma(w)$ is crucial.

\begin{observation}
\label{factorization-of-sigma}
Let $w \in [k]^n$ be a word and let $\conv(w)$ be its convexification.  The permutation $\sigma(w)$ has a factorization
$\sigma(w) = s_{i_1} \cdots s_{i_r}$ into adjacent transpositions such that for all $1 \leq j \leq r$, at most one of the indices 
$i_j$ and $i_{j}+1$ lie in the set of initial positions $\initial(s_{i_{j+1}} \cdots s_{i_r}.\conv(w))$ for the word
$s_{i_{j+1}} \cdots s_{i_r}.\conv(w)$.
\end{observation}

In other words, we can realize the transformation $\conv(w) \leadsto \sigma(w).\conv(w) = w$ 
as a sequence of adjacent transpositions,
none of which interchange two initial positions of the word at any stage of the process.

\begin{proposition}
\label{word-representatives}
Let $w \in [k]^n$ be an arbitrary word.  We have the equality of spaces
$\overline{C}_w = \overline{\Omega}_{\conv(w)} \sigma(w)$,
where we identify $\sigma(w)^{-1}$ with its $n \times n$ permutation matrix.  
Moreover, under the presentation 
$H^{\bullet}((\PP^{k-1})^n) = \ZZ[\xx_n]/\langle x_1^k, \dots, x_n^k \rangle$ given in
\eqref{classical-projective-cohomology-presentation}, 
the class $[ \overline{C}_w]$ is represented by the polynomial $\sigma(w)^{-1}.\symm_{\std(\conv(w))} \in \ZZ[\xx_n]$.
\end{proposition}

\begin{proof}
Lemma~\ref{convex-word-representatives} gives this result when $w$ is convex, or equivalently when $\sigma(w) = 1$ is
the identity permutation in $S_n$.  In general, we induct on the number of inversions in $\sigma(w)$. 

 If $v \in [k]^n$ is a
word $1 \leq i \leq n$ is such that both of $i$ and $i+1$ do not lie in the initial positions $\initial(v)$, the pattern matrix
(resp. omega pattern matrix) 
of $s_i.v$ is obtained by swapping columns $i$ and $i+1$ in the pattern matrix (resp. omega pattern matrix)
of $v$.  In symbols, we have
$\PM(s_i.v) = \PM(v)s_i$ and $\OPM(s_i.v) = \OPM(v)s_i$.  This leads to the equalities 
$C_{s_i.v} = C_v s_i$, $\Omega_{s_i.v} = \Omega_v s_i$ of subspaces of $(\PP^{k-1})^n$.  We also have the 
identification of cohomology classes 
\begin{equation*}
\begin{cases}
[\overline{C}_{s_i.v}] = [\overline{C}_v s_i] = s_i.[\overline{C}_v], \\
[\overline{\Omega}_{s_i.v}] = [\overline{\Omega}_v s_i] = s_i. [\overline{\Omega}_v].
\end{cases}
\end{equation*}
The result follows by induction on $\inv(\sigma(w))$.
\end{proof}

We are ready to describe the cohomology ring of $X_{n,k}$.  Recall that if $k \leq n$, the ring $R_{n,k}$ 
is the quotient of $\ZZ[\xx_n]$ given by
$R_{n,k} = \ZZ[\xx_n]/I_{n,k}$ where 
$I_{n,k} = \langle x_1^k, \dots, x_n^k, e_n(\xx_n), e_{n-1}(\xx_n), \dots, e_{n-k+1}(\xx_n) \rangle$.

\begin{theorem}
\label{cohomology-of-x}
Let $k \leq n$ be positive integers and consider the presentation of the cohomology ring
$H^{\bullet}((\PP^{k-1})^n) = \ZZ[\xx_n]/\langle x_1^k, \dots, x_n^k \rangle$ 
given in \eqref{classical-projective-cohomology-presentation}.  
\begin{enumerate}
\item  With respect to the cellular decomposition of $(\PP^{k-1})^n$ whose open cells 
are $\{C_w \,:\, w \in [k]^n \}$ given in Lemma~\ref{omega-structural-result} (4), the space $X_{n,k}$ may be expressed as a 
difference
\begin{equation}
X_{n,k} = 
(\PP^{k-1})^n - \bigcup_{\substack{w \in [k]^n \\ w \notin \WW_{n,k}}} \overline{C}_w,
\end{equation}
so that $X_{n,k}$ is obtained from $(\PP^{k-1})^n$ by removing a union of cell closures.
\item  The subvariety $X_{n,k}$ of $(\PP^{k-1})^n$ has a cellular decomposition whose
cells $\{ C_w \,:\, w \in \WW_{n,k} \}$ are indexed by length $n$ Fubini words with maximum letter $k$,
or equivalently $k$-block ordered set partitions of size $n$.
\item The cohomology presentation given by Theorem~\ref{cohomology-presentation-theorem} realizes 
$H^{\bullet}(X_{n,k})$ explicitly as the canonical quotient of $H^{\bullet}((\PP^{k-1})^n)$ given by
\begin{equation}
H^{\bullet}((\PP^{k-1})^n) = \ZZ[\xx_n]/\langle x_1^k, \dots, x_n^k \rangle \twoheadrightarrow R_{n,k} = H^{\bullet}(X_{n,k}).
\end{equation}
In particular, we have an isomorphism of graded rings $H^{\bullet}(X_{n,k}) \cong R_{n,k}$, where each variable
$x_i$ generating $R_{n,k}$ has $\deg(x_i) = 2$. Under this isomorphism, $x_i$ corresponds to the first Chern class of the $i$\textsuperscript{th} dual tautological line bundle over $X_{n,k}$.
\item  For any Fubini word $w \in \WW_{n,k}$, the cohomology class $[\overline{C}_w] \in H^{\bullet}(X_{n,k})$ of the 
closure of the cell $C_w$ of $X_w$ is represented by the polynomial
$\sigma(w)^{-1}.\symm_{\std(\conv(w))} \in \ZZ[\xx_n]$, where $\symm_{\std(\conv(w))} \in \ZZ[\xx_n]$ is the classical 
Schubert polynomial corresponding to the permutation $\std(\conv(w)) \in S_n$.
\end{enumerate}
\end{theorem}

\begin{proof}
(1)  We start with the inclusion $\subseteq$.
If $w \in [k]^n$ does not lie in $\WW_{n,k}$, the $k \times n$ $0,1$-matrix representing $w$ does not have full rank.
It follows that any matrix representing a point in $\overline{C}_w$ does not have full rank.  Since 
$X_{n,k} = \MMM_{n,k}/T$ and every matrix in $\MMM_{n,k}$ has full rank, we conclude that 
$X_{n,k} \cap \overline{C}_w = \varnothing$.  This proves the inclusion $\subseteq$.

To prove the inclusion $\supseteq$, we use the disjoint union decomposition $(\PP^{k-1})^n = \bigsqcup_{w \in [k]^n} C_w$
provided by Lemma~\ref{omega-structural-result} (4).  For any Fubini word $w \in \WW_{n,k}$, any matrix representing 
a point in $C_w$ lies in $\MMM_{n,k}$, so that $C_w \subseteq X_{n,k}$.  It follows that 
\begin{equation*}
(\PP^{k-1})^n - \bigcup_{\substack{w \in [k]^n \\ w \notin \WW_{n,k}}} \overline{C}_w \subseteq 
\bigcup_{w \in \WW_{n,k}} C_w \subseteq X_{n,k}.
\end{equation*}

(2)  The space $(\PP^{k-1})^n$ admits a cellular decomposition $(\PP^{k-1})^n_{\bullet}$ with 
 cells $\{ C_w \,:\, w \in [k]^n \}$.  By Lemma~\ref{single-cell-lemma}, we may refine 
 $(\PP^{k-1})^n_{\bullet}$ if necessary and assume that every difference $(\PP^{k-1})^n_{i} - (\PP^{k-1})^n_{i+1}$
 in $(\PP^{k-1})^n_{\bullet}$  consists of a single cell.  
 Now apply (1) of this theorem and the second paragraph of the proof of 
 Theorem~\ref{cohomology-presentation-theorem}.

(3)  Let $J_{n,k} \subseteq \ZZ[\xx_n]$ be the ideal generated by the variable powers $x_1^k, x_2^k, \dots, x_n^k$ together
with the polynomials
\begin{equation}
\{ \sigma(w)^{-1}.\symm_{\std(\conv(w))} \,:\, \text{$w = w_1 \dots w_n \in [k]^n$ and $w \notin \WW_{n,k}$} \}
\end{equation}
corresponding to each of the classes $[\overline{C}_w]$ in Proposition~\ref{word-representatives} such that at least
one letter in $1, 2, \dots, k$ does not appear in the word $w$.
Theorem~\ref{cohomology-presentation-theorem} and Proposition~\ref{word-representatives} imply
that the inclusion $j : X_{n,k} \hookrightarrow (\PP^{k-1})^n$ together with the presentation
\eqref{classical-projective-cohomology-presentation} realize $H^{\bullet}(X_{n,k})$ as the canonical quotient map
\begin{equation*}
H^{\bullet}(X_{n,k}) = \ZZ[\xx_n]/\langle x_1^k, \dots, x_n^k \rangle \twoheadrightarrow \ZZ[\xx_n]/J_{n,k} =: S_{n,k}.
\end{equation*}
Writing $\ell_1, \ldots, \ell_n$ for the tautological line bundles over $(\PP^{k-1})^n$, the restriction of $\ell_i^*$ to $X_{n,k}$ is the pullback bundle $j^* \ell_i^*$, and naturality of Chern classes gives $j^*(x_i) = j^*(c_1(\ell_i^*)) = c_1(j^* \ell_i^*)$. By item (2) of this theorem, Lemma~\ref{homology-from-cellular}, and Poincar\'e duality, the $\ZZ$-module $S_{n,k}$
is free and 
\begin{equation}
\mathrm{rank}(S_{n,k})  = |\WW_{n,k}| = k! \cdot \Stir(n,k).
\end{equation}

We aim to show that $I_{n,k} = J_{n,k}$, so that $R_{n,k} = S_{n,k}$.  We first prove the inclusion
$I_{n,k} \subseteq J_{n,k}$.  To start, for $1 \leq i \leq k$ let 
$w^i \in [k]^n$ be the unique weakly increasing word with letters $[k] - \{ i \}$ whose first $k-1$ letters are distinct.
For example, we have $w^3 \in [6]^8 = 12456666$.
The word $w^i$ is always convex, so that $\std(\conv(w^i)) = \std(w^i)$ and $\sigma(w^i) \in S_n$ is the identity permutation.
A direct computation shows $\std(w^i) = 12 \dots \hat{i} \dots n(n+1)i \in S_{n+1}$ 
is a Grassmannian permutation with descent at $n$
corresponding to the partition shape $(1^{n-i+1})$.  Therefore, the polynomial
$\symm_{\std(w^i)}$ is the Schur function $s_{(1^{n-i+1})}(\xx_n) = e_{n-i+1}(\xx_n)$.
Therefore $e_{n-i+1}(\xx_n)$ appears among the generators of $J_{n,k}$ for all $1 \leq i \leq k$ and we have
 the containment $I_{n,k} \subseteq J_{n,k}$.
 
 Since $I_{n,k} \subseteq J_{n,k}$, we have a canonical projection of $\ZZ$-modules 
 $\pi: R_{n,k} \twoheadrightarrow S_{n,k}$.  We know that $S_{n,k}$ is a free $\ZZ$-module of rank
 $k! \cdot \Stir(n,k)$.  In order to show that $\pi$ is an isomorphism it is enough to verify the following
 \begin{quote}
 {\bf Claim:}  {\em  The quotient $R_{n,k} = \ZZ[\xx_n]/I_{n,k}$ is a free $\ZZ$-module of rank $k! \cdot \Stir(n,k)$.}
 \end{quote}
 
 In order to prove our claim, we make use of the following `skip' terminology of Haglund-Rhoades-Shimozono \cite{HRS}.
 If $S = \{s_1 < \cdots < s_r \} \subseteq [n]$ is any subset, the corresponding {\em skip monomial} is
 $\xx(S) = x_{s_1}^{s_1} x_{s_2}^{s_2 - 1} \cdots x_{s_r}^{s_r - r + 1}$.  The exponent sequence of $\xx(S)$ is 
 recorded in the {\em skip composition} $\gamma(S) = (\gamma_1, \dots, \gamma_n)$ given by
 \begin{equation*}
 \gamma(S)_i = \begin{cases}
 s_j - j + 1 & \text{if $i = s_j \in S$} \\\
 0 & \text{if $i \notin S$}.
 \end{cases}
 \end{equation*}
 We will also need the {\em reverse skip composition} $\gamma(S)^* = (\gamma_n, \dots, \gamma_1)$.  
 As an example, if $n = 8$ and $S = \{2,4,5,8\}$ then $\xx(S) = x_2^2 x_4^3 x_5^3 x_8^5$ so that 
 $\gamma(S) = (0,2,0,3,3,0,0,5)$ and $\gamma(S)^* = (5,0,0,3,3,0,2,0)$.

 Let $m$ be any monomial in the variables $x_1, \dots, x_n$.  Following \cite[Defn. 4.4]{HRS} we say
 $m$ is {\em $(n,k)$-nonskip} if 
 \begin{itemize}
 \item  $x_i^k \nmid m$ for all $1 \leq i \leq k$, and
 \item  $\xx(S) \nmid m$ for all $S \subseteq [n]$ with $|S| = n-k+1$.
 \end{itemize}

 Writing $\NS_{n,k}$ for the set of all $(n,k)$-nonskip monomials, it is proven in
 \cite[Thm. 4.9]{HRS} that $|\NS_{n,k}| = |\OP_{n,k}| = k! \cdot \Stir(n,k)$.
 Furthermore, \cite[Thm 4.13]{HRS} shows that $\NS_{n,k}$ descends to a $\QQ$-basis for the 
 $\QQ$-vector space $\QQ \otimes_{\ZZ} R_{n,k}$.  In fact, if we endow the monomials in $x_1, \dots, x_n$ with
 the lexicographic term ordering \cite[Thm 4.13]{HRS} together with \cite[Proof of Thm. 4.10]{HRS} imply that 
 $\NS_{n,k}$ is the {\em standard monomial basis} for the quotient
 \begin{equation*}
 \QQ \otimes_{\ZZ} R_{n,k} = \QQ[\xx_n]/\langle x_1^k, \dots, x_n^k, e_n(\xx_n), e_{n-1}(\xx_n), \dots, e_{n-k+1}(\xx_n) \rangle.
 \end{equation*}
 
 We claim that the monomials in $\NS_{n,k}$ actually descend to a $\ZZ$-basis for $R_{n,k}$.  The monomials in 
 $\NS_{n,k}$ are linearly independent (over $\ZZ$) in $R_{n,k}$ because they are linearly independent (over $\QQ$)
 in $\QQ \otimes_{\ZZ} R_{n,k}$.
 It suffices to show that $\NS_{n,k}$ spans $R_{n,k}$ over $\ZZ$.
 
 In order to show that $\NS_{n,k}$ spans $R_{n,k}$ over $\ZZ$, let $m$ 
 be any monomial in the variables
 $x_1, \dots, x_n$.  We show that $m + I_{n,k}$ lies in the $\ZZ$-span of $\NS_{n,k}$.
 If $m \in \NS_{n,k}$ or if $x_i^k \mid m$ for some $1 \leq i \leq n$ this is obvious, so we may assume that 
 $\xx(S) \mid m$ for some $S \subseteq [n]$ with $|S| = n-k+1$.
 
 For any (weak) composition $\gamma = (\gamma_1, \dots, \gamma_n)$ of length $n$, let 
 $\kappa_{\gamma}(\xx_n) = \kappa_{\gamma}(x_1, \dots, x_n) \in \ZZ_{\geq 0}[\xx_n]$ be the 
 associated {\em Demazure character} (otherwise known as a {\em key polynomial});
 see \cite[Sec. 2]{HRS} for its definition.   Let
 $\kappa_{\gamma}(\xx_n^*) = \kappa_{\gamma}(x_n, \dots, x_1) \in \ZZ_{\geq 0}[\xx_n]$ be the 
 Demazure character with the variables reversed.
By \cite[Lem. 3.4]{HRS} (and in particular \cite[Eqn. 3.5]{HRS}) we have
$\kappa_{\gamma(S)^*}(\xx_n^*) \in I_{n,k}$.  Moreover, the polynomial $\kappa_{\gamma(S)^*}(\xx_n^*)$ has the form
(see e.g. \cite[Lem. 3.5]{HRS})
\begin{equation}
\label{demazure-form}
\kappa_{\gamma(S)^*}(\xx_n^*)  = 
\xx(S) + \text{a $\ZZ_{\geq 0}$-linear combination of monomials $< \xx(S)$ in lexicographic order}.
\end{equation}
If we let $m'$ be the quotient monomial $m' = m/\xx(S)$, multiply both sides of Equation~\ref{demazure-form} by $m'$,
and reduce modulo $I_{n,k}$, we see that 
\begin{equation}
\label{m-form}
m  \equiv
 \text{a $\ZZ$-linear combination of monomials $< m$ in lexicographic order}
\end{equation}
inside the ring $R_{n,k}$.

 By induction on the lexicographic term order, we conclude that $m + I_{n,k}$ lies in the $\ZZ$-span of $\NS_{n,k}$.  Therefore
 $\NS_{n,k}$ spans $R_{n,k}$ over $\ZZ$.  This completes the proof of the Claim and item (3) of the theorem.
 
(4)  This follows from Proposition~\ref{word-representatives} and Theorem~\ref{cohomology-presentation-theorem}.
\end{proof}

Recall that $R_{n,k}$ is not only a graded ring but a $\ZZ[S_n]$-module. On the geometric side, $X_{n,k}$ is stable under the (right) action of $S_n$ on $(\PP^{k-1})^n$ permuting lines, so there is an induced (left) $S_n$-action on $H^{\bullet}(X_{n,k})$.
\begin{theorem} \label{thm:equivariance} The isomorphism $H^{\bullet}(X_{n,k}) \cong R_{n,k}$ of Theorem~\ref{cohomology-of-x}(3) is an isomorphism of graded $S_n$-modules. \end{theorem}
\begin{proof} Theorem~\ref{cohomology-of-x}(3) gives a commuting diagram
\begin{equation*}
\begin{CD}
\frac{\ZZ[x_1, \ldots, x_n]}{\langle x_1^k, \ldots, x_n^k \rangle} @>>> H^{\bullet}((\PP^{k-1})^n)\\
@VVV                    @VV{j^*}V\\
R_{n,k} @>>>         H^{\bullet}(X_{n,k})
\end{CD}
\end{equation*}
where the vertical arrows are quotient maps and the horizontal arrows are isomorphisms. To show that the bottom arrow is $S_n$-equivariant, it suffices to check that the other three arrows are. This is clear for the lefthand arrow. The inclusion $j : X_{n,k} \hookrightarrow (\PP^{k-1})^n$ is $S_n$-equivariant, so the righthand arrow $j^*$ is also equivariant. The top arrow sends $x_i$ to $c_1(\ell_i^*)$. Identifying $\pi \in S_n$ with its action on $(\PP^{k-1})^n$, the pullback bundle $\pi^*\ell_i$ is $\ell_{\pi(i)}$, so the top arrow is equivariant by naturality of Chern classes: $\pi^*c_1(\ell_i^*) = c_1(\pi^*\ell_i^*) = c_1(\ell_{\pi(i)}^*) = x_{\pi(i)}$.

\end{proof}

Proposition~\ref{word-representatives} and
Theorem~\ref{cohomology-of-x} motivate the following definition of a `Schubert polynomial'
$\symm_w$ attached to an arbitrary word $w \in [k]^n$.

\begin{defn}
\label{fubini-schuberts}
Let $k \leq n$ be positive integers and let $w = w_1 \dots w_n \in [k]^n$ be a word.  The 
{\em Schubert polynomial} 
$\symm_w \in \ZZ[\xx_n]$ attached to $w$ is 
$\symm_w := \sigma(w)^{-1}.\symm_{\std(\conv(w))}$.
\end{defn}

When $k = n$ and $w = w_1 \dots w_n \in S_n$ is a permutation, 
Definition~\ref{fubini-schuberts}
reduces to the classical notion of a Schubert polynomial.
In the proof of Theorem~\ref{cohomology-of-x} (3) it was shown that we have the
equality of ideals
\begin{equation}
J_{n,k} := \langle \{ x_1^k, \dots, x_n^k \} \cup \{ \symm_w \,:\, w \in [k]^n - \WW_{n,k} \} \rangle
= \langle x_1^k, \dots, x_n^k, e_n, e_{n-1}, \dots, e_{n-k+1} \rangle = I_{n,k}.
\end{equation}
The generators of $J_{n,k}$ are the natural polynomials coming from the cellular
decomposition of $X_{n,k}$, whereas the smaller generating set of $I_{n,k}$
presents the quotient $R_{n,k}$.  
Our proof that the generators of $J_{n,k}$ which do not appear as generators of $I_{n,k}$
was indirect and used a Demazure character identity established in \cite{HRS};
it might be interesting to get a direct proof of redundancy.

The set $\{ \symm_w \,:\, w \in \WW_{n,k} \}$ 
of Schubert polynomials corresponding to Fubini words forms a
basis for the quotient $R_{n,k}$.

\begin{corollary}
\label{schubert-basis}
Let $k \leq n$ be positive integers.  
\begin{enumerate}
\item
The quotient $R_{n,k}$ is a free $\ZZ$-module with basis 
$\{ \symm_w \,:\, w \in \WW_{n,k} \}$.  
\item
The polynomial $\symm_w$ has nonnegative integer coefficients and is homogeneous of degree 
${k \choose 2} + (n-k)(k-1) - \dim(w)$.
\end{enumerate}
\end{corollary}

\begin{proof}
(1)   This follows from items (2), (3), and (4) of Theorem~\ref{cohomology-of-x}.

(2)  If $w \in \WW_{n,k}$ is a Fubini word, the dimension of the cell $C_w$ is 
\begin{equation*}
\dim(C_w) = 
\dim(U) + \dim(\widehat{C}_w) = {k \choose 2} + \dim(w).
\end{equation*}
The codimension of $C_w$ inside the $n(k-1)$-dimensional space $(\PP^{k-1})^n$ is therefore
\begin{equation*}
n(k-1) - {k \choose 2} - \dim(w) = {k \choose 2} + (n-k)(k-1) - \dim(w).
\end{equation*}
\end{proof}

In the case $k = n$, the Schubert basis $\{ \symm_w \,:\, w \in S_n \}$ 
of $R_n$ is known to have positive structure constants.  That is, for all $u, v \in S_n$ we have
\begin{equation}
\symm_u \cdot \symm_v = \sum_{w \in S_n} c_{u,v}^w \symm_w \quad \text{(mod $I_n$)},
\end{equation}
where  $c_{u,v}^w \in \ZZ_{\geq 0}$.  The only known proof that the 
$c_{u,v}^w$ are nonnegative uses the geometry of the flag variety $\Fl(n)$; it is an
open problem to find a combinatorial interpretation of the $c_{u,v}^w$.

The basis $\{ \symm_w \,:\, w \in \WW_{n,k} \}$ of $R_{n,k}$ does {\em not}
have positive structure constants for arbitrary $k \leq n$.
As an example, when $(n,k) = (4,3)$ we have
\begin{equation}
\symm_{1123} \cdot \symm_{1232} = - \symm_{1132} + 2 \symm_{2213}
\quad \text{(mod $I_{4,3}$)}.
\end{equation}

Recall that the {\em Poincar\'e series} of any topological space $X$ is 
$P_X(q) = \sum_{k \geq 0} \beta_k(X) \cdot q^k \in \ZZ[[q]]$, where the {\em Betti number}
$\beta_k(X)$ is the rank of the $\ZZ$-module
$H^k(X)$ (here we assume that $\beta_k$ is finite for all $k$).

\begin{corollary}
\label{poincare-polynomial}
The Poincar\'e series of $X_{n,k}$ is $P_{X_{n,k}}(q) = \rev_q ([k]!_{q^2} \cdot \Stir_{q^2}(n,k))$.
\end{corollary}

In general, the polynomial $P_{X_{n,k}}(q)$ is not palindromic.  For example, we have
$P_{X_{3,2}}(q) = 2q^4 + 3q^2 + 1$.

The proof of Theorem~\ref{cohomology-of-x} (3) shows that various polynomials lie in the ideal $I_{n,k} \subseteq \ZZ[\xx_n]$.

\begin{corollary}
\label{members-of-i}
Let $k \leq n$ be positive integers and let $w \in [k]^n$ be any word in which the letters $1, 2, \dots, k$ 
do not all appear.  We have
$\symm_w \in I_{n,k}$.
\end{corollary}

\section{The cohomology of $X_{n,k}$ via Chern classes}
\label{Chern}


In the last section we computed the cohomology ring $H^{\bullet}(X_{n,k})$ 
using a cellular decomposition of $X_{n,k}$.
In this section we recast $H^{\bullet}(X_{n,k})$ in terms of 
Chern classes of the $n$ tautological line bundles
 $\ell_1, \dots, \ell_n \twoheadrightarrow X_{n,k}$.

The space 
$X_{n,k} = \{ (\ell_1, \dots, \ell_n) \,:\, \text{$\ell_i$ a line in $\CC^k$ and $\ell_1 + \cdots + \ell_n = \CC^k$} \}$ has 
$n$ line bundles $\ell_1, \dots, \ell_n \twoheadrightarrow X_{n,k}$ together with their duals
$\ell_1^*, \dots, \ell_n^* \twoheadrightarrow X_{n,k}$.
For $1 \leq i \leq n$, we denote the first Chern class of $\ell_i^*$ by 
\begin{equation}
y_i := c_1(\ell_i^*) \in H^2(X_{n,k})
.\end{equation}
It is evident that $\pi.y_i = y_{\pi_i}$ for all $\pi \in S_n$ and $1 \leq i \leq n$.

\begin{theorem}
\label{chern-class-cohomology}
Let $k \leq n$ be positive integers and let 
\begin{equation}
\label{chern-equation}
H^{\bullet}(X_{n,k}) = R_{n,k} =
\ZZ[\xx_n]/\langle x_1^k, \dots, x_n^k, e_n(\xx_n), e_{n-1}(\xx_n), \dots, e_{n-k+1}(\xx_n) \rangle 
\end{equation}
be the 
presentation of $H^{\bullet}(X_{n,k})$ given in Theorem~\ref{cohomology-of-x}.

\begin{enumerate}
\item
For $1 \leq i \leq n$, the Chern class $y_i = c_1(\ell_i^*) \in H^2(X_{n,k})$ is represented by $x_i$.
\item
The identification \eqref{chern-equation} of $H^{\bullet}(X_{n,k})$ may be viewed as an isomorphism of graded
$\ZZ[S_n]$-modules.
\end{enumerate}
\end{theorem}

\begin{proof}
(2) follows immediately from (1), so it suffices to prove $(1)$.  In order to prove $(1)$ we show the equalities 
in $R_{n,k}$: 
\begin{equation}
\label{first-identifications}
\begin{cases}
x_1 + \cdots + x_{i-1} + x_i = y_1 + \cdots + y_{i-1} + y_i & \text{if $1 \leq i \leq k$}, \\
x_1 + \cdots + x_{k-1} + x_i = y_1 + \cdots + y_{k-1} + y_i & \text{if $k+1 \leq i \leq n$}.
\end{cases}
\end{equation}
The system of equalities \eqref{first-identifications} implies $x_i = y_i$ for all $1 \leq i \leq n$.

Consider the $n$ Fubini words in $\WWW_{n,k}$:
\begin{align*}
& w^{(1)} = 213 \dots k k \dots k, 
 w^{(2)} = 132 \dots k k \dots k,  
  \dots , 
 w^{(k-1)} = 123  \dots (k-2)k(k-1)(k-1) \dots (k-1), \\
& w^{(k)} = 123 \dots (k-1) (k-1) k \dots k, 
 w^{(k+1)} = 1234 \dots k  (k-1) \dots k, 
 \dots, 
w^{(n)} = 1234 \dots k k \dots  k (k-1).
\end{align*}
As an example, if $(n,k) = (8,4)$ the words $w^{(1)}, \dots, w^{(8)} \in \WW_{8,4}$ and their convexifications,
standardizations, and permutations are given as follows:
\begin{center}
\begin{tabular}{c | c | c | c | c}
$i$ & $w^{(i)}$ & $\conv(w^{(i)})$ & $\std(\conv(w^{(i)}))$ & $\sigma(w^{(i)})$ \\ \hline
1 & 21344444 & 21344444 & 21345678 & 12345678 \\
2 & 13244444 & 13244444 & 13245678 & 12345678 \\
3 & 12433333 & 12433333 & 12435678 & 12345678 \\
4 & 12334444 & 12334444 & 12354678 & 12345678 \\
5 & 12343444 & 12334444 & 12354678 & 12354678 \\
6 & 12344344 & 12334444 & 12354678 & 12364578 \\
7 & 12344434 & 12334444 & 12354678 & 12374568 \\
8 & 12344443 & 12334444 & 12354678 & 12384567
\end{tabular}
\end{center}
The words $w^{(i)}$ are convex when $1 \leq i \leq k$ and not convex when $k+1 \leq i \leq n$.
A direct computation 
shows that the Schubert polynomials $\symm_{w^{(i)}}$ are given by
\begin{equation}
\symm_{w^{(i)}} = \begin{cases}
x_1 + \cdots + x_{i-1} + x_i & \text{if $1 \leq i \leq k$}, \\
x_1 + \cdots + x_{k-1} + x_i & \text{if $k+1 \leq i \leq n$}.
\end{cases}
\end{equation}
We will finish the proof by means of two claims.

{\bf Claim 1:}  {\em It suffices to prove \eqref{first-identifications} in the range $1 \leq i \leq k$.}

By Theorem~\ref{cohomology-of-x} (4), the class $[\overline{C}_{w^{(i)}}] \in H^2(X_{n,k})$ is represented
by $\symm_{w^{(i)}}$ for $1 \leq i \leq n$.  For $k+1 \leq i \leq n$, 
we have the equality of pattern matrices $\PM(w^{(k)}) \sigma(w^{(i)}) = \PM(w^{(i)})$, where 
$\sigma(w^{(i)})^{-1}$ is thought of as an $n \times n$ permutation matrix acting on the columns 
of $\PM(w^{(k)})$.
This leads to the equality of cells $C_{w^{(k)}} \sigma(w^{(i)})^{-1} = C_{w^{(i)}}$ and in cohomology we have
\begin{equation}
\sigma(w^{(i)}).[\overline{C}_{w^{(k)}}] = [\overline{C}_{w^{(k)}} \sigma(w^{(i)})^{-1}] =
[\overline{C}_{w^{(i)}}].
\end{equation}
This matches the action of $\sigma(w^{(i)})$ on the representatives of $[\overline{C}_{w^{(k)}}]$
(namely $x_1 + \cdots + x_{k-1} + x_k$) and $[\overline{C}_{w^{(i)}}]$ (namely $x_1 + \cdots + x_{k-1} + x_i$).
In terms of Chern classes we have
\begin{equation*}
\sigma(w^{(i)}).(y_1 + \cdots + y_k) = \sigma(w^{(i)}). c_1(\ell_1^* \otimes \cdots \otimes \ell_k^*) = 
c_1(\ell_{1}^* \otimes \cdots \otimes \ell_{k-1}^* \otimes \ell_{i}^*) = y_{1} + \cdots + y_{k-1} + y_i.
\end{equation*}
The equality $y_1 + \cdots + y_k = x_1 + \cdots + x_k$ therefore implies the equality
$y_1 + \cdots + y_{k-1} + y_i = x_1 + \cdots  + x_{k-1} + x_i$.

{\bf Claim 2:}  {\em  \eqref{first-identifications}  holds in the range $1 \leq i \leq k$.}

Let $1 \leq i \leq k$.
It is enough to show that 
$c_1(\ell_1^* \otimes \cdots \otimes \ell_i^*) =  [ \overline{C}_{w^{(i)}}]$.
We do this by showing that $\overline{C}_{w^{(i)}}$ is the zero locus of a sufficiently generic 
section of the bundle $\ell_1^* \otimes \cdots \otimes \ell_i^* = (\ell_1 \otimes \cdots \otimes \ell_i)^*$.

We have a linear functional $\delta: \CC^k \otimes \cdots \otimes \CC^k \rightarrow \CC$ (with $i$ copies of $\CC^k$)
given by letting $\delta(v_1 \otimes \cdots \otimes v_i)$ be the $[i] \times [i]$-minor of the matrix 
$(v_1 \cdots v_i)$ with columns $v_1, \dots, v_i \in \CC^k$.  We get a section
$s: X_{n,k} \rightarrow (\ell_1 \otimes \cdots \otimes \ell_i)^*$ by 
\begin{equation*}
s: (\ell_1, \dots, \ell_n) \mapsto \text{restriction of $\delta$ to $\ell_1 \otimes \cdots \otimes \ell_i$}.
\end{equation*}
The vanishing locus $Z = \{ \ell_{\bullet} \in X_{n,k} \,:\, s(\ell_{\bullet}) = 0 \}$ is precisely
\begin{equation*}
Z = \{  (\ell_1, \dots, \ell_n) \in X_{n,k} \,:\, \Delta_{[i], [i]}(\ell_1 \cdots \ell_n) = 0 \},
\end{equation*}
where we interpret $(\ell_1 \cdots \ell_n)$ as the $k \times n$ matrix, defined up to column rescaling,
whose $i^{th}$ column is a nonzero vector in $\ell_i$.

We need to prove $\overline{C}_{w^{(i)}} = Z$.  Since $w^{(i)}$ is convex and
$\std(w^{(i)}) = w^{(i)}$, Lemma~\ref{convex-word-representatives}
gives $\overline{C}_{w^{(i)}} = \overline{\Omega}_{w^{(i)}}$.  
By intersecting the description $\overline{\Omega}_{w^{(i)}}$ (with the closure taken inside $(\PP^{k-1})^n$)
of Lemma~\ref{omega-closure} with $X_{n,k}$ we have 
\begin{equation*}
\overline{\Omega}_{w^{(i)}} = \bigcup_{\substack{v \in \WW_{n,k} \\ r(v) \leq r(w^{(i)})}}  \Omega_v.
\end{equation*}
By the definition of $w^{(i)}$, for $1 \leq a \leq k$ and $1 \leq b \leq n$ we have
\begin{equation*}
r(w^{(i)}):  (a,b) \mapsto \begin{cases}
\min(a,b) & \text{if $(a,b) \neq (i,i)$} \\
i-1 & \text{if $(a,b) = (i,i)$}.
\end{cases}
\end{equation*}
Given any Fubini word $v \in \WW_{n,k}$, we have $r(v) \leq r(w^{(i)})$ if and only if the rank function $r(v)$ satisfies
$r(v)(i,i) < i$.  This means 
\begin{equation*}
\overline{\Omega}_{w^{(i)}} = \bigcup_{\substack{v \in \WW_{n,k} \\ r(v) \leq r(w^{(i)})}}  \Omega_v
= \{ (\ell_1, \dots, \ell_n) \in X_{n,k} \,:\, \Delta_{[i], [i]}(\ell_1 \cdots \ell_n) = 0 \} = Z.
\end{equation*}
Since our section $s: X_{n,k} \rightarrow \ell_1^* \otimes \cdots \otimes \ell_i^*$ is sufficiently
generic
(i.e., the subvariety $s^{-1}(0)$ has the expected codimension)
we conclude that 
\begin{equation*}
y_1 + \cdots + y_i = c_1(\ell_1^* \otimes \cdots \otimes \ell_i^*) = [\overline{C}_{w^{(i)}}] = x_1 + \cdots + x_i,
\end{equation*}
as desired.
\end{proof}

By Theorem~\ref{chern-class-cohomology}, the Chern classes $y_1, \dots, y_n \in H^2(X_{n,k})$ satisfy the relations 
$y_1^k = \cdots = y_n^k = 0$ and 
$e_r(y_1, \dots, y_n) = 0$ for all $r > n-k$.
Both of these relations in cohomology can be seen directly using standard properties of Chern classes.

Let $\hat{\ell}_i$ be the $i^{th}$ tautological line bundle over the $n$-fold product
$(\PP^{k-1})^n$ and let $\hat{\ell}_i^*$ be its dual.
If $\iota: X_{n,k} \hookrightarrow (\PP^{k-1})^n$ is the inclusion and 
$\iota^*: H^{\bullet}((\PP^{k-1})^n) \rightarrow H^{\bullet}(X_{n,k})$ is the induced map on cohomology,
since $\ell_i^*$ is the pullback of $\widehat{\ell}_i^*$ to $X_{n,k}$ under $\iota$ we have
$y_i^k = c_1(\ell_i^*)^k = \iota^* c_1( \widehat{\ell}_i^*)^k = \iota^*(0) = 0$.

Let $r > n-k$.  The addition map gives a surjection of vector bundles over $X_{n,k}$ given by
\begin{equation}
\label{vector-bundle-surjection}
\ell_1 \oplus \cdots \oplus \ell_n \twoheadrightarrow \ell_1 + \cdots + \ell_n = \CC^k,
\end{equation}
where $\CC^k$ is the trivial rank $k$ bundle over $X_{n,k}$.    Dualizing \eqref{vector-bundle-surjection}
we get an injection 
\begin{equation}
\label{vector-bundle-injection}
(\CC^k)^* \hookrightarrow (\ell_1 \oplus \cdots \oplus \ell_n)^* = \ell_1^* \oplus \cdots \oplus \ell_n^*
\end{equation}
which embeds the rank $k$ trivial bundle $(\CC^k)^*$ as a subbundle of 
$\ell_1^* \oplus \cdots \oplus \ell_n^*$.

Since any short exact sequence of vector bundles
splits, the injection \eqref{vector-bundle-injection} implies an isomorphism of bundles
\begin{equation}
\label{vector-bundle-direct-sum}
\ell_1^* \oplus \cdots \oplus \ell_n^* \cong (\CC^k)^* \oplus E,
\end{equation}
where $E$ is a rank $n-k$ vector bundle on $X_{n,k}$. 
Taking the $r^{th}$ Chern class of both sides of \eqref{vector-bundle-direct-sum}
and applying the Whitney sum formula gives
\begin{equation}
e_r(y_1, \dots, y_n) = c_r(\ell_1^* \oplus \cdots \oplus \ell_n^*) = c_r( (\CC^k)^* \oplus E) = c_r(E) =  0,
\end{equation}
since $E$ has rank $n-k < r$.

If we assume that the map 
$\iota^*: H^{\bullet}((\PP^{k-1})^n) \twoheadrightarrow H^{\bullet}(X_{n,k})$ is surjective,
the above two paragraphs imply that $H^{\bullet}(X_{n,k})$ is a quotient of $R_{n,k}$.
Unfortunately, it seems that the cellular structure of Theorem~\ref{cohomology-of-x}
is necessary to bound the rank of $H^{\bullet}(X_{n,k})$ from below and
show $H^{\bullet}(X_{n,k}) = R_{n,k}$.

\section{Stability}
\label{Stability}

We have a natural embedding of symmetric groups $S_n \subset S_{n+1}$ obtained by associating $w \in S_n$
to the permutation $w \times 1 \in S_{n+1}$ which sends $i$ to $w_i$ for $1 \leq i \leq n$ and fixes $n+1$.
This embedding is compatible with Schubert polynomials in the sense that 
$\symm_w = \symm_{w \times 1}$, and so it makes sense to consider Schubert polynomials
$\symm_w$ indexed by permutations $w$ in the infinite symmetric group $S_{\infty} = \bigcup_{n \geq 1} S_n$
of eventually constant bijections $w: \ZZ_{> 0} \rightarrow \ZZ_{> 0}$.
The set $\{ \symm_w \,:\, w \in S_{\infty} \}$ forms a $\ZZ$-basis for 
the ring $\ZZ[x_1, x_2, \dots ]$ of polynomials in infinitely many variables.
This stability result may be viewed at the level of Schubert classes as 
coming from the
embedding $GL_n(\CC) \rightarrow GL_{n+1}(\CC)$ sending $a \in GL_n(\CC)$ 
to the matrix direct sum $a \oplus (1)$.

The embedding $S_n \subset S_{n+1}$ 
given by $w \mapsto w \times 1$ extends to Fubini words as
follows.  Given $w \in \WW_{n,k}$, let $w \times 1 \in \WW_{n+1,k+1}$ be the Fubini word obtained by appending 
a $k+1$ to the end of $w$, so that e.g. $211 \times 1 = 2113$.  Unfortunately, this combinatorial embedding does not 
have nice Schubert stability properties for $k < n$:
we have $\symm_{211} = x_1$ whereas $\symm_{211 \times 1}(\xx_4) = \symm_{2113}(\xx_4) = x_1^2 + x_1 x_2 + x_1 x_3$.
Geometrically, the degree discrepancy in this example corresponds to the fact that 
$C_{211}$ has codimension $1$ in $X_{3,2}$ and $C_{2113}$ has codimension $2$ in $X_{4,3}$.

Given $k \leq n$, there are two other natural ways 
$w \mapsto w \circledast 1$ and $w \mapsto 1 \times w$
to embed $\WW_{n,k}$ into a larger set of Fubini words.
These embeddings have nicer Schubert stability properties and can be 
interpreted geometrically in terms of the varieties $X_{n,k}$.

For $w \in \WW_{n,k}$, let $1 \times w \in \WW_{n+1,k+1}$ be the  word obtained by prepending a $1$ to the start of
$w$, and then increasing every letter in $w$ by $1$.  For example, we have
\begin{equation*}
1 \times 3313424 = 14424535.
\end{equation*}
If $\initial(w) = \{i_1 < \cdots < i_m \}$,
we  define $w \circledast 1 \in \WW_{n+1,k}$ by appending $w_{i_m}$ to the end of $w$, so that
\begin{equation*}
3313424 \circledast 1 = 33134242.
\end{equation*}
The $\times$ and $\circledast$ maps lead to two towers of sets
\begin{align*}
\cdots \hookrightarrow \WW_{n-1,k-1} \hookrightarrow &\WW_{n, k} \hookrightarrow \WW_{n+1,k+1} \hookrightarrow \cdots \\
\cdots \hookrightarrow \WW_{n-1,k} \hookrightarrow &\WW_{n,k} \hookrightarrow \WW_{n+1,k} \hookrightarrow \cdots
\end{align*}

We will prove stability properties for our Fubini word Schubert polynomials 
$\{ \symm_w \,:\, \WW_{n,k} \}$ under the two embeddings described above.
The embedding $\WW_{n,k} \subseteq \WW_{n+1,k}$ is easier to understand.

\begin{theorem}
\label{fixed-k-embedding}
For any $w \in \WW_{n,k}$ we have $\symm_{w \circledast 1} = \symm_w$.
\end{theorem}

We will give two proofs of this result: combinatorial and geometric.

\begin{proof}
{\em (Combinatorial)}
Write $\initial(w) = \{i_1 < \cdots < i_m \}$ and
$w = w_1 \dots w_n$, so  $w \circledast 1 = w_1 \dots w_n w_{i_m}$.
If $\conv(w) = v_1 \dots v_n$ we have $v_n = w_{i_m}$ and 
$\conv(w \circledast 1) = v_1 \dots v_n w_{i_m}$.  This implies that both
$\std(\conv(w \circledast 1)) = \std(\conv(1)) \times 1 \in S_{n+1}$ and
$\sigma(w \circledast 1) = \sigma(w) \times 1 \in S_{n+1}$.  Therefore, we have
\begin{equation*}
\symm_{w \circledast 1} = \sigma(w \circledast 1)^{-1}.\symm_{\std(\conv(w \circledast 1))}
= ( \sigma(w)^{-1} \times 1).\symm_{\std(\conv(w)) \times 1} =
\sigma(w)^{-1}.\symm_{\std(\conv(w))} = \symm_w.
\end{equation*}
\end{proof}

\begin{proof}
{\em (Geometric)}
We have a closed embedding $X_{n,k} \hookrightarrow X_{n+1,k}$ given by 
$(\ell_1, \dots, \ell_n) \mapsto (\ell_1, \dots, \ell_n, {\bf 1})$, where ${\bf 1}$ is the line of constant vectors
in $\CC^k$.  With respect to this embedding, the cell $C_{w \circledast 1} \subseteq X_{n+1,k}$ meets
$X_{n,k}$ transversely and we have $C_{w \circledast 1} \cap X_{n,k} = C_w$.
\end{proof}

We want to describe the stability result of Theorem~\ref{fixed-k-embedding}
in terms of the cohomology of the limit space formed by the union of the $X_{n,k}$
for $k$ fixed.  To do this, we will need the following well-known lemma describing the 
cohomology of an increasing union of spaces.

\begin{lemma}
\label{inverse-limit-cohomology-lemma}
Let $Y_1 \subseteq Y_2 \subseteq \cdots$ be an increasing chain of topological spaces. Let
\begin{equation*}
Y = \lim_{\substack{\longrightarrow \\ i \geq 1}} Y_i = \bigcup_{i \geq 1} Y_i
\end{equation*}
by the union of these spaces.  Assume that
\begin{itemize}
\item any compact subspace $C \subseteq Y$ satisfies $C \subseteq Y_i$ for some $i$, and
\item  the groups $H_j(Y_i)$ and $H^j(Y_i)$ vanish for all $i$ and odd $j$.
\end{itemize}
The inclusions $Y_1 \hookrightarrow Y_2 \hookrightarrow \cdots$ gives rise to 
maps $H^{\bullet}(Y_1) \leftarrow H^{\bullet}(Y_2) \leftarrow \cdots$.  We have an isomorphism
\begin{equation}
H^{\bullet}(Y) = H^{\bullet}(\lim_{\substack{\longrightarrow \\ i \geq 1}} Y_i) \cong
\lim_{\substack{\longleftarrow \\ i \geq 1}} H^{\bullet}(Y_i).
\end{equation}
\end{lemma}

Lemma~\ref{inverse-limit-cohomology-lemma} allows us to express the cohomology
of an increasing union of spaces as the inverse limit of the cohomologies of these spaces.
The technical hypotheses of Lemma~\ref{inverse-limit-cohomology-lemma} 
hold in our setting because the spaces $X_{n,k}$ admit a (complex) cellular decomposition.

\begin{proof}
The inclusions $Y_i \hookrightarrow Y$ induce maps on cohomology
$H^{\bullet}(Y) \rightarrow H^{\bullet}(Y_i)$, and therefore a ring map
$f: H^{\bullet}(Y) \rightarrow \varprojlim H^{\bullet}(Y_i)$.
We need only show that the ring map $f$ is an isomorphism.

By our hypothesis on compact subsets of $Y$ and  \cite[Prop. 3.33]{Hatcher},
the natural map $\varinjlim H_{\bullet}(Y_i) \rightarrow H_{\bullet}(Y)$ on homology
is an isomorphism.  
In particular, we have $H_j(Y) = 0$ for all odd $j$.

On the other hand, the Universal Coefficient Theorem
tells us that for each $i, j$ we have an exact sequence
\begin{equation}
0 \rightarrow \mathrm{Ext}(H_{j-1}(Y_i)) \rightarrow H^j(Y_i) \rightarrow 
\mathrm{Hom}(H_j(Y_i), \ZZ) \rightarrow 0.
\end{equation}
By our vanishing assumptions, we see that 
$H^j(Y_i) \cong \mathrm{Hom}(H_j(Y_i), \ZZ)$ for all $i$.
This gives rise to the identifications
\begin{equation}
\mathrm{Hom}( H_j(Y), \ZZ) = 
\mathrm{Hom}( \varinjlim H_j(Y_i), \ZZ) = \varprojlim \mathrm{Hom} (H_j(Y_i), \ZZ) =
\varprojlim H^j(Y_i).
\end{equation}
For the second equality above, see \cite[Proof of Prop. 3F.5]{Hatcher}.

Finally, the Universal Coefficient Theorem applied to $Y$ says that for each $j$,
\begin{equation}
0 \rightarrow \mathrm{Ext}(H_{j-1}(Y)) \rightarrow H^j(Y) \rightarrow 
\mathrm{Hom}(H_j(Y), \ZZ) \rightarrow 0.
\end{equation}
Since $H_j(Y) = 0$ if $j$ is odd, we get the desired isomorphisms of graded abelian 
groups
\begin{equation}
\label{identification-sequence}
H^{\bullet}(Y) = \mathrm{Hom}(H_{\bullet}(Y),\ZZ) = \varprojlim H^{\bullet}(Y_i).
\end{equation}
Since the ring map $f: H^{\bullet}(Y) \rightarrow \varprojlim H^{\bullet}(Y_i)$ is 
also a map of graded abelian groups, the universal property of the direct limit 
forces $f$ to be obtained by the chain of isomorphisms \eqref{identification-sequence}.
This means that $f$ is an isomorphism.
\end{proof}

The closed embedding $X_{n,k} \hookrightarrow X_{n+1,k}$ given by $(\ell_1, \dots, \ell_n) \mapsto (\ell_1, \dots, \ell_n, {\bf 1})$ 
gives rise to the direct limit 
\begin{equation}
X_{\bullet, k} := \lim_{\substack{\longrightarrow \\ n \geq k}} X_{n,k} = \bigcup_{n \geq k} X_{n,k}.
\end{equation}
Applying Lemma~\ref{inverse-limit-cohomology-lemma}, we see that 
the limit $X_{\bullet, k}$ has cohomology 
\begin{equation} \label{eq:direct-limit-cohom}
H^{\bullet}(X_{\bullet, k}) = H^{\bullet}(\lim_{\substack{\longrightarrow \\ n \geq k}} X_{n,k}) =
\lim_{\substack{\longleftarrow \\ n \geq k}} H^{\bullet}(X_{n,k}) =
\lim_{\substack{\longleftarrow \\ n \geq k}} R_{n,k} =: R_{\bullet,k}
\end{equation}
where the inverse limit 
is defined with respect to the system $R_{n+1,k} \twoheadrightarrow R_{n,k}$
induced by the variable assignment $(x_1, \dots, x_n, x_{n+1}) \mapsto (x_1, \dots, x_n, 0)$.
If we fix $n_0 \geq k$ and $w \in \WW_{n_0, k}$, the image of 
$\symm_w$ in $H^{\bullet}(X_{\bullet, k}) = R_{\bullet,k}$ represents the class of the image of 
$\overline{C}_w$ in $X_{\bullet,k}$.
We may identify the inverse limit $H^{\bullet}(X_{\bullet, k}) = R_{\bullet,k}$ as follows.

\begin{proposition}
\label{inverse-limit-identification}
Let $k$ be a positive integer and consider the ring
\begin{equation*}
P_{\bullet,k} := \ZZ[[ x_1, x_2, \dots ]]/ \langle x_1^k, x_2^k, \dots \rangle
\end{equation*}
formed by truncating the ring of formal power series $\ZZ[[x_1, x_2, \dots ]]$.
We have a canonical isomorphism $R_{\bullet, k} \cong P_{\bullet, k}$.
\end{proposition}

\begin{proof}
We show that $P_{\bullet, k}$ satisfies the universal property defining $R_{\bullet,k}$.
The idea is to approximate the inverse system defining $R_{\bullet,k}$ with
a simpler inverse system whose limit is $P_{\bullet,k}$.

For any $n \geq k$, consider the quotient 
\begin{equation}
P_{n,k} := \ZZ[\xx_n]/\langle x_1^k, x_2^k, \dots, x_n^k \rangle.
\end{equation}
We have a canonical projection $P_{n+1,k} \twoheadrightarrow P_{n,k}$
obtained by evaluating $x_{n+1} = 0$.  There are also projection maps
$P_{\bullet,k} \twoheadrightarrow P_{n,k}$ obtained by evaluating $x_i = 0$ for all
$i > n$.  It can be shown that $P_{\bullet,k}$ is the inverse limit of the 
system $P_{n+1,k} \twoheadrightarrow P_{n,k}$.

We want to approximate the system $R_{n+1,k} \twoheadrightarrow R_{n,k}$ by the 
system $P_{n+1,k} \twoheadrightarrow P_{n,k}$.  The main idea here is that for all 
$n \geq k$, we have a surjection $R_{n+k,k} \twoheadrightarrow P_{n,k}$
induced by setting $x_i = 0$ for all $n+1 \leq i \leq n+k$.  Indeed, this surjection
comes from the corresponding map $P_{n+k,k} \twoheadrightarrow P_{n,k}$
and the fact that $e_{n+k}, e_{n+k-1}, \dots, e_{n+1} = 0$ in $P_{n,k}$.
To summarize, we have the following commutative diagram:
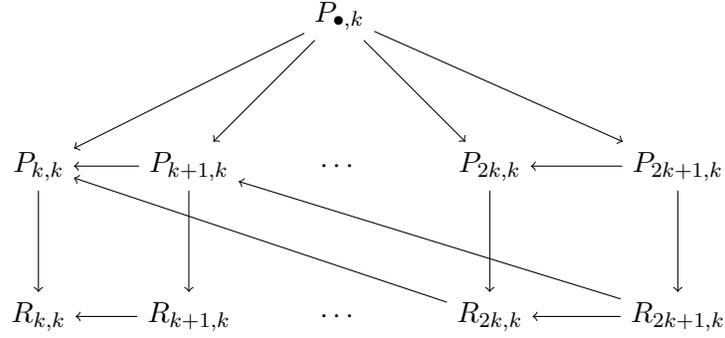
\begin{center}
\begin{tikzpicture}[scale = 1]

\node(A) at (0,2) {$P_{k,k}$};

\node(B) at (2,2) {$P_{k+1,k}$};

\node(C) at (4,2) {$\cdots$};

\node(D) at (6,2) {$P_{2k,k}$};

\node(E) at (8.5,2) {$P_{2k+1,k}$};

\node(F) at (0,0) {$R_{k,k}$};

\node(G) at (2,0) {$R_{k+1,k}$};

\node(H) at (4,0) {$\cdots$};

\node(I) at (6,0) {$R_{2k,k}$};

\node(J) at (8.5,0) {$R_{2k+1,k}$};

\node(K) at (4,4) {$P_{\bullet,k}$};

\draw[ -> ] (B) -- (A);
\draw [->] (A) -- (F);
\draw [->] (E) -- (D);
\draw [->] (B) -- (G);
\draw [->] (G) -- (F);
\draw [->] (D) -- (I);
\draw [->] (E) -- (J);
\draw [->] (I) -- (A);
\draw [->] (J) -- (I);
\draw [->] (J) -- (B);
\draw [->] (K) -- (A);
\draw [->] (K) -- (B);
\draw [->] (K) -- (D);
\draw [->] (K) -- (E);

\end{tikzpicture}
\end{center}
Since $P_{\bullet,k}$ is the inverse limit of the directed system formed by the middle row
$P_{n+1,k} \twoheadrightarrow P_{n,k}$,
the universal property of the inverse limit immediately implies that $P_{\bullet,k}$
is the inverse limit of the directed system formed by the bottom two rows.
We want to show that $P_{\bullet,k}$ is also the inverse limit of the 
bottom row $R_{n+1,k} \twoheadrightarrow R_{n,k}$.

For each $n$, the above diagram gives a surjection 
$P_{\bullet,k} \twoheadrightarrow R_{n,k}$.  Suppose we have a commutative ring
$Q$ equipped with morphisms $Q \rightarrow R_{n,k}$ which commute with the maps
$R_{n+1,k} \twoheadrightarrow R_{n,k}$.  We need to show that there exists a unique
map $Q \rightarrow P_{\bullet,k}$ making the following diagram commute:
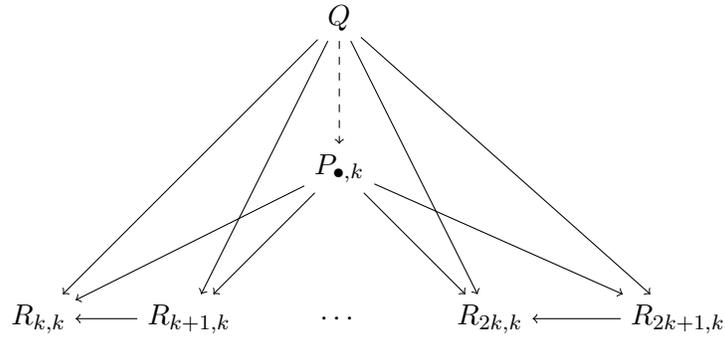
\begin{center}
\begin{tikzpicture}[scale = 1]

\node(A) at (0,2) {$R_{k,k}$};

\node(B) at (2,2) {$R_{k+1,k}$};

\node(C) at (4,2) {$\cdots$};

\node(D) at (6,2) {$R_{2k,k}$};

\node(E) at (8.5,2) {$R_{2k+1,k}$};

\node(K) at (4,4) {$P_{\bullet,k}$};

\node(L) at (4,6) {$Q$};

\draw[ -> ] (B) -- (A);

\draw [->] (E) -- (D);

\draw [->] (L) -- (A);
\draw [->] (L) -- (B);
\draw [->] (L) -- (D);
\draw [->] (L) -- (E);

\draw [->] (K) -- (A);
\draw [->] (K) -- (B);
\draw [->] (K) -- (D);
\draw [->] (K) -- (E);

\draw [dashed, ->] (L) -- (K);

\end{tikzpicture}
\end{center}
Indeed, this diagram commutes if and only if the following larger diagram commutes:
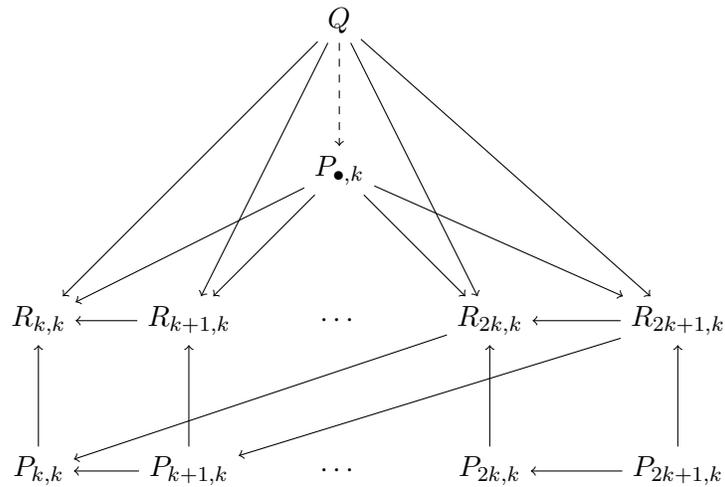
\begin{center}
\begin{tikzpicture}[scale = 1]

\node(A) at (0,2) {$R_{k,k}$};

\node(B) at (2,2) {$R_{k+1,k}$};

\node(C) at (4,2) {$\cdots$};

\node(D) at (6,2) {$R_{2k,k}$};

\node(E) at (8.5,2) {$R_{2k+1,k}$};

\node(F) at (0,0) {$P_{k,k}$};

\node(G) at (2,0) {$P_{k+1,k}$};

\node(H) at (4,0) {$\cdots$};

\node(I) at (6,0) {$P_{2k,k}$};

\node(J) at (8.5,0) {$P_{2k+1,k}$};

\node(K) at (4,4) {$P_{\bullet,k}$};

\node(L) at (4,6) {$Q$};
\draw [->] (D) -- (F);
\draw[ -> ] (B) -- (A);
\draw [->] (G) -- (F);
\draw [->] (E) -- (D);
\draw [->] (J) -- (I);
\draw [->] (F) -- (A);
\draw [->] (G) -- (B);
\draw [->] (I) -- (D);
\draw [->] (J) -- (E);
\draw [->] (E) -- (G);
\draw [->] (L) -- (A);
\draw [->] (L) -- (B);
\draw [->] (L) -- (D);
\draw [->] (L) -- (E);

\draw [->] (K) -- (A);
\draw [->] (K) -- (B);
\draw [->] (K) -- (D);
\draw [->] (K) -- (E);

\draw [dashed, ->] (L) -- (K);

\end{tikzpicture}
\end{center}
Since $P_{\bullet,k}$ is the inverse limit of the system $P_{n+1,k} \twoheadrightarrow P_{n,k}$,
there is a unique morphism $Q \rightarrow P_{\bullet,k}$ such that we have the equality of 
compositions
\begin{equation}
(Q \rightarrow R_{n+k,k} \rightarrow P_{n,k}) =
(Q \rightarrow P_{\bullet,k} \rightarrow R_{n+k,k} \rightarrow P_{n,k})
\end{equation}
for all $n \geq k$.  But this implies the equality of compositions
\begin{equation}
(Q \rightarrow R_{n+k,k} \rightarrow P_{n,k} \rightarrow R_{n,k}) =
(Q \rightarrow P_{\bullet,k} \rightarrow R_{n+k,k} \rightarrow P_{n,k} \rightarrow R_{n,k}),
\end{equation}
so that the morphism $Q \rightarrow P_{\bullet,k}$ makes the lower diagram commute.
\end{proof}

The embedding $\WW_{n,k} \subseteq \WW_{n+1,k+1}$
given by $w \mapsto 1 \times w$ is more intricate at the level of 
Schubert polynomials.
Recall that if $f(\xx_n) = f(x_1, \dots, x_n)$ is a polynomial, we denote the variable reversal by
$f(\xx_n^*) = f(x_n, \dots, x_1)$.

\begin{theorem}
\label{variable-k-embedding}
Let $w \in \WW_{n,k}$.  We have $\symm_{1 \times w}(\xx_{n+1}^*)|_{x_{n+1} = 0} = \symm_w(\xx_n^*)$.
\end{theorem}

We again have combinatorial and geometric proofs of this result.

\begin{proof} {\em (Combinatorial)}
It is easy to see the equality of words $\conv(1 \times w) = 1 \times \conv(w) \in \WW_{n+1,k+1}$ and the equalities of 
permutations
$\std(\conv(1 \times w)) = 1 \times \std(\conv(w)) \in S_{n+1}$, and 
$\sigma(1 \times w) = 1 \times \sigma(w) \in S_{n+1}$.  This means that
\begin{equation*}
\symm_{1 \times w}(\xx_{n+1}) = \sigma(1 \times w)^{-1}.\symm_{\std(\conv(1 \times w))}(\xx_{n+1}) =
 (1 \times \sigma(w))^{-1}.\symm_{1 \times \std(\conv(w))}(\xx_{n+1}).
\end{equation*}
We are therefore reduced to showing $\symm_{1 \times v}(\xx_{n+1}^*)|_{x_{n+1} = 0} = \symm_v(\xx_n^*)$
for any {\em permutation} $v \in S_n$.  

If $v_0 = n (n-1) \dots 1 \in S_n$ is the reversal permutation we have $1 \times v_0 = 1 (n+1) n \dots 2 \in S_{n+1}$.
We have
\begin{align*}
\symm_{1 \times v_0}(\xx_{n+1}) 
&= \partial_1  \cdots \partial_{n-2} \partial_{n-1} \partial_n ( x_1^n  \cdots x_{n-2}^3 x_{n-1}^2 x_n ) \\
&= \partial_1 \cdots \partial_{n-2} \partial_{n-1} (x_1^n  \cdots x_{n-2}^3 x_{n-1}^2 ) \\
&= \partial_1 \cdots \partial_{n-2} (x_1^n  \cdots x_{n-2}^3 x_{n} + x_1^n  \cdots x_{n-2}^3 x_{n-1}) = \cdots \\
&= x_2^{n-1} x_3^{n-2} \cdots x_{n} + \text{(terms involving $x_1$)},
\end{align*}
so that  $\symm_{1 \times v_0}(\xx_{n+1}^*)|_{x_{n+1} = 0} = x_2 x_3^2 \cdots x_{n}^{n-1} =  \symm_{v_0}(\xx_n^*)$.
For a general permutation $v = v_1 \dots v_n \in S_n$ with $v_i > v_{i+1}$, let 
$v' = v_1 \dots v_{i+1} v_{i} \dots v_n \in S_n$.  
We may inductively assume $\symm_{1 \times v}(\xx_{n+1}^*)|_{x_{n+1} = 0} = \symm_v(\xx_n^*)$.
Since $\symm_{1 \times v'}(\xx_{n+1}) = \partial_{i+1} \symm_{1 \times v}(\xx_{n+1})$ and
$\symm_{v'}(\xx_n) = \partial_i \symm_v(\xx_n)$, after reversing variables we deduce
$\symm_{1 \times v'}(\xx_{n+1}^*)|_{x_{n+1} = 0} = \symm_{v'}(\xx_n^*)$.
\end{proof}

\begin{proof} {\em (Geometric)}
Let $\iota: \CC^k \hookrightarrow \CC^{k+1}$ be the embedding sending a column vector
$(z_1, \dots, z_k)^T$ to the column vector $(0, z_1, \dots, z_k)^T$.
The map $\iota$ leads to a closed embedding $j: X_{n,k} \hookrightarrow X_{n+1,k+1}$ sending 
an $n$-tuple of lines $(\ell_1, \dots, \ell_n)$ in $\CC^k$ to the $n$-tuple
$({\bf 1}, \iota(\ell_1), \dots, \iota(\ell_n))$, where ${\bf 1}$ is the line of constant vectors in $\CC^{k+1}$.  
Given any word $w \in \WW_{n,k}$, the cell $C_{1 \times w}$ in $X_{n+1,k+1}$ is transverse to $X_{n,k}$ and 
satisfies $C_{1 \times w} \cap X_{n,k} = C_w$.

If $(\ell_1, \dots, \ell_n)$ are the tautological line bundles over $X_{n,k}$ (from left to right)
and $(\ell'_1, \ell'_2, \dots \ell'_{n+1})$ are the corresponding line bundles over $X_{n+1,k+1}$, then $j$
pulls $(\ell'_1, \ell'_2, \dots, \ell'_{n+1})$ back to $(\CC, \ell_1, \dots, \ell_n)$, where $\CC$ is the constant line bundle
on $X_{n,k}$.
The induced map $j^{*}: H^{\bullet}(X_{n+1,k+1}) \rightarrow H^{\bullet}(X_{n,k})$ on cohomology
therefore acts on variables by $(x_1, x_2, \dots, x_{n+1}) \mapsto (0, x_1, \dots, x_n)$.  
This implies $\symm_{1 \times w}(x_1, x_2, \dots, x_{n+1}) = \symm_w(0, x_1, \dots, x_n)$, which 
is equivalent to 
$\symm_{1 \times w}(\xx_{n+1}^*)|_{x_{n+1} = 0} = \symm_w(\xx_n^*)$.
\end{proof}

Theorem~\ref{variable-k-embedding} has the following cohomological interpretation.  
Fix $k \leq n$.
The tower of closed embeddings $X_{n,k} \hookrightarrow X_{n+1,k+1}$ 
given in the geometric proof of Theorem~\ref{variable-k-embedding}
gives rise to the direct limit
\begin{equation}
X_{n + \bullet,  k + \bullet} := \lim_{\substack{\longrightarrow \\ m \geq 0}} X_{n+m, k+m}  = \bigcup_{m \geq 0} X_{n+m,k+m}.
\end{equation}
We have $X_{n + \bullet, k + \bullet} = X_{(n + 1) + \bullet, (k + 1) + \bullet}$.  
By Lemma~\ref{inverse-limit-cohomology-lemma}, the cohomology of
$X_{n + \bullet, k + \bullet}$ is the inverse limit
\begin{equation}
\label{variable-k-cohomology}
H^{\bullet}(X_{n + \bullet,  k + \bullet}) =
H^{\bullet}(\lim_{\substack{\longrightarrow \\ m \geq 0}} X_{n+m, k+m}) = 
\lim_{\substack{\longleftarrow \\ m \geq 0}} H^{\bullet}(X_{n+m, k+m}) =
\lim_{\substack{\longleftarrow \\ m \geq 0}} R_{n+m, k+m},
\end{equation}
where the last inverse limit is with respect to the system $R_{n+m+1, k+m+1} \twoheadrightarrow R_{n+m,k+m}$ induced
by the variable assignment $(x_1, x_2, \dots, x_{n+m+1}) \mapsto (0, x_1, \dots, x_{n+m})$.

The Fubini word Schubert polynomials $\symm_w(\xx_n)$ for $w \in \WW_{n,k}$ give rise to classes
in $H^{\bullet}(X_{n + \bullet, k + \bullet})$ as follows.
For any $m \geq 1$, define $1^m \times w \in \WW_{n,k}$ inductively by 
$1^m \times w = 1 \times (1^{m-1} \times w)$.  Theorem~\ref{variable-k-embedding} implies that the limit
\begin{equation}
\symm^*_w(\xx) := \lim_{m \rightarrow \infty} \symm_{1^m \times w}(\xx_{n+m}^*)
\end{equation}
is a well-defined formal power series in the variable set $\xx = (x_1, x_2, \dots )$.
The formal power series $\symm^*_w(\xx)$ represents the class $[ \overline{C}_w ]$ in
$H^{\bullet}(X_{n + \bullet, k + \bullet})$.

As an example computation of $\symm^*_w(\xx)$, consider $w = 2123 \in \WW_{4,3}$.  The following table
\begin{center}
\begin{tabular}{c | c | c | c}
$m$ & $1^m \times w \in \WW_{4+m,3+m}$ & $\symm_{1^m \times w}(\xx_{4 + m})$ & $\symm_{1^m \times w}(\xx_{4+m}^*)$ 
\\ \hline  
$0$ & $2123$ & $x_1^2 x_3 + x_1 x_3^2$ & $x_2^2 x_4 + x_2 x_4^2$ \\ \hline 
$1$ & $13234$ & 
$\begin{smallmatrix} \\ x_1^2 x_2 + x_1^2 x_4 + x_1 x_2^2 + 2 x_1 x_2 x_4 \\  
+ x_1 x_4^2 + x_2^2 x_4 + x_2 x_4^2 \end{smallmatrix}$ &
$\begin{smallmatrix} \\  x_2^2 x_4 + x_2 x_4^2 + x_2^2 x_5 + x_2 x_5^2 \\
+ 2 x_2 x_4 x_5 + x_4^2 x_5 + x_4 x_5^2  \end{smallmatrix}$ \\ \hline
$2$ & $124345$ & 
$\begin{smallmatrix} \\ x_1^2 x_2 + x_1 x_2^2 + x_1^2 x_3 +
2 x_1 x_2 x_3 \\
+ x_2^2 x_3 + x_1 x_3^2 + x_2 x_3^2 \\
+ x_1^2 x_5 + 2 x_1 x_2 x_5 + x_2^2 x_5 \\
+ 2 x_1 x_3 x_5 + 2 x_2 x_3 x_5 + x_3^2 x_5 \\
+ x_1 x_5^2 + x_2 x_5^2 + x_3 x_5^2  \end{smallmatrix}$ &
$\begin{smallmatrix} \\ x_2^2 x_4 + x_2 x_4^2 + x_2^2 x_5 + x_2 x_5^2 \\
+ 2 x_2 x_4 x_5 + x_4^2 x_5 + x_4 x_5^2 \\
+ x_2^2 x_6 + x_2 x_6^2 + 2 x_2 x_4 x_6  \\
+ 2 x_2 x_5 x_6 + 2 x_4 x_5 x_6 + x_4^2 x_6 \\
 + x_4 x_6^2 
+ x_5 x_6^2 + x_5^2 x_6
\end{smallmatrix}$
\end{tabular}
\end{center}
illustrates the stability property of Theorem~\ref{variable-k-embedding}:
in the last column, to go from row 3 to row 2 we set $x_6 = 0$ and to go from row
2 to row 1 we set $x_5 = 0$.  This table also suggests the result
\begin{equation*}
\symm^*_{2123}(\xx) = \lim_{m \rightarrow \infty} \symm_{1^m \times 2123}(\xx_{4+m}^*)
= s_{(2,1)}(x_2, x_4, x_5, x_6, \dots ),
\end{equation*}
where $s_{(2,1)}$ is the Schur function corresponding to $(2,1) \vdash 3$.
In particular, the representative
$\symm^*_{2123}(\xx)$ of $[ \overline{C}_w]$ in $H^{\bullet}(X_{4+\bullet,3+\bullet})$
is genuinely a formal power series (as opposed to a polynomial).

The variable reversal involved in Theorem~\ref{variable-k-embedding}
arises in the inverse limit  
$R_{n + \bullet, k + \bullet}$ 
of the system $R_{n+1,k+1} \twoheadrightarrow R_{n,k}$
as follows.
The following commutative diagram

\begin{center}
\begin{tikzpicture}[scale = 1]

\node (B) at (4, 2) {$R_{n + \bullet, k + \bullet}$};

\node (D) at (-2, 0) {$R_{n,k}$};
\node (E) at (4, 0) {$R_{n+1,k+1}$};
\node (F) at (10, 0) {$R_{n+2,k+2} \longleftarrow \cdots$};

\node(G) at (-2,-3) {$(0, x_1, \dots, x_n)$ };
\node(H) at (4,-3) {$(x_1, x_2, \dots, x_{n+1})$ };

\node(I) at (4,-3.5) {$(0, x_1, \dots, x_{n+1)}$ };
\node(J) at (10,-3.5) {$(x_1, x_2, \dots, x_{n+2})$ };

\node(K) at (-2,-2) {$\frac{\ZZ[x_1, \dots, x_n]}
{ \left\langle \begin{smallmatrix}
x_1^k, \dots, x_n^k \\
e_n, \dots, e_{n-k+1}
\end{smallmatrix} \right\rangle}$};
\node(L) at (4,-2) {$\frac{\ZZ[x_1, \dots, x_n, x_{n+1}]}
{ \left\langle \begin{smallmatrix}
x_1^{k+1}, \dots, x_n^{k+1}, x_{n+1}^{k+1}, \\
e_{n+1}, e_{n}, \dots, e_{n-k+1}
\end{smallmatrix} \right\rangle}$};

\node(M) at (10,-2) {$\frac{\ZZ[x_1, \dots, x_n, x_{n+1}, x_{n+2}]}
{ \left\langle \begin{smallmatrix}
x_1^{k+2}, \dots, x_n^{k+2}, x_{n+1}^{k+2}, x_{n+2}^{k+2} \\
e_{n+2}, e_{n+1}, e_n,  \dots, e_{n-k+1}
\end{smallmatrix} \right\rangle}$};
\node(L) at (4,-2) {$\frac{\ZZ[x_1, \dots, x_n, x_{n+1}]}
{ \left\langle \begin{smallmatrix}
x_1^{k+1}, \dots, x_n^{k+1}, x_{n+1}^{k+1}, \\
e_{n+1}, e_{n}, \dots, e_{n-k+1}
\end{smallmatrix} \right\rangle}$};

\draw [|->] (H) -- (G);
\draw [|->] (J) -- (I);

\draw [->] (B) -- (D);
\draw [->] (B) -- (F);
\draw [->]  (E) -- (D);
\draw [->]  (F) -- (E);

\draw [->] (B) -- (E);

\draw [->] (D) -- (K);
\draw [->] (E) -- (L);
\draw [->] (F) -- (M);

\node (N) at (-2.5,-0.7) {$\mathrm{id.}$};
\node (O) at (3.5,-0.7) {$\mathrm{id.}$};
\node (P) at (9.5,-0.7) {$\mathrm{id.}$};

\end{tikzpicture}
\end{center}
illustrates that $R_{n + \bullet, k + \bullet}$ is mapped onto uniquely
by the quotient
\begin{equation} 
S_{n + \bullet, k + \bullet} := \ZZ[[x_1, x_2, \dots ]]/ \langle e_r \,:\, r > n-k \rangle
\end{equation}
of the formal power series ring $\ZZ[[\xx]] = \ZZ[x_1, x_2, \dots ]]$, where 
diagram commutativity 
forces us to define the map $S_{n + \bullet, k + \bullet} \rightarrow R_{n+m, k+m}$ by 
the variable assignment 
\begin{equation*}
(x_1, x_2, \dots, x_{n+m}, x_{n+m+1}, x_{n+m+2}, \dots ) \mapsto
(x_{n+m}, \dots, x_2, x_1, 0, 0, \dots ).
\footnote{We believe that $S_{n + \bullet, k + \bullet}$ is in fact the inverse limit 
$R_{n + \bullet, k + \bullet}$.}
\end{equation*}
The variable reversal involved in the maps out of 
$S_{n + \bullet, k + \bullet}$ corresponds to the variable reversal
in the definition of $\symm^*_w(\xx)$.

We can also consider the limit as $m \rightarrow \infty$ of the polynomials $\symm_{1^m \times w}(\xx_{n+m})$,
where we do not reverse variables.
The {\em Stanley symmetric function} attached to a permutation $w \in S_n$ is 
\begin{equation}
F_w(\xx) := 
\lim_{m \rightarrow \infty} \symm_{1^m \times w}(\xx_{n+m}), \quad w \in S_n.
\end{equation}
It turns out that  $F_w(\xx)$ is a well defined formal power series in the infinite variable set $\xx= (x_1, x_2, \dots )$ 
and is symmetric in the variables $\xx$.

In the case $k < n$, the corresponding limiting procedure (for Fubini words)
gives the Stanley symmetric functions (for permutations).
If $w \in \WW_{n,k}$ and $m \geq 1$,
then
$\std(\conv(1^m \times w)) = 1^m \times \std(\conv(w))$ and
$\sigma(1^m \times w) = 1^m \times \sigma(w)$.  Given any $w \in \WW_{n,k}$, we have
\begin{align}
\lim_{m \rightarrow \infty} \symm_{1^m \times w}(\xx_{n+m}) &= 
\lim_{m \rightarrow \infty} (1^m \times \sigma(w))^{-1}.\symm_{1^m \times \std(\conv(w))}(\xx_{n+m}) \\
&= \lim_{m \rightarrow \infty}  \symm_{1^m \times \std(\conv(w))}(\xx_{n+m}) \\
&= F_{\std(\conv(w))}(\xx),
\end{align}
where the 
second equality used the fact that the limit as $m \rightarrow \infty$ of 
$\symm_{1^m \times \std(\conv(w))}(\xx_{n+m})$ is a well-defined formal power series and
the fact that the permutation $(1^m \times \sigma(w))^{-1} \in S_{n+m}$ acts only on the last $n$
variables $x_{m+1}, x_{m+2}, \dots, x_{m+n}$.

\section{Variations on $X_{n,k}$}
\label{Variations}

There are several generalizations of the ring $R_{n,k}$  which have appeared in the literature.
The first extends $R_{n,k}$ to include a third parameter.

\begin{defn}
\label{three-parameter-definition}
Let $s \leq k \leq n$ be positive integers.  The ideal $I_{n,k,s} \subseteq \ZZ[\xx_n]$ is 
defined by
\begin{equation*}
I_{n,k,s} := \langle x_1^k, x_2^k, \dots, x_n^k, e_n(\xx_n), e_{n-1}(\xx_n), \dots, e_{n-s+1}(\xx_n) \rangle.
\end{equation*}
Let $R_{n,k,s} := \ZZ[\xx_n]/I_{n,k,s}$ be the corresponding quotient.
\end{defn}

The version $R_{n,k,s}^{\QQ} := \QQ \otimes_{\ZZ} R_{n,k,s}$ of the $R_{n,k,s}$ rings were 
defined
by Haglund-Rhoades-Shimozono \cite[Defn. 6.1]{HRS}.  When $s = k$ we have $R_{n,k,k} = R_{n,k}$.

The quotient $R_{n,k,s}$ carries a graded action of $S_n$.
If we set $m = k-s$, the isomorphism type of $R_{n,k,s}^{\QQ}$ is characterized by \cite[Eqn. 6.23]{HRS}
\begin{equation}
q^{m \choose 2} {k \brack m}_q \grFrob(R_{n,k,s}^{\QQ}; q) = 
e_m^{\perp} \, \grFrob(R_{n+m,k}^{\QQ}; q),
\end{equation}
where $e_m^{\perp}: \Lambda \rightarrow \Lambda$ is the operator dual to multiplication by
$e_m$ under the Hall inner product:
\begin{equation}
\langle e_m^{\perp} \, f, g \rangle = \langle f, e_m g \rangle, \quad \text{for all $f,g  \in \Lambda.$}
\end{equation}
The rings  $R_{n,k,s}^{\QQ}$ were used in a crucial way in \cite{HRS} to inductively
understand the $R_{n,k}^{\QQ}$.

For $s \leq k \leq n$, define a class of words $\WW_{n,k,s} \subseteq [k]^n$ by
\begin{equation}
\WW_{n,k,s} := \{ w = w_1 \dots w_n \in [k]^n \,:\, \text{$1, 2, \dots, s$ appear in $w$} \}.
\end{equation}
We have $\WW_{n,k,k} = \WW_{n,k}$ and as ungraded $S_n$-modules 
\cite[Eqn. 6.2]{HRS} there holds the isomorphism
$R_{n,k,s}^{\QQ} \cong \QQ[\WW_{n,k,s}]$.

\begin{lemma}
\label{three-parameter-free}
The ring $R_{n,k,s}$ is a free $\ZZ$-module of rank $|\WW_{n,k,s}|$.
\end{lemma}

\begin{proof}
This is similar to the argument given in the proof of Theorem~\ref{cohomology-of-x} (3);
we use the notation presented there.

Let $\NS_{n,k,s}$ be the family of monomials 
$m$ in the variables $x_1, x_2, \dots, x_n$ such that 
\begin{itemize}
\item  $x_i^k \nmid m$ for all $1 \leq i \leq n$, and
\item  $\xx(S) \nmid m$ for all $S \subseteq [n]$ with $|S| = n-s+1$.
\end{itemize}
Here $\xx(S)$ is 
By \cite[Sec. 6]{HRS}, the set $\NS_{n,k,s}$ is the standard monomial basis for the 
quotient $R_{n,k,s}^{\QQ}$ with respect to the lexicographical term order.
In particular, the set $\NS_{n,k,s}$ is linearly independent (over $\ZZ$)
in $R_{n,k,s}$.  Moreover, by \cite[Lem. 6.6]{HRS} we have $|\NS_{n,k,s}| = |\WW_{n,k,s}|$.

It remains to show that $\NN_{n,k,s}$ spans $R_{n,k,s}$ over $\ZZ$.  The
argument uses \cite[Eqn. 3.5]{HRS} and is very similar to that in the proof of
Theorem~\ref{cohomology-of-x} (3); we leave the details to the reader.
\end{proof}

We are ready to define our spaces which will have cohomology $R_{n,k,s}$.

\begin{defn}
\label{three-paramter-x}
Let $s \leq k \leq n$ be positive integers.
Denote by
 $\MMM_{n,k,s}$  the space of $k \times n$ complex matrices $m$ such that the 
topmost $s \times m$ submatrix of $m$ has full rank.  The space $\MMM_{n,k,s}$ is closed
under the left action of $U \subseteq GL_k(\CC)$ and the right action of $T \subseteq GL_n(\CC)$.  
Define the quotient space $X_{n,k,s}$ by 
\begin{equation*}
X_{n,k,s} := \MMM_{n,k,s}/T
.\end{equation*}
\end{defn}

The quotient $X_{n,k,s}$ can be viewed in terms of line configurations as follows. 
Let $\pi: \CC^k \twoheadrightarrow \CC^s$ be the  projection
$(z_1, \dots, z_k) \mapsto (z_1, \dots, z_s)$ which forgets the last $k-s$ coordinates.  We have
\begin{equation}
X_{n,k,s} = \{ (\ell_1, \dots, \ell_n) \,:\, \ell_i \in \PP^{k-1}, \, \pi(\ell_1 + \cdots + \ell_n) = \CC^s \}.
\end{equation}
If $k = s$ we have $X_{n,k,k} = X_{n,k}$.  In general, the space $X_{n,k,s}$ is an open
subvariety of $(\PP^{k-1})^n$.
The next result shows that $X_{n,k,s}$ has cohomology $R_{n,k,s}$.

\begin{theorem}
\label{three-parameter-x-cohomology}
Let $s \leq k \leq n$ be positive integers and consider the open
subvariety $X_{n,k,s}$ of the $n$-fold product $(\PP^{k-1})^n$.
\begin{enumerate}
\item  For any word $w \in [k]^n$ let $C_w \subseteq (\PP^{k-1})^n$ be the corresponding cell.  We have
\begin{equation*}
X_{n,k,s} = \bigsqcup_{w \in \WW_{n,k,s}} C_w.
\end{equation*}
\item  The space $X_{n,k,s}$ admits a cellular decomposition with cells
$\{ C_w \,:\, w \in \WW_{n,k,s} \}$.
\item   The cohomology presentation given by 
Theorem~\ref{cohomology-presentation-theorem} realizes 
 $H^{\bullet}(X_{n,k,s})$ explicitly as the canonical quotient of
 $H^{\bullet}((\PP^{k-1})^n) = 
 \ZZ[\xx_n]/\langle x_1^{k}, \dots, x_n^{k} \rangle$ given by
\begin{equation}
H^{\bullet}((\PP^{k-1})^n) = 
\ZZ[\xx_n]/\langle x_1^{k}, \dots, x_n^{k} \rangle \twoheadrightarrow
R_{n,k,s} = H^{\bullet}(X_{n,k,s}).
\end{equation}
\item  With respect to the isomorphism given above, for any word $w \in \WW_{n,k,s}$ the class
$[\overline{C}_w] \in H^{\bullet}(X_{n,k,s})$ is represented by the Schubert polynomial
$\symm_w = \sigma(w)^{-1}.\symm_{\std(\conv(w))} \in \ZZ[\xx_n]$.
\end{enumerate}
\end{theorem}

\begin{proof}
(1)  For any word $w \in [k]^n$, the $k \times n$ $0,1$-matrix $m$ representing 
$w$ lies in $\MMM_{n,k,s}$ if and only if $w \in \WW_{n,k,s}$.  Since 
$C_w = U m T/T$ and $X_{n,k,s} = \MMM_{n,k,s}/T$ 
is closed under the action of $U$, the result follows.

(2) is proven in the same way as Theorem~\ref{cohomology-of-x} (2).

(3) By (2), we know that $H^{\bullet}(X_{n,k,s})$ is a free $\ZZ$-module of 
rank $|\WW_{n,k,s}|$.  By Lemma~\ref{three-parameter-free}, the module
$R_{n,k,s}$ is also a free $\ZZ$-module of rank $|\WW_{n,k,s}|$.
As in the proof of Theorem~\ref{cohomology-of-x} (2), it suffices to exhibit
the generators $e_n(\xx_n), e_{n-1}(\xx_n), \dots, e_{n-s+1}(\xx_n)$ of 
$I_{n,k,s}$ as polynomials of the form 
$\sigma(w)^{-1}.\symm_{\std(\conv(w))}$ for some strategically chosen
words $w \in [k]^n - \WW_{n,k,s}$.
If we define $w^i$ as in the proof of Theorem~\ref{cohomology-of-x} (3),
we see that $w^i \notin \WW_{n,k,s}$ for $1 \leq i \leq s$, and that 
$\sigma(w^i)^{-1}.\symm_{\std(\conv(w^i))} = e_{n-i+1}(\xx)$.

(4) is analogous to Theorem~\ref{cohomology-of-x} (4) and is proven in the same way.
\end{proof}

The second variation on $R_{n,k}$ is as follows.

\begin{defn}
Let $n, k,$ and $r$ be positive integers with $k \leq n$.  The ideal 
$J_{n,k,r} \subseteq \ZZ[\xx_n]$ is defined by
\begin{equation*}
J_{n,k,r} := \langle h_{k+1}(\xx_n), h_{k+2}(\xx_n), \dots, h_{k+n}(\xx_n),
e_n(\xx_n), e_{n-1}(\xx_n), \dots, e_{n-r+1}(\xx_n) \rangle.
\end{equation*}
We let $T_{n,k,r} := \ZZ[\xx_n]/J_{n,k,r}$ be the corresponding quotient.
\end{defn}

The ring $T_{n,k,r}^{\QQ} := \QQ \otimes_{\ZZ} T_{n,k,r}$ was defined by
Rhoades-Wilson \cite{RhoadesWilson}
\footnote{where it is denoted $R_{n,k,r}$ and the ideal $J_{n,k,r}$ is written $I_{n,k,r}$}.
In the case $r = 1$, 
\cite[Prop. 5.3]{RhoadesWilson} gives the following connection between
$T_{n,k,1}^{\QQ}$ and unprimed delta operators:
\begin{equation}
\grFrob(T_{n,k,1}^{\QQ}; q) =
\Delta_{h_k e_n} e_n |_{t = 0} =
\Delta_{s_{(k,1^{n-1})}} e_n |_{t = 0}.
\end{equation}

Roughly speaking, we can think of $T_{n,k,1}$  as the `coinvariant algebra'
attached to the (unprimed) Delta operator indexed by the hook Schur function
$s_{(k-1,1^{n-1})}$ in the same way as $R_{n,k}$ is the `coinvariant algebra'
attached to the (primed) Delta operator corresponding to the elementary symmetric
function $e_{k-1}$.

For any positive integers $(n,k,r)$ with $r \leq n$, the graded $S_n$-module
$T_{n,k,r}^{\QQ}$ is a sum of graded shifts of the classical coinvariant 
ring.  In particular, we have \cite[Thm. 4.2]{RhoadesWilson}:
\begin{align}
\grFrob(T_{n,k,r}^{\QQ}; q) &= {n + k - r \brack k}_q \cdot \grFrob(R_n^{\QQ};q) \\
&= {n + k - r \brack k}_q \cdot \sum_{T \in \SYT(n)} q^{\maj(T)} s_{\mathrm{shape}(T)}.
\end{align}
In particular, the Hilbert series of $T_{n,k,r}^{\QQ}$ is 
(see \cite[Thm. 3.6]{RhoadesWilson})
\begin{equation}
\Hilb(T_{n,k,r}^{\QQ}; q) = {n + k - r \brack k}_q \cdot [n]!_q
\end{equation}
and the dimension of $T_{n,k,r}^{\QQ}$ is
\begin{equation}
\label{vector-space-dimension-tail-quotient}
\dim(T_{n,k,r}^{\QQ}; q) = {n + k - r \choose k} \cdot n!.
\end{equation}
We will need the following strengthening of 
Equation~\ref{vector-space-dimension-tail-quotient} to the level of $\ZZ$-modules.

\begin{lemma}
\label{t-z-module-rank}
The quotient $T_{n,k,r}$ is a free $\ZZ$-module of rank ${n + k - r \choose k} \cdot n!$.
\end{lemma}

\begin{proof}
Following \cite[Defn. 3.4]{RhoadesWilson}, let us call a monomial
$m$ in the variables $x_1, \dots, x_n$ 
{\em $(n,k,r)$-good} if 
\begin{itemize}
\item  we have $\xx(S) \nmid m$ for all $S \subseteq [n]$ with $|S| = n-r+1$, and
\item we have $x_i^{k+i} \nmid m$ for all $1 \leq i \leq n$.
\end{itemize}
Endow monomials in the variables $\xx_n$ with the lexicographic term order. 
In \cite[Thm. 3.6]{RhoadesWilson} it is proven that the family 
$\mathcal{G}_{n,k,r}$ of $(n,k,r)$-good
monomials is the standard monomial basis for $T_{n,k,r}^{\QQ}$.
In particular, by \cite[Thm. 3.6]{RhoadesWilson}
(and in fact \cite[Lem. 3.5]{RhoadesWilson}) we have
$|\mathcal{G}_{n,k,r}| = {n + k -r \choose k} \cdot n!$ is the desired rank.

We want to prove that $\mathcal{G}_{n,k,r}$ descends to a $\ZZ$-basis for $T_{n,k,r}$.
The image of $\mathcal{G}_{n,k,r}$ in the quotient is linearly independent over $\ZZ$
because it is linearly independent over $\QQ$.  We are reduced to proving that 
$\mathcal{G}_{n,k,r}$ spans $T_{n,k,r}$ over $\ZZ$.

Let $m$ be an arbitrary monomial in the variable set $\xx_n$.  We prove that 
$m + J_{n,k,r}$ lies in the $\ZZ$-span of $\mathcal{G}_{n,k,r}$.  If 
$m \in \mathcal{G}_{n,k,r}$ this is obvious, so either 
\begin{enumerate}
\item there exists $S \subseteq [n]$ with $|S| = n-r+1$ such that $\xx(S) \mid m$, or
\item there exists $1 \leq i \leq n$ such that $x_i^{k+i} \mid m$.
\end{enumerate}

If $(1)$ holds, set $m' := m/\xx(S)$.  By \cite[Lem. 3.4]{HRS} (and \cite[Eqn. 3.5]{HRS})
the variable reversed Demazure character $\kappa_{\gamma(S)^*}(\xx_n^*)$ lies
in $J_{n,k,r}$.  This implies
$m' \cdot \kappa_{\gamma(S)^*}(\xx_n^*) \in J_{n,k,r}$.  Since
$m' \cdot \xx(S) = m$ is the lexicographical leading term of the product
$m' \cdot \kappa_{\gamma(S)^*}(\xx_n^*)$ and $\kappa_{\gamma(S)^*}(\xx_n^*)$
has integer coefficients, modulo $J_{n,k,r}$ we have
\begin{equation}
m \equiv \text{a $\ZZ$-linear combination of monomials $< m$ in lexicographic order}.
\end{equation}
By induction, the monomial $m$ lies in the span of $\mathcal{G}_{n,k,r}$ 
modulo $J_{n,k,r}$.

If $(2)$ holds, observe that
\begin{equation}
h_d(x_{\ell}, x_{\ell + 1}, \dots, x_n) -
x_{\ell} \cdot h_{d-1}(x_{\ell}, x_{\ell + 1}, \dots, x_n) =
h_d (x_{\ell+1}, \dots, x_n),
\end{equation}
 for all degrees $d > 0$ and all $1 \leq \ell \leq n-1$. Since
 $h_{k+1}(\xx_n), h_{k+2}(\xx_n), \dots, h_{k+i}(\xx_n)$ are among the generators of 
$J_{n,k,r}$, we deduce that
\begin{equation}
h_{k+i}(x_i, x_{i+1}, \dots, x_n) \in J_{n,k,r}.
\end{equation}
Modulo the ideal $J_{n,k,r}$ we therefore have
\begin{equation}
\label{j-equivalence}
x_i^{k+i} \equiv \text{a $\ZZ$-linear combination of monomials $< x_i^{k+1}$ 
in lexicographic oder.}
\end{equation}
If we let $m' := m/x_i^{k+i}$ and multiply both sides of the congruence
\eqref{j-equivalence} by $m'$, we get
\begin{equation}
m \equiv \text{a $\ZZ$-linear combination of monomials $< m$ in lexicographic order}.
\end{equation}
modulo $J_{n,k,r}$ and we are done by induction on the lexicographic term order.
\end{proof}

The number ${n + k - r \choose k} \cdot n!$ is the number of words 
$w = w_1 \dots w_n$ of length $n$ from the alphabet $[n+k]$
which have distinct letters such that the letters $1, 2, \dots, r$ all appear in $w$.
Let $\TTT_{n,k,r}$ be the family of such words.  For example, we have
\begin{equation*}
\TTT_{2,2,1} = \{12, 21, 13, 31, 14, 41 \}.
\end{equation*}
The set $\TTT_{n,k,r}$ is equivalent to the family of `tail positive words'
which served as a combinatorial model for $T_{n,k,r}^{\QQ}$
in \cite{RhoadesWilson}.

Our variety for the quotient $T_{n,k,r}$ is as follows.  

\begin{defn}
Let $n, k,$ and $r$ be positive integers with $r \leq n$.  Denote by $\NNN_{n,k,r}$
the space of $(n+k) \times n$ complex matrices $m$ such that
both $m$ itself and the topmost $r \times n$-submatrix of $m$ have full rank.
The space $\NNN_{n,k,r}$ is closed under the left action of 
$U \subseteq GL_k(\CC)$ and the right action of $T \subseteq GL_n(\CC)$.
Define the quotient space $Y_{n,k,r}$ by
\begin{equation*}
Y_{n,k,r} := \NNN_{n,k,r}/T.
\end{equation*}
\end{defn}

The space $Y_{n,k,r}$ is an open subvariety of $(\PP^{n+k-1})^n$.  In particular,
if $\pi: \CC^{n+k} \twoheadrightarrow \CC^r$  is projection onto the first $r$ coordinates,
we have
\begin{equation}
Y_{n,k,r} = \{ (\ell_1, \dots, \ell_n) \in (\PP^{n+k-1})^n \,:\, 
\dim(\ell_1 + \cdots + \ell_n) = n \text{ and }
\pi(\ell_1 + \cdots + \ell_n) = \CC^r \}.
\end{equation}

\begin{theorem}
\label{three-parameter-y-cohomology}
Let $n, k,$ and $r$ be positive integers with $r \leq n$ and consider the open
subvariety $Y_{n,k,r}$ of the $n$-fold product $(\PP^{n+k-1})^n$.
\begin{enumerate}
\item  For any word $w \in [n+k]^n$ let $C_w \subseteq (\PP^{n+k-1})^n$ be the corresponding cell.  We have
\begin{equation*}
Y_{n,k,r} = \bigsqcup_{w \in \TTT_{n,k,r}} C_w.
\end{equation*}
\item  The space $Y_{n,k,r}$ admits a cellular decomposition with cells
$\{ C_w \,:\, w \in \TTT_{n,k,r} \}$.
\item   The cohomology presentation given by 
Theorem~\ref{cohomology-presentation-theorem} realizes 
 $H^{\bullet}(Y_{n,k,r})$ explicitly as the canonical quotient of
 $H^{\bullet}((\PP^{n+k-1})^n) = 
 \ZZ[\xx_n]/\langle x_1^{n+k}, \dots, x_n^{n+k} \rangle$ given by
\begin{equation}
H^{\bullet}((\PP^{n+k-1})^n) = 
\ZZ[\xx_n]/\langle x_1^{n+k}, \dots, x_n^{n+k} \rangle \twoheadrightarrow
T_{n,k,r} = H^{\bullet}(Y_{n,k,r}).
\end{equation}
In particular, we have 
$\langle x_1^{n+k}, \dots, x_n^{n+k} \rangle \subseteq J_{n,k,r}$.
\item  With respect to the isomorphism given above, for any word $w \in \TTT_{n,k,r}$ 
the class
$[\overline{C}_w] \in H^{\bullet}(Y_{n,k,r})$ is represented by the Schubert polynomial
$\symm_w = \sigma(w)^{-1}.\symm_{\std(\conv(w))} \in \ZZ[\xx_n]$.
\end{enumerate}
\end{theorem}

\begin{proof}

 (1)  Let $w = w_1 \dots w_n \in [n+k]^n$ be an arbitrary word. 
 If $m$ is the $(n+k) \times n$ matrix given by 
 $m_{i,j} = 1$ if $w_j = i$ and $m_{i,j} = 0$ otherwise,
 we have $m \in \NNN_{n,k,r}$ if and only if $w \in \TTT_{n,k,r}$.
 Since $mT \in C_w$, this proves the containment $\subseteq$.

If $w \in \TTT_{n,k,r}$, the definition of the $(n+k) \times k$ pattern matrix
$\PM(w)$ implies that every matrix in $U \widehat{C}_w$ lies in $\NNN_{n,k,r}$,
so that $C_w \subseteq Y_{n,k,r}$.  This proves the containment $\supseteq$.

(2) This is proven in the same way as Theorem~\ref{cohomology-of-x} (2).

(3) Let $L_{n,k,r} \subseteq \ZZ[\xx_n]$ be the ideal in $\ZZ[\xx_n]$ generated
by the variable powers $x_1^{n+k}, x_2^{n+k}, \dots, x_n^{n+k}$ together with
the polynomials
\begin{equation*}
\{ \sigma(w)^{-1}.\symm_{\std(\conv(w))} \,:\, w \in [n+k]^n - \TTT_{n,k,r} \}.
\end{equation*}
By cellular theory we have
$H^{\bullet}(Y_{n,k,r}) \cong \ZZ[\xx_n]/L_{n,k,r}$, so it suffices to
show  $L_{n,k,r} = J_{n,k,r}$.

We start by demonstrating the inclusion $J_{n,k,r} \subseteq L_{n,k,r}$ by 
showing that every generator of $J_{n,k,r}$ lies in $L_{n,k,r}$.
For the generators $e_n(\xx_n), e_{n-1}(\xx_n), \dots, e_{n-r+1}(\xx_n)$
this is a straightforward task: for $1 \leq i \leq r$ if we let $u^i$ denote the word
$1 2 \dots \hat{i} \dots n(n+1)$ (where we have omitted $i$),
we have $u^i \in [n+k]^n - \TTT_{n,k,r}$ and 
$\sigma(u^i) \in S_n$ is the identity permutation so that
\begin{equation}
\sigma(u^i)^{-1}.\symm_{\std(\conv(u^i))} =
\symm_{\std(\conv(u^i))}.
\end{equation}
Since 
\begin{equation}
\std(\conv(u^i)) = 12 \dots \hat{i} \dots n(n+1) i (n+2) (n+3) \dots (n+k)
\end{equation}
is Grassmannian, we deduce
$\symm_{\std(\conv(u^i))} = e_{n-i+1}(\xx_n)$.  We conclude that 
$e_{n-i+1}(\xx_n) \in L_{n,k,r}$.

Now we handle the slightly trickier case of the generators 
$h_{k+1}(\xx_n), h_{k+2}(\xx_n), \dots, h_{k+n}(\xx_n) \in J_{n,k,r}$.
For $1 \leq i \leq n-1$ let $v^i$ be the word 
$v^i = 1 2 \dots (i-1) ii (i+1) \dots (n-1) \in [n+k]^n$ (here we have two copies of $i$).
Then
\begin{equation}
\std(\conv(v^i)) = 1 2 \dots i (n+k+1) (i+1) \dots (n-1) n (n+1) \dots (n+k)
\end{equation}
is Grassmannian and $\sigma(v^i) \in S_n$ is the identity permutation,
so that 
\begin{equation}
\sigma(v^i)^{-1}.\symm_{\std(\conv(v^i))} = 
h_{n+k-i}(x_1, x_2, \dots, x_i, x_{i+1}) \in L_{n,k,r}, \text{for $1 \leq i \leq n-1$}.
\end{equation}
Taking $i = n-1$ we get $h_{k+1}(x_1, \dots, x_n) = h_{k+1}(\xx_n) \in L_{n,k,r}$.
Since 
\begin{equation}
h_{n+k-i+1}(x_1, \dots, x_i, x_{m+1}) - 
x_m \cdot h_{n+k-i}(x_1, \dots, x_i, x_{m+1}) =
h_{n+k-i+1}(x_1, \dots, x_m)
\end{equation}
for all $i$ and $m$, we deduce that 
\begin{equation*}
h_{k+1}(\xx_n), h_{k+2}(\xx_n), \dots, h_{k+n-1}(\xx_n) \in L_{n,k,r}.
\end{equation*}

It remains to show $h_{k+n}(\xx_n) \in L_{n,k,r}$.  For $1 \leq i \leq n-1$
the dual Pieri rule implies
\begin{equation}
\label{ideal-system-of-equations}
s_{(k+i+1,1^{n-i-1})}(\xx_n) = h_{k+i}(\xx_n) \cdot e_{n-i-1}(\xx_n) - 
s_{(k+i, 1^{n-i})}(\xx_n).
\end{equation}
Since we know $h_{k+1}(\xx_n), h_{k+2}(\xx_n), \dots, h_{k+n-1}(\xx_n) \in L_{n,k,r}$,
the system of equations \eqref{ideal-system-of-equations}
shows that $h_{k+n}(\xx_n) = s_{(k+n)}(\xx_n) \in L_{n,k,r}$ if and only if
$s_{(k+1,1^{n-1})}(\xx_n) \in L_{n,k,r}$.  Consider the word
$u = 234 \dots (n-1)nn \in [n+k]^n - \TTT_{n,k,r}$.  The word
\begin{equation}
\std(\conv(u)) = 234 \dots (n-1)n(n+k+1)1 (n+1)(n+2) \dots (n+k)
\end{equation}
is Grassmannian and $\sigma(u) \in S_n$ is the identity permutation.
We deduce that
\begin{equation}
\sigma(u)^{-1}.\symm_{\std(\conv(u))} = s_{(k+1,1^{n-1})}(\xx_n) \in L_{n,k,r},
\end{equation}
so that $h_{k+n}(\xx_n) \in L_{n,k,r}$.  This completes the proof that 
$J_{n,k,r} \subseteq L_{n,k,r}$.

Since $J_{n,k,r} \subseteq L_{n,k,r}$ we have a canonical surjection of 
$\ZZ$-modules 
\begin{equation*}
T_{n,k,r} = \ZZ[\xx_n]/J_{n,k,r} \twoheadrightarrow \ZZ[\xx_n]/L_{n,k,r} 
\cong H^{\bullet}(Y_{n,k,r}).
\end{equation*}
By (2) of this theorem, we know that the image of this surjection is a
free $\ZZ$-module of rank $|\TTT_{n,k,r}|$.  
By Lemma~\ref{t-z-module-rank} we know $T_{n,k,r}$ is also
a free $\ZZ$-module of rank $|\TTT_{n,k,r}|$.
This forces the surjection to be an isomorphism.

(4) is analogous to Theorem~\ref{cohomology-of-x} (4) and is proven in the same way.
\end{proof}

The fact that $h_{k+i}(\xx_n) \in L_{n,k,r}$ for $i \geq 1$ shown in the proof of
Theorem~\ref{three-parameter-y-cohomology} (3) using the dual Pieri rule can 
also be seen using Chern classes as follows.  Let $\ell_1, \dots, \ell_n$ be the tautological
line bundles over $Y_{n,k,r}$ corresponding to the $n$ matrix columns 
and let $\CC^{n+k}$ be the trivial vector bundle over 
$Y_{n,k,r}$ of dimension $n+k$.  
By the definition of $Y_{n,k,r}$, the addition map 
\begin{align}
\ell_1 \oplus \cdots \oplus \ell_n &\rightarrow \CC^{n+k} \\
(v_1, \dots, v_n) &\mapsto v_1 + \cdots + v_n
\end{align}
is injective.  There exists a rank $k$ bundle $U \subseteq \CC^{n+k}$ with
$\CC^{n+k} = U \oplus (\ell_1 \oplus \cdots \oplus \ell_n)$.
Consequently, the power series in $t$ with coefficients in 
$H^{\bullet}(Y_{n,k,r}) = \ZZ[\xx_n]/L_{n,k,r}$ given by
\begin{equation}
c_{\bullet}(U) = 
 \frac{c_{\bullet}(\CC^{n+k})}{c_{\bullet}(\ell_1 \oplus \cdots \oplus \ell_n)} = \frac{1}{\sum_{p = 0}^{n} (-1)^p e_p(\xx_n)t^p} = \sum_{p=0}^{\infty} h_p(\xx_n)t^p
\end{equation}
vanishes above degree $k$.



\section{Closing remarks}
\label{Closing}

\subsection{Bruhat order}
The {\em (strong) Bruhat order} on the symmetric group $S_n$ is the transitive closure of the
relation $u < ut$ whenever $u, t \in S_n$, $t = (i,j)$ is a transposition, and 
$\inv(u) < \inv(ut)$.
Bruhat order may be viewed as the closure-containment order on Schubert cells.
Let $X_w := B w B/B \subset GL_n(\CC)/B = \Fl(n)$ be the Schubert cell attached 
to a permutation $w \in S_n$.  Given $w, v \in S_n$, the following three conditions are equivalent:
\begin{itemize}
\item  $v \leq w$ in Bruhat order,
\item  the closure $\overline{X}_w$ meets the cell $X_v$ (i.e. $\overline{X}_w \cap X_v \neq \varnothing$),
\item  we have $X_v \subseteq \overline{X}_w$.
\end{itemize}
These three conditions may be summarized in that, for any permutation $w \in S_n$, we have
$\overline{X}_w = \bigsqcup_{v \leq w} X_v$.  The Schubert cells
$\{X_w \,:\, w \in S_n \}$ give rise to a CW decomposition of $\Fl(n)$.

It is natural to try to extend Bruhat order from $S_n$ to $\WW_{n,k}$ using the space $X_{n,k}$.
Although $X_{n,k}$ is a smooth manifold, it is not homotopy equivalent to a compact smooth manifold.
This can be seen from the failure of its Poincar\'e polynomial to be palindromic in general. 
The cells $\{ C_w \,:\, w \in \WW_{n,k} \}$ therefore {\em cannot} form a CW decomposition of $X_{n,k}$.
Moreover, it can happen that $\overline{C}_w \cap C_v \neq \varnothing$ and 
yet $C_v \not\subseteq \overline{C}_w$.  

\begin{example}
\label{bad-closures}
Let $n = 4, k = 3$, and define $w, v \in \WW_{4,3}$ by $w = 1323$ and $v = 1123$.
The pattern matrices of $w$ and $v$ are
\begin{equation*}
\PM(w) = \begin{pmatrix}
1 & \star & \star & \star \\
0 & 0 & 1 & 0 \\
0 & 1 & 0 & 1 
\end{pmatrix}, \quad
\PM(v) = \begin{pmatrix}
1 & 1 & \star & \star \\
0 & 0 & 1 & \star \\
0 & 0 & 0 & 1
\end{pmatrix},
\end{equation*} 
so that the closure of the cell $C_w$ contains
\begin{equation*}
\lim_{t \rightarrow \infty}
\begin{pmatrix} 1 & t & 0 & 0 \\ 0 & 0 & 1 & 0 \\ 0 & 1 & 0 & 1 \end{pmatrix} =
\lim_{t \rightarrow \infty}
\begin{pmatrix} 1 & 1 & 0 & 0 \\ 0 & 0 & 1 & 0 \\ 0 & 1/t & 0 & 1 \end{pmatrix} =
\begin{pmatrix} 1 & 1 & 0 & 0 \\ 0 & 0 & 1 & 0 \\ 0 & 0 & 0 & 1 \end{pmatrix}.
\end{equation*}
Therefore we have $\overline{C}_w \cap C_v \neq \varnothing$.
On the other hand, we have $\Delta_{\{1,2\},\{2,4\}} = 0$ on the cell $C_w$ but the corresponding
equation does not hold on all of $C_v$.  Therefore $C_v \not\subseteq \overline{C}_w$.
\end{example}

By Example~\ref{bad-closures}, the relations $\preceq$ and $\preceq'$ on $\WW_{n,k}$ given by 
\begin{align}
v \preceq w &\Leftrightarrow C_v \subseteq \overline{C}_w \\
v \preceq' w &\Leftrightarrow \overline{C}_w \cap C_v \neq \varnothing
\end{align}
are different for general $k \leq n$.  
The relation $\preceq$ is transitive, and so defines a partial order on $\WW_{n,k}$.
The relation $\preceq'$ is not transitive, but can be used to define a partial order.

\begin{proposition}
\label{no-loop-proposition}
The transitive closure of the relation $\preceq'$ do not contain loops.
\end{proposition}

\begin{proof}
We use the fact that 
the space $X_{n,k}$ admits a cellular decomposition
$(X_{n,k} = X_0 \supset X_1 \supset X_2 \supset \cdots \supset X_m = \varnothing)$
with $m = |\WW_{n,k}| = k! \cdot \Stir(n,k)$
where each difference $X_i - X_{i+1}$ is a single cell $C_w$ for some $w \in \WW_{n,k}$.
If $C_w = X_i - X_{i+1}$, the fact that $X_i$ is closed means $\overline{C}_w \subseteq X_i$.
In particular, the closure $\overline{C}_w$ can only meet those cells of the form
$C_v = X_j - X_{j+1}$ for $j > i$.
\end{proof}

By Proposition~\ref{no-loop-proposition}, we could use either $\preceq$ or $\preceq'$ 
to define a partial order on $\WW_{n,k}$.
Example~\ref{bad-closures} shows that $\preceq'$ is poorly behaved with respect to dimension
in that we can have $v \prec' w$ but $\dim(C_v) = \dim(C_w)$.  
We therefore define
the {\em Bruhat order}
$\leq$ on $\WW_{n,k}$ to be the 
relation 
\begin{center}
$v \leq w$ if and only if $C_w \subseteq \overline{C}_v$.
\end{center}
This is dual to the order $\preceq$, but coincides with the classical Bruhat order when $k = n$
and $\WW_{n,n} = S_n$.

\begin{example}
The
Bruhat order on $\WW_{3,2}$ is the following poset:
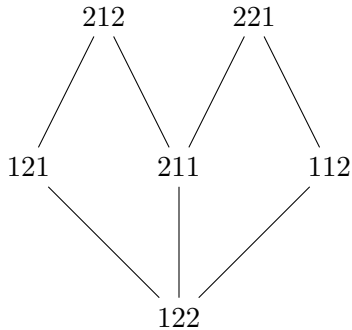
\begin{center}
\begin{tikzpicture}[scale = 1]

\node(A) at (0,0) {$122$};

\node(B) at (-2,2) {$121$};

\node(C) at (0,2) {$211$};

\node(D) at (2,2) {$112$};

\node(E) at (-1,4) {$212$};

\node(F) at (1,4) {$221$};

\draw[ - ] (B) -- (A);
\draw [ -] (C) -- (A);
\draw [-] (D) -- (A);
\draw [-] (B) -- (E);
\draw [-] (C) -- (E);
\draw [-] (C) -- (F);
\draw [-] (D) -- (F);

\end{tikzpicture}
\end{center}
\end{example}

Although the Bruhat order on $\WW_{n,k}$ always has a unique minimal element
(namely the word $1 2 3 \dots k k \dots k$), this example shows that 
$\WW_{n,k}$ typically has more than one maximal element.
The authors believe that $\WW_{n,k}$ is ranked by codimension.

Let $P$ be a finite ranked poset
with rank function $\rho: P \rightarrow \ZZ_{\geq 0}$ 
and let $\mu: P \times P \rightarrow \ZZ$ be the 
M\"obius function of $P$.
The poset $P$ is called {\em Eulerian} if for any $x \leq y$ in $P$ we have
$\mu(x,y) = (-1)^{\rho(y) - \rho(x)}$.  The Bruhat order on $S_n$ is known to be 
Eulerian.  The example above may suggest that $\WW_{n,k}$ is also Eulerian,
but this is  not the case in general.

\begin{example}
The Bruhat order on $\WW_{4,3}$ is the following poset:
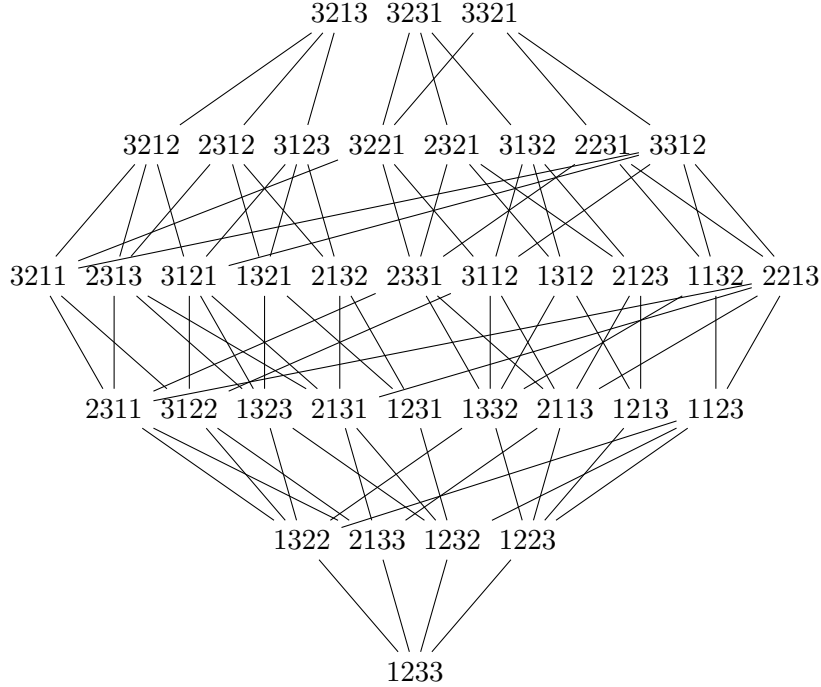
\begin{center}
\begin{tikzpicture}
\node (0and0) at (0.0,-4.38) {$1233$};  \node (1and0) at (-1.5,-2.63) {$1322$};  \node (1and1) at (-0.5,-2.63) {$2133$};  \node (1and2) at (0.5,-2.63) {$1232$};  \node (1and3) at (1.5,-2.63) {$1223$};  \node (2and0) at (-4.0,-0.88) {$2311$};  \node (2and1) at (-3.0,-0.88) {$3122$};  \node (2and2) at (-2.0,-0.88) {$1323$};  \node (2and3) at (-1.0,-0.88) {$2131$};  \node (2and4) at (0.0,-0.88) {$1231$};  \node (2and5) at (1.0,-0.88) {$1332$};  \node (2and6) at (2.0,-0.88) {$2113$};  \node (2and7) at (3.0,-0.88) {$1213$};  \node (2and8) at (4.0,-0.88) {$1123$};  \node (3and0) at (-5.0,0.88) {$3211$};  \node (3and1) at (-4.0,0.88) {$2313$};  \node (3and2) at (-3.0,0.88) {$3121$};  \node (3and3) at (-2.0,0.88) {$1321$};  \node (3and4) at (-1.0,0.88) {$2132$};  \node (3and5) at (0.0,0.88) {$2331$};  \node (3and6) at (1.0,0.88) {$3112$};  \node (3and7) at (2.0,0.88) {$1312$};  \node (3and8) at (3.0,0.88) {$2123$};  \node (3and9) at (4.0,0.88) {$1132$};  \node (3and10) at (5.0,0.88) {$2213$};  \node (4and0) at (-3.5,2.63) {$3212$};  \node (4and1) at (-2.5,2.63) {$2312$};  \node (4and2) at (-1.5,2.63) {$3123$};  \node (4and3) at (-0.5,2.63) {$3221$};  \node (4and4) at (0.5,2.63) {$2321$};  \node (4and5) at (1.5,2.63) {$3132$};  \node (4and6) at (2.5,2.63) {$2231$};  \node (4and7) at (3.5,2.63) {$3312$};  \node (5and0) at (-1.0,4.38) {$3213$};  \node (5and1) at (0.0,4.38) {$3231$};  \node (5and2) at (1.0,4.38) {$3321$};  \draw (0and0) -- (1and0);  \draw (0and0) -- (1and1);  \draw (0and0) -- (1and2);  \draw (0and0) -- (1and3);  \draw (1and0) -- (2and0);  \draw (1and0) -- (2and1);  \draw (1and0) -- (2and2);  \draw (1and0) -- (2and5);  \draw (1and0) -- (2and8);  \draw (1and1) -- (2and0);  \draw (1and1) -- (2and1);  \draw (1and1) -- (2and3);  \draw (1and1) -- (2and6);  \draw (1and2) -- (2and2);  \draw (1and2) -- (2and3);  \draw (1and2) -- (2and4);  \draw (1and2) -- (2and8);  \draw (1and3) -- (2and5);  \draw (1and3) -- (2and6);  \draw (1and3) -- (2and7);  \draw (1and3) -- (2and8);  \draw (2and0) -- (3and0);  \draw (2and0) -- (3and1);  \draw (2and0) -- (3and5);  \draw (2and0) -- (3and10);  \draw (2and1) -- (3and0);  \draw (2and1) -- (3and2);  \draw (2and1) -- (3and6);  \draw (2and2) -- (3and1);  \draw (2and2) -- (3and2);  \draw (2and2) -- (3and3);  \draw (2and3) -- (3and1);  \draw (2and3) -- (3and2);  \draw (2and3) -- (3and4);  \draw (2and3) -- (3and10);  \draw (2and4) -- (3and3);  \draw (2and4) -- (3and4);  \draw (2and5) -- (3and5);  \draw (2and5) -- (3and6);  \draw (2and5) -- (3and7);  \draw (2and5) -- (3and9);  \draw (2and6) -- (3and5);  \draw (2and6) -- (3and6);  \draw (2and6) -- (3and8);  \draw (2and6) -- (3and10);  \draw (2and7) -- (3and7);  \draw (2and7) -- (3and8);  \draw (2and8) -- (3and9);  \draw (2and8) -- (3and10);  \draw (3and0) -- (4and0);  \draw (3and0) -- (4and3);  \draw (3and0) -- (4and7);  \draw (3and1) -- (4and0);  \draw (3and1) -- (4and1);  \draw (3and2) -- (4and0);  \draw (3and2) -- (4and2);  \draw (3and2) -- (4and7);  \draw (3and3) -- (4and1);  \draw (3and3) -- (4and2);  \draw (3and4) -- (4and1);  \draw (3and4) -- (4and2);  \draw (3and5) -- (4and3);  \draw (3and5) -- (4and4);  \draw (3and5) -- (4and6);  \draw (3and6) -- (4and3);  \draw (3and6) -- (4and5);  \draw (3and6) -- (4and7);  \draw (3and7) -- (4and4);  \draw (3and7) -- (4and5);  \draw (3and8) -- (4and4);  \draw (3and8) -- (4and5);  \draw (3and9) -- (4and6);  \draw (3and9) -- (4and7);  \draw (3and10) -- (4and6);  \draw (3and10) -- (4and7);  \draw (4and0) -- (5and0);  \draw (4and1) -- (5and0);  \draw (4and2) -- (5and0);  \draw (4and3) -- (5and1);  \draw (4and3) -- (5and2);  \draw (4and4) -- (5and1);  \draw (4and5) -- (5and1);  \draw (4and6) -- (5and2);  \draw (4and7) -- (5and2);  
\end{tikzpicture}
\end{center}
We leave it to the reader to check that $\mu(3211,3123) = 0$, so that $\WW_{4,3}$ is 
not Eulerian.  Indeed, the closed interval $[3211, 3123] = \{3211, 3212, 3213\}$ in 
this poset corresponds to the standard cellular decomposition of 
$U \times \PP^2 = \CC^3 \times \PP^2$.
\end{example}

The Bruhat order on $S_n$ can be described combinatorially
in two ways.  Given a permutation $w \in S_n$, the {\em rank function}
$r(w): [n] \times [n] \rightarrow \ZZ_{\geq 0}$ is the rank function of the permutation matrix
of $w$.  Explicitly, we have
\begin{equation}
r(w): (i,j) \mapsto | \{ 1 \leq i' \leq i \,:\, w_{i'} \leq j \} |.
\end{equation}
Given two permutations $w, v \in S_n$, the following three conditions are equivalent:
\begin{enumerate}
\item $w \leq v$ in Bruhat order on $S_n$,
\item $r(w) \geq r(v)$ pointwise as functions on $[n] \times [n]$,
\item there exist transpositions $t_1, \dots, t_r \in S_n$ such that 
$v = t_1 \cdots t_r w$ and $\inv(t_{i+1} \cdots t_r w) > \inv(t_i t_{i+1} \cdots t_r w)$
for all $1 \leq i \leq r$.
\end{enumerate}
Moreover, the cover relations in the Bruhat order on $S_n$ have the form
$w \lessdot tw$ where $\inv(tw) = \inv(w) + 1$.

\begin{problem}
Give a combinatorial description of the Bruhat order on $\WW_{n,k}$.  
What are the cover relations of this partial order?
\end{problem}

\subsection{Relating $X_{n,k}$ and $Y_{n,k}$}
The main object of study of this paper was the space
$X_{n,k} = \MMM_{n,k}/T$.  The double coset space
$Y_{n,k} = U \backslash \MMM_{n,k} /T$ was not geometrically well
behaved (in particular, not Hausdorff).  However, working over the finite field
$\FF_q$ we saw that $Y_{n,k}^{\FF_q}$ has size equal to the number of complete
flags in $\FF_q^n$.  
We conjecture that the spaces $X_{n,k}$ and $Y_{n,k}$ are topologically equivalent.

\begin{conjecture}
\label{x-y-equivalent}
Let $k \leq n$ be positive integers.  The spaces $X_{n,k}$ and $Y_{n,k}$
are homotopy equivalent.
\end{conjecture}

The natural surjection $X_{n,k} \twoheadrightarrow Y_{n,k}$ which sends 
a coset $mT$ to the double coset $UmT$ can be shown to be a fiber bundle with
(contractible) fiber $U$.  In general, a fiber bundle $E \twoheadrightarrow B$
with contractible fiber $F$ is a homotopy equivalence
if the base space $B$ is Hausdorff.  Unfortunately, there exist examples 
of fiber bundles $E \twoheadrightarrow B$ with contractible fibers which 
fail to be homotopy equivalences when  $B$ is not Hausdorff.
Despite this, we believe that the natural surjection $X_{n,k} \twoheadrightarrow Y_{n,k}$
is a homotopy equivalence.

\subsection{Torus actions}
The $k$-dimensional torus $T_k \subseteq GL_k(\CC)$ acts on 
$X_{n,k} = \MMM_{n,k}/T$ by the rule $t.(mT) = tmT$ for all $t \in T_k$ 
and $mT \in X_{n,k}$.
Explicitly, the torus $T_k$ acts by scaling the rows of the matrix $m$.

\begin{observation}
The fixed points of the action of $T_k$ on $X_{n,k}$ are in bijection with words
in $\WW_{n,k}$.  In particular, they are of the form $mT$ where 
\begin{equation*}
m_{i,j} = \begin{cases}
1 & w_j = i \\
0 & w_j \neq i,
\end{cases}
\end{equation*}
where $w = w_1 \dots w_n \in \WW_{n,k}$ is some Fubini word.
\end{observation}

Since the action of the 
torus $T_k$ on $X_{n,k}$ is continuous, we can ask for the $T_k$-equivariant 
cohomology of $X_{n,k}$.  
Explicitly, if $E T_k \rightarrow B T_k$ is the universal principal bundle of $T_k$,
the $T_k$-equivariant cohomology $H_{T_k}^{\bullet}(X_{n,k})$ is the 
singular cohomology of the fiber product $E T_k \times_{T_k} X_{n,k}$:
\begin{equation}
H^{\bullet}_{T_k}(X_{n,k}) := H^{\bullet}(E T_k \times_{T_k} X_{n,k}).
\end{equation}

\begin{problem}
Compute the $T_k$-equivariant cohomology $H^{\bullet}_{T_k}(X_{n,k})$.
\end{problem}

When $k = n$ we have the identification 
\begin{equation}
H^{\bullet}_{T_n}(\mathcal{F \ell}(n)) \cong
\ZZ[x_1, \dots, x_n, t_1, \dots, t_n]/
\langle e_i(x_1, \dots, x_n) - e_i(t_1, \dots, t_n) \,:\, 1 \leq i \leq n \rangle.
\end{equation}

\subsection{Springer fibers}
It is shown in \cite{HRS} that the quotient $R_{n,k}^{\QQ}$ has Frobenius image 
in the dual Hall-Littlewood basis given by
\begin{equation}
\label{dual-hall-littlewood-expansion}
\grFrob(R_{n,k}^{\QQ};q) = 
\rev_q \left[  
\sum_{\substack{\lambda \vdash n \\ \ell(\lambda) = k}}  
q^{\sum_i (i-1)(\lambda_i  - 1)} \cdot
{k \brack m_1(\lambda), m_2(\lambda), \dots, m_n(\lambda)}_q \cdot
Q'_{\lambda} 
\right].
\end{equation}
Here $\ell(\lambda)$ is the number of parts of a partition $\lambda$, $m_i(\lambda)$
is the multiplicity of a part $i$ in $\lambda$, and 
\begin{equation}
{k \brack m_1(\lambda), m_2(\lambda), \dots, m_n(\lambda)}_q =
\frac{[k]!_q}{[m_1(\lambda]!_q [m_2(\lambda)]!_q \cdots [m_n(\lambda)]!_q}
\end{equation}
is the $q$-multinomial coefficient.

Let $\lambda \vdash n$.
The symmetric function $Q'(\xx;q)$ has the following geometric interpretation.
Let $X$ be the $n \times n$ nilpotent matrix in Jordan form with Jordan 
blocks of sizes
$\lambda_1, \lambda_2, \dots$.  The {\em Springer fiber} $\mathcal{S}_{\lambda}$ is 
the subset of the flag manifold $\Fl(n)$ given by
\begin{equation}
\mathcal{B}_{\lambda} = \{ (V_1 \subset V_2 \subset \cdots V_n) \in \Fl(n) \,:\, X.V_i = V_i \}.
\end{equation}
Let $H^{\bullet}(\mathcal{B}_{\lambda}; \CC)$ be the cohomology ring of 
$\mathcal{B}_{\lambda}$ with complex coefficients. This cohomology ring carries
a subtle action of the symmetric group $S_n$ called the {\em Springer action}.
The isomorphism type of this module is
\begin{equation}
\grFrob(H^{\bullet}(\mathcal{B}_{\lambda}; \CC); q) = \rev_q [Q'_{\lambda}(\xx;q^2)].
\end{equation}

On the other hand, we know that $X_{n,k}$ has cohomology ring
equal to $R_{n,k}$.  This suggests the following problem.

\begin{problem}
Give a geometric proof of Equation~\eqref{dual-hall-littlewood-expansion} by 
exhibiting an appropriate stratification of $X_{n,k}$ into 
spaces homotopy equivalent to Springer fibers $\mathcal{B}_{\lambda}$.
\end{problem}

Let $\nu \vdash m$ be a partition and consider the associated delta operator
$\Delta'_{s_{\nu}}$ indexed by the Schur function $s_{\nu}$.
Haglund, Rhoades, and Shimozono proved \cite{HRSSchur} the following 
$Q'$-basis expansion of $\omega \Delta'_{s_{\nu}}$ at $t = 0$.
\begin{equation}
\label{dual-hl-expansion-schur}
\omega \Delta'_{s_{\nu}} e_n|_{t = 0} = \sum_{k = \ell(\nu) + 1}^{m+1}
P_{\nu,k-1}(q) \sum_{\substack{\lambda \vdash n \\ \ell(\lambda) = k}} 
q^{\sum_i (i-1)(\lambda_i - 1)} \cdot
{k \brack m_1(\lambda), m_2(\lambda), \dots, m_n(\lambda)}_q  \cdot
Q'_{\lambda},
\end{equation}
where
\begin{equation}
P_{\nu,k-1} = q^{m - {k \choose 2}} \sum_{\substack{\rho \vdash m \\ \ell(\rho) = k-1}}
q^{\sum_i (i-1) \rho_i} \cdot 
{k - 1 \brack m_1(\rho), m_2(\rho), \dots, m_n(\rho)}_q \cdot K_{\nu,\rho}(q)
\end{equation}
and $K_{\nu,\rho}(q)$ is the Kostka-Foulkes polynomial.
The case where $\nu = (1^{k-1})$ is a single column corresponds to the 
Delta Conjecture.

In this paper we proven that (up to applying $\QQ \otimes_{\ZZ} (-)$)
the cohomology ring of $X_{n,k}$ has graded Frobenius image
$(\rev_q \circ \omega) \Delta'_{e_{k-1}} e_n|_{t = 0, q = q^2}$.  The similarity
between the dual Hall-Littlewood
expansions of Equations~\eqref{dual-hall-littlewood-expansion} and 
\eqref{dual-hl-expansion-schur}  suggests the following problem.

\begin{problem}
\label{delta-nu-problem}
Let $\nu \vdash m$ be a partition.  Find a variety $X_{n,\nu}$ whose cohomology carries
an action of $S_n$ such that 
$\QQ \otimes_{\ZZ} H^{\bullet}(X_{n,\nu})$ has graded Frobenius image
$(\rev_q \circ \omega) \Delta'_{s_{\nu}} e_n|_{t = 0, q = q^2}$.
\end{problem}

When $\nu = (1^{k-1})$ the variety $X_{n,k}$ is a solution
to Problem~\ref{delta-nu-problem}.

\section{Acknowledgments}
\label{Acknowledgments}

The authors are grateful to Sara Billey, Linda Chen,
Thomas Lam, and Vic Reiner for  helpful conversations.
B. Rhoades was partially supported by NSF Grant DMS-1500838.

\end{document}